\documentclass[12pt]{amsart}

\usepackage[dvips]{color}
\usepackage{amsmath}
\usepackage{amsxtra}
\usepackage{amscd}
\usepackage{amsthm}
\usepackage{amsfonts}
\usepackage{amssymb}
\usepackage[mathscr]{eucal}
\usepackage{graphics}
\usepackage{physics}
\usepackage{hyperref}
\usepackage{mathtools}
\usepackage{enumitem}
\usepackage{tikz}
\usepackage{stmaryrd}
    \SetSymbolFont{stmry}{bold}{U}{stmry}{m}{n} 
\usepackage{verbatim}

\usepackage{orcidlink}

\usepackage{cite}

\usepackage{float}
\allowdisplaybreaks

\numberwithin{equation}{section}
\newtheorem{thm}{Theorem}[section]
\newtheorem*{thm-nonum}{Theorem}
\newtheorem{prop}[thm]{Proposition}
\newtheorem{lem}[thm]{Lemma}

\newtheorem{conj}[thm]{Conjecture}
\newtheorem{dfn}[thm]{Definition}

\theoremstyle{definition}
\newtheorem{exa}[thm]{Example}

\theoremstyle{remark}
\newtheorem{rem}[thm]{Remark}

\newcommand{\C}{{\mathbb C}}
\newcommand{\Cx}{{\mathbb C}^\times}
\newcommand{\Z}{{\mathbb Z}}

\newcommand{\cB}{\mathcal{B}}

\newcommand{\cI}{\mathcal{I}}

\newcommand{\cS}{\mathcal{S}}

\newcommand{\hI}{\hat{I}}

\newcommand{\g}{{\mathfrak{g}}}
\newcommand{\gs}{{\mathfrak{g}_\s}}
\newcommand{\ghs}{{\widehat{\mathfrak{g}}_\s}}

\newcommand{\gh}{\widehat{\mathfrak{g}}}

\renewcommand{\sl}{\mathfrak{sl}}
\newcommand{\osp}{\mathfrak{osp}}

\newcommand{\h}{\mathfrak{h}}

\newcommand{\Uqgs}{U_q({\mathfrak{g}_\s})}
\newcommand{\Uqghs}{U_q({\widehat{\mathfrak{g}}_\s})}
\newcommand{\Uqgt}{U_q({\mathfrak{g}_\t})}
\newcommand{\Uqght}{U_q({\widehat{\mathfrak{g}}_\t})}

\newcommand{\Uq}{U_q}

\newcommand{\bs}[1]{{\boldsymbol #1 }}

\renewcommand{\Im}{\mathop{\rm Im}}
\newcommand{\Ker}{\mathop{\rm Ker}}

\newcommand{\Sym}{\mathrm{Sym}}

\newcommand{\s}{\mathbf{s}}
\renewcommand{\t}{\mathbf{t}}

\newcommand{\D}{\Delta}

\newcommand{\lb}[1]{\llbracket #1 \rrbracket}

\usepackage[final]{microtype}

\author[L. Bezerra, V. Futorny and I.Kashuba]{Luan Bezerra,  Vyacheslav Futorny,  and Iryna Kashuba}

\address[bezerra.luan@gmail.com ]{\textsc{Luan Pereira Bezerra \orcidlink{0000-0002-4979-4935}}: Instituto de Ciências Exatas, Universidade Federal de Minas Gerais, Belo Horizonte, Brazil}

 \address[futorny@sustech.edu.cn]{\textsc{Vyacheslav Futorny \orcidlink{0000-0002-4701-8879}}: Shenzhen International Center for Mathematics, Southern University of Science and Technology, China}
\address[kashuba@sustech.edu.cn]{\textsc{Iryna Kashuba  \orcidlink{0000-0001-9672-668X}}: Shenzhen International Center for Mathematics, Southern University of Science and Technology, China}

\begin{document}
\begin{title}[Drinfeld realization for Orthosymplectic superalgebras]{Drinfeld Realization for Quantum Affine Orthosymplectic Superalgebras}
\end{title}

	\begin{abstract}
A well-defined braid groupoid action is an essential tool for constructing the new Drinfeld realization of a quantum affine superalgebra. For quantum affine orthosymplectic superalgebras (types B, C, and D), this action was not fully defined, as the braid operators $T_i$ were known only up to normalization factors. In this paper, we solve this problem by providing the explicit formulas for these operators for any choice of parity. This yields a well-defined braid group action on the direct sum of these superalgebras. As a consequence, we use this action to formally introduce the new Drinfeld realization $U_q^D(\widehat{\mathfrak{g}}_s)$ for these types and prove that the corresponding Drinfeld-Jimbo quantum group $U_q(\widehat{\mathfrak{g}}_s)$ is its surjective homomorphic image. We conjecture that this map is an isomorphism.

\end{abstract}

\maketitle

\section{Introduction}

Quantum groups first arose in the work of Faddeev and his school from a quantum version of the inverse scattering method \cite{Faddeev,KS}. They were initially presented as associative algebras whose relations are expressed in terms of quantum R-matrices (the R-matrix realization) \cite{RS1}. This was soon followed by a more formal definition by Drinfeld and Jimbo \cite{Drinfeld1985,Jimbo}, who introduced the ``Drinfeld-Jimbo" realization of quantum groups as $q$-deformations $U_q(\widehat{\mathfrak{g}})$ of universal enveloping algebras, defined by finite Chevalley-type generators and $q$-deformed Serre relations.

In 1988, Drinfeld introduced a ``new" realization $U_q^D(\widehat{\mathfrak{g}})$ \cite{Drinfeld1987}, which presents the algebra using infinitely many loop-like generators. This second realization proved indispensable for the study of finite-dimensional representations \cite{ChariPressley94, ChariPressley98} and vertex representations \cite{Jing90}.

Drinfeld stated the equivalence of these two realizations \cite{Drinfeld1987}. A surjective homomorphism from the new realization to the Drinfeld-Jimbo realization was constructed by Beck \cite{beck1994} for the untwisted case, and later extended by Damiani \cite{damiani2012} for the twisted case. The full isomorphism was eventually established by Damiani in \cite{damiani2015}. A key ingredient in Beck's construction is the action of the affine braid group on $U_q(\widehat{\mathfrak{g}})$, as defined by Lusztig  \cite{Lusztig1989}.

When one moves from (even) Lie algebras to Lie superalgebras, these foundational questions become significantly more complex. The theory of quantum affine superalgebras, initiated in \cite{ZhangY,CWWZ}, must deal with several new challenges:

\begin{enumerate}
    \item A single Lie superalgebra $\mathfrak{g}_s$ has many non-isomorphic Cartan data, parameterized by a parity sequence $s \in \mathcal{S}$.
    \item The Weyl group is replaced by a more intricate object, the Weyl groupoid \cite{Yamane1999}.
    \item The braid group $B_N$ no longer acts on a single algebra $U_q(\widehat{\mathfrak{g}}_s)$, but rather on the direct sum of all algebras over all possible parities, $U_q(\widehat{\mathfrak{g}}_\bullet) := \bigoplus_{s\in\mathcal{S}} U_q(\widehat{\mathfrak{g}}_s)$.
    \item The presence of odd roots introduces new, complex Serre relations.
\end{enumerate}

The foundational work in this ``super" setting was done by Yamane \cite{Yamane1999}, who provided the Drinfeld-Jimbo realization for quantum affine superalgebras of types A-G. For type A, Yamane used a braid groupoid action to construct the new Drinfeld realization and proved the existence of a surjective homomorphism $U_q^D(\widehat{\mathfrak{g}}) \to U_q(\widehat{\mathfrak{g}})$. This was only recently proven to be an isomorphism by Lin, Yamane, and Zhang \cite{LinYamaneZhang} (under certain conditions).

Progress on other types has been steady, though often focused on specific cases or different approaches. For instance, Zhang \cite{Zhang2013} constructed a generating set of PBW type for type A with standard parity, and Tsymbaliuk \cite{Tsymbaliuk_2017} later constructed a PBW basis for any parity. Heckenberger et al. \cite{Heckenberger2007} established the Drinfeld realization and surjectivity for $D^{(1)}(2,1;x)$. For type B, Xu and Zhang \cite{XuZhang} provided the Drinfeld realization when the Dynkin diagram does not have isotropic odd roots using smash products, and Wu, Lin, and Zhang \cite{WLZ} recently demonstrated the equivalence with the R-matrix realization for type B in standard parity.

However, for the general orthosymplectic types (B, C, and D) and for any parity, a crucial piece of the Beck/Damiani approach has been missing. While Yamane \cite[Proposition 7.4.1]{Yamane1999} introduced the fundamental isomorphism operators $T_{i,\s}$ that generate the braid groupoid, they were defined only up to a normalization factor. The explicit coefficients required for these operators to satisfy the braid relations (cf. \cite{Lusztig1989}) were not computed. Without these explicit coefficients, a well-defined braid group action — the primary tool used by Beck and Damiani — is not available. This gap has prevented the construction of the Drinfeld realization and the proof of Beck's homomorphism for types B, C, and D in full generality.

The importance of these explicit coefficients is underscored by the recent work of Wu, Lin, and Zhang \cite{WLZ2025}. Utilizing the formulas for the braid action established in the present work (Theorem 4.2), they successfully proved the full isomorphism (Conjecture 5.9) for the specific case of type B with standard parity.

In this paper, we solve this problem in its full generality. Our first main result is the explicit computation of these normalization factors for the braid groupoid operators $T_{i,s}$ for all quantum affine orthosymplectic superalgebras (types B, C, and D) and for any choice of parity $s \in \mathcal{S}$.

\begin{thm-nonum}[\ref{thm-braid action}]
    The operators $T_{i,s}$ with the explicit formulas given in Section \ref{sect-formulas} satisfy the braid relations \eqref{braid relations ends}-\eqref{braid relations middle}. This defines an action of the affine braid group $B_N$ on the direct sum of superalgebras $U_q(\widehat{\mathfrak{g}}_\bullet)$.
\end{thm-nonum}

This theorem provides the key technical tool that was missing. Using this well-defined braid group action, we follow Beck's strategy to construct the loop-like generators for the Drinfeld ``new" realization. The root vectors are obtained by applying the automorphisms $T_{\omega_i}$ (built from the $T_{i,s}$ operators) to the simple Chevalley generators.

This allows us to introduce the new Drinfeld realization $U_q^D(\widehat{\mathfrak{g}}_s)$ for orthosymplectic types (Definition \ref{def drinfeld real}) and prove our second main result:

\begin{thm-nonum}[\ref{thm_main}]
For any parity $s \in \mathcal{S}$, there exists a surjective homomorphism of superalgebras $\psi_s: U_q^D(\widehat{\mathfrak{g}}_s) \to U_q(\widehat{\mathfrak{g}}_s)$.
\end{thm-nonum}

This theorem establishes that the Drinfeld-Jimbo algebra $U_q(\widehat{\mathfrak{g}}_s)$ is a homomorphic image of the new Drinfeld algebra $U_q^D(\widehat{\mathfrak{g}}_s)$. This is the first essential step toward proving a full isomorphism, which we conjecture to hold. The construction of a PBW basis and the proof of injectivity for $\psi_s$ will be the subject of a subsequent paper.

The structure of the present paper is as follows. In Section~\ref{sect-affine-Lie-superalgebras}, we recall the basic notions on orthosymplectic Lie superalgebras and their root systems, parities, and the affine Weyl groupoid. In Section~\ref{sect-quantum-JD}, we define the Drinfeld-Jimbo realization $U_q(\widehat{\mathfrak{g}}_s)$. In Section~\ref{sect-formulas}, we define the braid group operators $T_{i,s}$, provide their explicit formulas, and prove that they satisfy the braid relations. Finally, in Section~\ref{sect-quantum-newD}, we introduce the new Drinfeld realization $U_q^D(\widehat{\mathfrak{g}}_s)$ and use the braid group action to construct the surjective homomorphism $\psi_s$, proving our main result.

\section{Affine Lie superalgebras}\label{sect-affine-Lie-superalgebras}
\subsection{Orthosymplectic Lie superalgebras}

We work over the field of complex numbers $\mathbb{C}$.

A superalgebra is a $\Z_2$-graded algebra $A = A_0 \oplus A_1$. Elements of $A_0$ are called even, and elements of $A_1$ are called odd. We denote the parity of an element $v \in A_i$ by $|v| = i$, where $i \in \Z_2$.

Fix $m, n \in \Z_{\geq 0}$, $n\neq 0$, and let $N=m+n$. 

Let $\g = \g_0 \oplus \g_1$ be the Lie superalgebra $\osp(2m+1,2n)$ of type $B$, the Lie superalgebra $\osp(2m,2n)$ of type $D$ ($m\geq 2$), or the Lie superalgebra $\osp(2,2n)$  of type $C$. Let $\widehat{\g}$ be the corresponding affine Lie superalgebra, as described in \cite{Yamane1999}. The set of Dynkin nodes is denoted by $I = {1,2,\dots,N}$ and $\hat{I} = {0,1,\dots,N}$, respectively. The elements of $\hat{I}$ are always considered modulo $N+1$. 

Let $\h\subset \g_0$ be a fixed Cartan subalgebra of $\g$ and consider the root space decomposition of $\g$
\begin{align*}
	\g = \h\oplus \bigoplus_{\alpha\in \D}\g_\alpha, 
\end{align*}
where $\D=\{\alpha\in \h^*\setminus\{0\}\mid \g_\alpha\neq 0\}$ is the root system of $\g$, and $\g_\alpha = \{x\in \g\mid [h,x]=\alpha(h)x\ \forall h\in \h\}$ is the root space associated to the root $\alpha$. Recall that every root is either purely even or purely odd, meaning that, for every $\alpha\in \D$, we have $\g_\alpha\subset \g_0$ or $\g_\alpha\subset \g_1$, respectively. In particular, $\D=\D_0\cup\D_1$, where $\D_i=\{\alpha\in \D\mid \g_\alpha\subset \g_i\}$ for $i\in \Z_2$. 

Lie superalgebras $\g$ and $\gh$ have many non-isomorphic \emph{Cartan data} (Dynkin diagrams, isomorphism classes of sets of simple roots and Cartan matrices). The equivalence classes of Cartan data for both $\g$ and $\gh$ are
parameterized by the set $\cS$ of $N$-tuples of $\pm 1$ with exactly $m$ positive coordinates
\begin{align*}
	\cS:=\{\s=(s_1,\dots,s_N)|s_i\in \{-1,+1\},\#\{i|s_i=1\}=m\}. 
\end{align*}
An element of $\cS$ is called a \emph{parity sequence}.

\begin{rem}\label{standard}
    The parity sequence $\s=(1,\dots,1,-1,\dots,-1)$ is called the \emph{standard} parity sequence.
\end{rem}

If $\s$ is a parity sequence of $\g$, we denote by $\gs$ the Lie superalgebra $\g$ with fixed Cartan data given by $\s$. Moreover, whenever the subscript $\s$ appears, the object is considered to be associated with $\gs$ with a fixed choice of parity. However, we will often omit the subscript $\s$ if the chosen parity is clear from the context.

Let $\s\in \cS$ be a parity sequence and consider the integral lattice $P_\s$ with basis $\{\delta\}\cup\{ \varepsilon_i\}_{i=1}^{N}$,  and bilinear form given by
\begin{align*}
	( \delta| \varepsilon_i )_\s=0, \quad( \varepsilon_i| \varepsilon_j )_\s= s_i\delta_{i,j}, \quad i,j=1,\dots, N. 
\end{align*}

The various data will be described in detail using the notation of parity sequences $\s \in \cS$ in the following subsection. We adopt the following conventions for the Dynkin diagrams. Recall that a simple root $\alpha$ of $\gs$ may be even, odd and isotropic ($(\alpha|\alpha)=0$), or odd and non-isotropic. We say its corresponding node $i$ has parity $|i|:=|\alpha_i|$ and it will be depicted in the Dynkin diagram as follows:
\begin{itemize}
	\item If $\alpha$ is even, then we draw an empty circle \begin{tikzpicture}[baseline=-3pt]
	      \draw (0,0pt) circle (0.15cm);
	\end{tikzpicture}.
	\item If $\alpha$ is odd and isotropic, then we draw a crossed circle \begin{tikzpicture}[baseline=-3pt]
	      \draw (0,0pt) circle (0.15cm);
	      \draw[rotate=45] (-0.15,0)--(0.15,0);
	      \draw[rotate=-45] (-0.15,0)--(0.15,0); 
	\end{tikzpicture}.
	\item If $\alpha$ is odd and non-isotropic, then we draw a filled circle \begin{tikzpicture}[baseline=-3pt]
	      \draw[fill=black] (0,0pt) circle (0.15cm);
	\end{tikzpicture}.
\end{itemize}

In many cases, it is useful to draw the diagram allowing multiple options for the parities of the nodes. 

\begin{itemize}
	\item If a node is either even (\begin{tikzpicture}[baseline=-3pt]
	      \draw (0,0pt) circle (0.15cm);
	\end{tikzpicture}) or odd and isotropic (\begin{tikzpicture}[baseline=-3pt]
	\draw (0,0pt) circle (0.15cm);
	\draw[rotate=45] (-0.15,0)--(0.15,0);
	\draw[rotate=-45] (-0.15,0)--(0.15,0); 
	\end{tikzpicture}), then we draw a cross \begin{tikzpicture}[baseline=-3pt]
	\draw[rotate=45] (-0.15,0)--(0.15,0);
	\draw[rotate=-45] (-0.15,0)--(0.15,0); 
	\end{tikzpicture}.
	\item If a node is either even (\begin{tikzpicture}[baseline=-3pt]
	      \draw (0,0pt) circle (0.15cm);
	\end{tikzpicture}) or odd and non-isotropic (\begin{tikzpicture}[baseline=-3pt]
	\draw[fill=black] (0,0pt) circle (0.15cm);
	\end{tikzpicture}), then we draw two concentric circles with the smaller one filled \begin{tikzpicture}[baseline=-3pt]
	\draw (0,0pt) circle (0.15cm);
	\draw[fill=black] (0,0pt) circle (0.05cm);
	\end{tikzpicture}.
	\if{    \item If a node is either even (\begin{tikzpicture}[baseline=-3pt]
		\draw (0,0pt) circle (0.15cm);
		\end{tikzpicture}), odd and isotropic (\begin{tikzpicture}[baseline=-3pt]
		\draw (0,0pt) circle (0.15cm);
		\draw[rotate=45] (-0.15,0)--(0.15,0);
		\draw[rotate=-45] (-0.15,0)--(0.15,0); 
		\end{tikzpicture}), or odd and non-isotropic (\begin{tikzpicture}[baseline=-3pt]
		\draw[fill=black] (0,0pt) circle (0.15cm);
		\end{tikzpicture}), then we combine the cross and the small filled circle  \begin{tikzpicture}[baseline=-3pt]
		\draw[rotate=45] (-0.15,0)--(0.15,0);
		\draw[rotate=-45] (-0.15,0)--(0.15,0); 
		\draw[fill=black] (0,0pt) circle (0.05cm);
		\end{tikzpicture}.}\fi
\end{itemize}

The nodes $i$ and $j$ of a Dynkin diagram are linked by $\dfrac{4(\alpha_i,\alpha_j)_{\s}^2}{(\alpha_i,\alpha_i)_{\s}(\alpha_j,\alpha_j)_{\s}}$ edges if $(\alpha_i,\alpha_i)_{\s}(\alpha_j,\alpha_j)_{\s}\neq 0$,  and $|(\alpha_i,\alpha_j)_{\s}|$ if $(\alpha_i,\alpha_i)_{\s}(\alpha_j,\alpha_j)_{\s}=0$. If nodes $i$ and $j$ are linked and $|(\alpha_i,\alpha_i)_{\s}|>|(\alpha_j,\alpha_j)_{\bs{s}}|$, we add an arrow pointing towards the node $j$. Similar to the case of the nodes, we sometimes allow multiple options for the edges in a diagram. A dotted line \begin{tikzpicture}[baseline=-3pt]
\draw[dotted] (0,0) -- (.5,0);
\end{tikzpicture}  means that the edge may or may not be present, depending on the parities of the nodes. 

\subsection{Root systems}
\subsubsection{Type B root systems}
Fix $\s \in \cS$.
Although the root system itself does not depend on $\s$, the choice of the parity $\s$ fixes a basis of $P_\s$, and this changes the expression of the roots as sums of simple roots.

The root system $\Delta_\s = \Delta_{0,\s} \cup \Delta_{1,\s}$ of $\gs$ is given by:
\begin{align*}
	  & \Delta_\s = \{\pm \varepsilon_i \pm \varepsilon_j, \pm \varepsilon_i, \pm 2\varepsilon_k \mid 1 \leq i,j,k \leq N, s_k = -1\},\\                         
	  & \Delta_{0,\s} = \{\pm \varepsilon_i \pm \varepsilon_j, \pm \varepsilon_k, \pm 2\varepsilon_l \mid 1 \leq i,j,k,l \leq N, s_i = s_j, s_k = 1, s_l = -1\},\\ 
	  & \Delta_{1,\s} = \{\pm \varepsilon_i \pm \varepsilon_j, \pm \varepsilon_k \mid 1 \leq i,j,k \leq N, s_i = -s_j, s_k = -1\}.
\end{align*} 

The simple roots are given by:
\begin{align*}
	\alpha_i=\begin{cases}
	\varepsilon_{i}-\varepsilon_{i+1}, & \text{ if } i=1,\dots,N-1; \\
	\varepsilon_{N},                   & \text{ if } i=N.           
	\end{cases}
\end{align*}

The parities of the simple roots $\alpha_i$, $i \in I$, are given by:
\begin{align*}
	|\alpha_i|=\begin{cases}
	(1-s_i s_{i+1})/2, & \text{ if } i=1,\dots,N-1; \\
	(1-s_N)/2,         & \text{ if } i=N.           
	\end{cases}
\end{align*}

Note that if $s_N = -1$, the simple root $\alpha_N$ is odd and non-isotropic.

The null root $\delta$ is even for all $\s$.

The longest positive root $\theta$ of $\Delta_\s$ depends on $\s$ and is given by:
\begin{align*}
	\theta = \begin{cases}
	\varepsilon_{1} + \varepsilon_{2}, & \text{ if } s_1 = 1;  \\
	2\varepsilon_{1},                  & \text{ if } s_1 = -1. 
	\end{cases}
\end{align*}

This implies that the parity of the simple root $\alpha_0 = \delta - \theta$ of $\ghs$ is given by:
\begin{align*}
	|\alpha_0| = \begin{cases}
	|\alpha_1|, & s_1 = 1;  \\
	0,          & s_1 = -1. 
	\end{cases}
\end{align*}

The root system of $\ghs$ is $\hat{\Delta}_\s=\{\alpha +r\delta \, |\, \alpha\in \Delta_\s \text{ and }  r\in\Z \}\cup (\Z\setminus\{0\} )\delta$.

Due to the two possible expressions for the longest root, the general shape of the associated Dynkin diagram also depends on the parity $\s$.

Following the nomenclature of \cite{Yamane1999}, the Dynkin diagram in Figure \ref{CBN} has shape $CB_N$, while the Dynkin diagram in Figure \ref{DBN} has shape $DB_N$. Also, we will write $(CB_N)_j$ and $(DB_N)_j$, where $j \equiv \sum_{i=1}^N|\alpha_i| \mod 2$, to specify the parity of the number of simple odd roots in the diagram.

\begin{figure}[H]
	\centering
	\begin{tikzpicture}[scale=1.5]
		\draw[rotate=45] (-0.1,0)--(0.1,0);
		\draw[rotate=-45] (-0.1,0)--(0.1,0);
		
		\draw (-1,0) circle (0.1);
		\draw (-0.87,0.05)--(-0.15,0.05);	
		\draw (-0.87,-0.05)--(-0.15,-0.05);	
		\draw[shift={(-0.1,0)}, rotate=45] (-0.15,0)--(0,0);
		\draw[shift={(-0.1,0)}, rotate=-45] (-0.15,0)--(0,0);
		 
		\draw (0.1,0)--(0.9,0);
			
		\draw[shift={(1,0)}, rotate=45] (-0.1,0)--(0.1,0);
		\draw[shift={(1,0)}, rotate=-45] (-0.1,0)--(0.1,0);
					
		\draw[dashed] (1.1,0)--(1.9,0);	
			
		\draw[shift={(2,0)}, rotate=45] (-0.1,0)--(0.1,0);
		\draw[shift={(2,0)}, rotate=-45] (-0.1,0)--(0.1,0);
					
		\draw (2.1,0)--(2.9,0);	
			
		\draw[shift={(3,0)}, rotate=45] (-0.1,0)--(0.1,0);
		\draw[shift={(3,0)}, rotate=-45] (-0.1,0)--(0.1,0);	
				
		\draw (3.1,0.05)--(3.8,0.05);	
		\draw (3.1,-0.05)--(3.8,-0.05);	
		\draw[shift={(3.86,0)}, rotate=45] (-0.15,0)--(0,0);
		\draw[shift={(3.86,0)}, rotate=-45] (-0.15,0)--(0,0);	
			
		
		\draw (4,0) circle (0.1);
		\draw[fill=black] (4,0) circle (0.033);
		
		\node [below] at (-1,-0.05) {\small $0$};
		\node [below] at (0,-0.05) {\Tiny $1$};
		\node [below] at (1,-0.05) {\Tiny $2$};
		\node [below] at (2,-0.05) {\Tiny $N-2$};
		\node [below] at (3,-0.05) {\tiny $N-1$};
		\node [below] at (4,-0.05) {\Tiny $N$};
	\end{tikzpicture}
	\caption{General Dynkin diagram of shape $CB_N$}
	\label{CBN}
\end{figure}

\begin{figure}[H]
	\centering
	\begin{tikzpicture}[scale=1.5]
		\draw[rotate=45] (-0.1,0)--(0.1,0);
		\draw[rotate=-45] (-0.1,0)--(0.1,0);
				
		\draw[rotate=45] (-0.8,0)--(-0.15,0);
		\draw[rotate=45] (0,0.15)--(0,0.8);
		
		\draw[rotate=45] (-0.85,0)--(-1.05,0);
		\draw[rotate=45] (-0.95,0.1)--(-0.95,-0.1);
		        
		\draw[rotate=-45] (-0.95,0.1)--(-0.95,-0.1);
		\draw[rotate=-45] (-0.85,0)--(-1.05,0);
		
		\draw[dotted] (-0.62,0.55)--(-0.62,-0.55);
		\draw[dotted] (-0.72,0.55)--(-0.72,-0.55);
		 
		\draw (0.1,0)--(0.9,0);
			
		\draw[shift={(1,0)}, rotate=45] (-0.1,0)--(0.1,0);
		\draw[shift={(1,0)}, rotate=-45] (-0.1,0)--(0.1,0);
					
		\draw[dashed] (1.1,0)--(1.9,0);	
			
		\draw[shift={(2,0)}, rotate=45] (-0.1,0)--(0.1,0);
		\draw[shift={(2,0)}, rotate=-45] (-0.1,0)--(0.1,0);
					
		\draw (2.1,0)--(2.9,0);	
			
		\draw[shift={(3,0)}, rotate=45] (-0.1,0)--(0.1,0);
		\draw[shift={(3,0)}, rotate=-45] (-0.1,0)--(0.1,0);	
				
		\draw (3.1,0.05)--(3.8,0.05);	
		\draw (3.1,-0.05)--(3.8,-0.05);	
		\draw[shift={(3.86,0)}, rotate=45] (-0.15,0)--(0,0);
		\draw[shift={(3.86,0)}, rotate=-45] (-0.15,0)--(0,0);	
			
		
		\draw (4,0) circle (0.1);
		\draw[fill=black] (4,0) circle (0.033);
		
		\node [left] at (-0.7,0.67) {\Tiny $0$};
		\node [left] at (-0.7,-0.67) {\Tiny $1$};
		\node [below] at (0,-0.05) {\Tiny $2$};
		\node [below] at (1,-0.05) {\Tiny $3$};
		\node [below] at (2,-0.05) {\Tiny $N-2$};
		\node [below] at (3,-0.05) {\Tiny $N-1$};
		\node [below] at (4,-0.05) {\Tiny $N$};
	\end{tikzpicture}
	\caption{General Dynkin diagram of shape $DB_N$}
	\label{DBN}
\end{figure}

\begin{figure}[H]
	\centering
	\begin{tikzpicture}[scale=1.5]
		  
		\draw[shift={(0.035,-0.035)},,rotate=45] (-0.8,0)--(-0.19,0);
		\draw[shift={(-0.035, 0.035)},,rotate=45] (-0.8,0)--(-0.19,0);
		\draw[shift={(0.035,0.035)},rotate=45] (0,0.19)--(0,0.8);
		\draw[shift={(-0.035,-0.035)},rotate=45] (0,0.19)--(0,0.8);
		
		\draw[shift={(-0.1,0.1)}] (0,0)--(-0.15,0);
		\draw[shift={(-0.1,0.1)}] (0,0)--(0,0.15);
		
		\draw[shift={(-0.1,-0.1)}] (0,0)--(-0.15,0);
		\draw[shift={(-0.1,-0.1)}] (0,0)--(0,-0.15);
		
		\draw[rotate=45] (-0.85,0)--(-1.05,0);
		\draw[rotate=45] (-0.95,0.1)--(-0.95,-0.1);
		        
		\draw[rotate=-45] (-0.95,0.1)--(-0.95,-0.1);
		\draw[rotate=-45] (-0.85,0)--(-1.05,0);
		
		\draw[dotted] (-0.62,0.54)--(-0.62,-0.55);
		\draw[dotted] (-0.72,0.54)--(-0.72,-0.55);

		\draw (0,0) circle (0.1);
		\draw[fill=black] (0,0) circle (0.033);
		
		\node [left] at (-0.7,0.67) {\Tiny $0$};
		\node [left] at (-0.7,-0.67) {\Tiny $1$};
		\node [below] at (0,-0.05) {\Tiny $2$};
	\end{tikzpicture}
	\caption{General Dynkin diagram of shape $DB_2$}
	\label{DB2}
\end{figure}

The Dynkin diagram of the superalgebra $\ghs=\osp(2m+1,2n)^{(1)}_\s$ has shape $(CB_N)_1$ if $s_1=-1$, and $(DB_N)_0$ if $s_1=1$. \footnote{The Dynkin diagrams of shape $(CB_N)_0$ and $(DB_N)_1$ are diagrams of the twisted affine superalgebra  $A(2m-1,2n)^{(2)}$, see \cite{Yamane1999}.} Note that, if $m=0$, $\cS$ is a singleton and the Dynkin diagram always has shape $(CB_N)_1$. In particular, if $m=0$ and $n=1$, the Dynkin diagram is the following:
\begin{figure}[H]
	\centering
	\begin{tikzpicture}[scale=1.5]
		\draw (0,0) circle (0.1);
		
		\draw[shift={(0.87,0)}, rotate=45] (-0.15,0)--(0,0);
		\draw[shift={(0.87,0)}, rotate=-45] (-0.15,0)--(0,0);

		\draw (0.12,0.02)--(0.85,0.02);
		\draw (0.12,0.06)--(0.81,0.06);
		\draw (0.12,-0.02)--(0.85,-0.02);
		\draw (0.12,-0.06)--(0.81,-0.06);

		\draw[fill=black] (1,0) circle (0.1);
		 
		\node [below] at (0,-0.05) {\Tiny $0$};
		\node [below] at (1,-0.05) {\Tiny $1$};
	\end{tikzpicture}
	\caption{Dynkin diagram of shape $CB_1$}
	\label{CB1}
\end{figure}

The Cartan matrix of $\ghs$ is obtained from the parity sequence $\s$ by the following formulas:

\begin{align*}
	         & A_{0,0}=\begin{cases}                                                                                                                   
	s_1+s_2, & s_1=1;                                                                                                                                     \\
	4s_1,    & s_1=-1.                                                                                                                                    
	\end{cases}\\
	         & A_{0,1}=A_{1,0}=\begin{cases}                                                                                                           
	s_2-s_1, & s_1=1;                                                                                                                                     \\
	-2s_1,   & s_1=-1.                                                                                                                                    
	\end{cases}\\
	         & A_{0,2}=A_{2,0}=\begin{cases}                                                                                                           
	-s_2,    & s_1=1;                                                                                                                                     \\
	0,       & s_1=-1.                                                                                                                                    
	\end{cases}\\
	         & A_{N,N}=-A_{N,N-1}=-A_{N-1,N}=s_N.                                                                                                \\
	         & A_{i,j}=(s_i+s_{i+1})\delta_{i,j}-s_{i}\delta_{i,j+1}-s_{j}\delta_{i+1,j},\; \text{for all } i,j \in \hat{I} \text{ not treated above}. 
\end{align*}

\subsubsection{Types $C$ and $D$ root systems}
The root system $\Delta_\s=\Delta_{0,\s} \cup \Delta_{1,\s}$ of $\g_\s$  is 
\begin{align*}
    &\Delta_\s=\{\pm \varepsilon_i \pm \varepsilon_j,\pm 2\varepsilon_k|1\leq i,j,k\leq N, s_k=-1\}\\
    &\Delta_{0,\s}=\{\pm \varepsilon_i \pm \varepsilon_j,\pm 2\varepsilon_k |1\leq i,j,k\leq N, s_i=s_j, s_k=-1\}\\
    &\Delta_{1,\s}=\{\pm \varepsilon_i \pm \varepsilon_j |1\leq i,j\leq N, s_i=-s_j\}
\end{align*}

The simple roots associated with $\s$ are given by
$\alpha_i=\varepsilon_{i}-\varepsilon_{i+1}$, $i=1,\dots,N-1$, and
\begin{align*}
   \alpha_N=\begin{cases}
                    \varepsilon_{N-1}+\varepsilon_{N}, & \text{ if } s_N=1;\\
                    2\varepsilon_{N}, & \text{ if } s_N=-1.
                \end{cases}
\end{align*}

The longest positive root $\theta$ of $\Delta_\s$ also depends on $\s$ and is given by
\begin{align*}
    \theta=\begin{cases}
        \varepsilon_{1}+\varepsilon_{2}, & \text{ if } s_1=1;\\
                    2\varepsilon_{1}, & \text{ if } s_1=-1.
    \end{cases}
\end{align*}

The parities of the simple roots $\alpha_i$, $i\in I$ are given by $|\alpha_i|=(1-s_i s_{i+1})/2$, $i\in I$. The null root $\delta$ is even for all $\s$. The parities of the simple roots $\alpha_0 = \delta - \theta$ and $\alpha_N$ are given by
\begin{align*}
    |\alpha_0|=\begin{cases} |\alpha_1|, & s_1=1;\\
    0, & s_1=-1.
\end{cases}
 &&|\alpha_N|=\begin{cases} |\alpha_{N-1}|, & s_N=1;\\
    0, & s_N=-1.
\end{cases}
\end{align*}

Now, since we have two possible expressions for the longest root and two possible expressions for the simple root $\alpha_N$, the general shape of the associated Dynkin diagram has four possibilities. 

We again follow the nomenclature of \cite{Yamane1999}, the Dynkin diagrams in Figure \ref{CCN}, Figure \ref{CDN}, Figure \ref{DCN}, and Figure \ref{DDN} have shape $CC_N$, $CD_N$, $DC_N$, and $DD_N$, respectively. Also, we will write $(CC_N)_j$, $(CD_N)_j$, $(DC_N)_j$, and $(DD_N)_j$, where $j\equiv \sum_{i=1}^{N-1}|\alpha_i|  \mod 2$, to specify the parity of the number of simple odd roots in the diagram.

\begin{figure}[H]
    \centering
    \begin{tikzpicture}[scale=1.5]
        \draw (-1,0) circle (0.1);
        \draw (-0.87,0.05)--(-0.15,0.05);	
		\draw (-0.87,-0.05)--(-0.15,-0.05);	
		\draw[shift={(-0.1,0)}, rotate=45] (-0.15,0)--(0,0);
			\draw[shift={(-0.1,0)}, rotate=-45] (-0.15,0)--(0,0);	
    
		\draw[rotate=45] (-0.1,0)--(0.1,0);
		\draw[rotate=-45] (-0.1,0)--(0.1,0);

	\draw (0.1,0)--(0.9,0);
	
	\draw[shift={(1,0)}, rotate=45] (-0.1,0)--(0.1,0);
			\draw[shift={(1,0)}, rotate=-45] (-0.1,0)--(0.1,0);
			
	\draw[dashed] (1.1,0)--(1.9,0);	
	
	\draw[shift={(2,0)}, rotate=45] (-0.1,0)--(0.1,0);
			\draw[shift={(2,0)}, rotate=-45] (-0.1,0)--(0.1,0);
			
	\draw (2.1,0)--(2.9,0);	
	
	\draw[shift={(3,0)}, rotate=45] (-0.1,0)--(0.1,0);
			\draw[shift={(3,0)}, rotate=-45] (-0.1,0)--(0.1,0);	
		
	\draw (3.15,0.05)--(3.87,0.05);	
		\draw (3.15,-0.05)--(3.87,-0.05);	
		\draw[shift={(3.1,0)}, rotate=45] (0.15,0)--(0,0);
			\draw[shift={(3.1,0)}, rotate=-45] (0.15,0)--(0,0);	
	

        \draw (4,0) circle (0.1);
        
        \node [below] at (-1,-0.05) {\Tiny $0$};
	\node [below] at (0,-0.05) {\Tiny $1$};
	\node [below] at (1,-0.05) {\Tiny $2$};
	\node [below] at (2,-0.05) {\Tiny $N-2$};
	\node [below] at (3,-0.05) {\Tiny $N-1$};
	\node [below] at (4,-0.05) {\Tiny $N$};
\end{tikzpicture}
    \caption{Dynkin diagram of shape $CC_N$}
    \label{CCN}
\end{figure}

\begin{figure}[H]
    \centering
    \begin{tikzpicture}[scale=1.5]
        \draw (-1,0) circle (0.1);
        \draw (-0.87,0.05)--(-0.15,0.05);	
		\draw (-0.87,-0.05)--(-0.15,-0.05);	
		\draw[shift={(-0.1,0)}, rotate=45] (-0.15,0)--(0,0);
			\draw[shift={(-0.1,0)}, rotate=-45] (-0.15,0)--(0,0);	
    
		\draw[rotate=45] (-0.1,0)--(0.1,0);
		\draw[rotate=-45] (-0.1,0)--(0.1,0);

	\draw (0.1,0)--(0.9,0);
	
	\draw[shift={(1,0)}, rotate=45] (-0.1,0)--(0.1,0);
			\draw[shift={(1,0)}, rotate=-45] (-0.1,0)--(0.1,0);
			
	\draw[dashed] (1.1,0)--(1.9,0);	
	
	\draw[shift={(2,0)}, rotate=45] (-0.1,0)--(0.1,0);
			\draw[shift={(2,0)}, rotate=-45] (-0.1,0)--(0.1,0);
			
	\draw (2.1,0)--(2.9,0);	
	
	\draw[shift={(3,0)}, rotate=45] (-0.1,0)--(0.1,0);
			\draw[shift={(3,0)}, rotate=-45] (-0.1,0)--(0.1,0);


        \draw[shift={(3,0)}, rotate=135] (-0.8,0)--(-0.15,0);
        \draw[shift={(3,0)}, rotate=-45] (0,0.15)--(0,0.8);

        \draw[shift={(3,0)}, rotate=135] (-0.85,0)--(-1.05,0);
        \draw[shift={(3,0)}, rotate=135] (-0.95,0.1)--(-0.95,-0.1);
        
        \draw[shift={(3,0)}, rotate=-135] (-0.95,0.1)--(-0.95,-0.1);
	\draw[shift={(3,0)}, rotate=-135] (-0.85,0)--(-1.05,0);

        \draw[dotted] (3.62,0.55)--(3.62,-0.55);
        \draw[dotted] (3.72,0.55)--(3.72,-0.55);
        
        \node [below] at (-1,-0.1) {\Tiny $0$};
	\node [below] at (0,-0.1) {\Tiny $1$};
	\node [below] at (1,-0.1) {\Tiny $2$};
	\node [below] at (2,-0.1) {\Tiny $N-3$};
	\node [below] at (2.9,-0.1) {\Tiny $N-2$};
	\node [right] at (3.7,0.67) {\Tiny $N-1$};
        \node [right] at (3.7,-0.67) {\Tiny $N$};
\end{tikzpicture}
    \caption{Dynkin diagram of shape $CD_N$}
    \label{CDN}
\end{figure}

\begin{figure}[H]
    \centering
    \begin{tikzpicture}[scale=1.5]

 \draw[rotate=45] (-0.8,0)--(-0.15,0);
        \draw[rotate=45] (0,0.15)--(0,0.8);

        \draw[rotate=45] (-0.85,0)--(-1.05,0);
        \draw[rotate=45] (-0.95,0.1)--(-0.95,-0.1);
        
        \draw[rotate=-45] (-0.95,0.1)--(-0.95,-0.1);
	\draw[rotate=-45] (-0.85,0)--(-1.05,0);

        \draw[dotted] (-0.62,0.55)--(-0.62,-0.55);
        \draw[dotted] (-0.72,0.55)--(-0.72,-0.55);

		\draw[rotate=45] (-0.1,0)--(0.1,0);
		\draw[rotate=-45] (-0.1,0)--(0.1,0);

	\draw (0.1,0)--(0.9,0);
	
	\draw[shift={(1,0)}, rotate=45] (-0.1,0)--(0.1,0);
			\draw[shift={(1,0)}, rotate=-45] (-0.1,0)--(0.1,0);
			
	\draw[dashed] (1.1,0)--(1.9,0);	
	
	\draw[shift={(2,0)}, rotate=45] (-0.1,0)--(0.1,0);
			\draw[shift={(2,0)}, rotate=-45] (-0.1,0)--(0.1,0);
			
	\draw (2.1,0)--(2.9,0);	
	
	\draw[shift={(3,0)}, rotate=45] (-0.1,0)--(0.1,0);
			\draw[shift={(3,0)}, rotate=-45] (-0.1,0)--(0.1,0);	
		
	\draw (3.15,0.05)--(3.87,0.05);	
		\draw (3.15,-0.05)--(3.87,-0.05);	
		\draw[shift={(3.1,0)}, rotate=45] (0.15,0)--(0,0);
			\draw[shift={(3.1,0)}, rotate=-45] (0.15,0)--(0,0);	
	

        \draw (4,0) circle (0.1);
        
        \node [left] at (-0.7,0.67) {\Tiny $0$};
        \node [left] at (-0.7,-0.67) {\Tiny $1$};
	\node [below] at (0,-0.05) {\Tiny $1$};
	\node [below] at (1,-0.05) {\Tiny $2$};
	\node [below] at (2,-0.05) {\Tiny $N-2$};
	\node [below] at (3,-0.05) {\Tiny $N-1$};
	\node [below] at (4,-0.05) {\Tiny $N$};
\end{tikzpicture}
    \caption{Dynkin diagram of shape $DC_N$}
    \label{DCN}
\end{figure}

\begin{figure}[H]
    \centering
    \begin{tikzpicture}[scale=1.5]
\draw[rotate=45] (-0.8,0)--(-0.15,0);
        \draw[rotate=45] (0,0.15)--(0,0.8);

        \draw[rotate=45] (-0.85,0)--(-1.05,0);
        \draw[rotate=45] (-0.95,0.1)--(-0.95,-0.1);
        
        \draw[rotate=-45] (-0.95,0.1)--(-0.95,-0.1);
	\draw[rotate=-45] (-0.85,0)--(-1.05,0);

        \draw[dotted] (-0.62,0.55)--(-0.62,-0.55);
        \draw[dotted] (-0.72,0.55)--(-0.72,-0.55);

		\draw[rotate=45] (-0.1,0)--(0.1,0);
		\draw[rotate=-45] (-0.1,0)--(0.1,0);

	\draw (0.1,0)--(0.9,0);
	
	\draw[shift={(1,0)}, rotate=45] (-0.1,0)--(0.1,0);
			\draw[shift={(1,0)}, rotate=-45] (-0.1,0)--(0.1,0);
			
	\draw[dashed] (1.1,0)--(1.9,0);	
	
	\draw[shift={(2,0)}, rotate=45] (-0.1,0)--(0.1,0);
			\draw[shift={(2,0)}, rotate=-45] (-0.1,0)--(0.1,0);
			
	\draw (2.1,0)--(2.9,0);	
	
	\draw[shift={(3,0)}, rotate=45] (-0.1,0)--(0.1,0);
			\draw[shift={(3,0)}, rotate=-45] (-0.1,0)--(0.1,0);


        \draw[shift={(3,0)}, rotate=135] (-0.8,0)--(-0.15,0);
        \draw[shift={(3,0)}, rotate=-45] (0,0.15)--(0,0.8);

        \draw[shift={(3,0)}, rotate=135] (-0.85,0)--(-1.05,0);
        \draw[shift={(3,0)}, rotate=135] (-0.95,0.1)--(-0.95,-0.1);
        
        \draw[shift={(3,0)}, rotate=-135] (-0.95,0.1)--(-0.95,-0.1);
	\draw[shift={(3,0)}, rotate=-135] (-0.85,0)--(-1.05,0);

        \draw[dotted] (3.62,0.55)--(3.62,-0.55);
        \draw[dotted] (3.72,0.55)--(3.72,-0.55);
        
        \node [left] at (-0.7,0.67) {\Tiny $0$};
        \node [left] at (-0.7,-0.67) {\Tiny $1$};
	\node [below] at (0,-0.1) {\Tiny $2$};
	\node [below] at (1,-0.1) {\Tiny $3$};
	\node [below] at (2,-0.1) {\Tiny $N-3$};
	\node [below] at (2.9,-0.1) {\Tiny $N-2$};
	\node [right] at (3.7,0.67) {\Tiny $N-1$};
        \node [right] at (3.7,-0.67) {\Tiny $N$};
\end{tikzpicture}
    \caption{Dynkin diagram of shape $DD_N$}
    \label{DDN}
\end{figure}

The Dynkin diagram of the superalgebra $\osp(2m,2n)^{(1)}$ has shape $(CC_N)_0$, if $s_1=s_N=-1$, shape $(CD_N)_1$, if $s_1=-s_N=-1$, shape $(DC_N)_1$, if $s_1=-s_N=1$, and shape $(DD_N)_0$, if $s_1=s_N=1$.\footnote{The Dynkin diagrams of shapes $(CC_N)_1$, $(CD_N)_0$, $(DC_N)_0$, and $(DD_N)_1$ are diagrams of the twisted affine superalgebra  $A(2m-1,2n-1)^{(2)}$, see \cite{Yamane1999}.}

The Cartan matrix of $\osp(2m,2n)^{(1)}$ is obtained from the parity sequence $\s$ by the following formulas:

\begin{align*}
    & A_{0,0}=\begin{cases}
        s_1+s_2,& s_1=1;\\
        4s_1, & s_1=-1.
    \end{cases}
                && A_{N,N}=\begin{cases}
        s_{N-1}+s_N,& s_N=1;\\
        4s_N, & s_N=-1.
    \end{cases}\\
    & A_{0,1}=A_{1,0}=\begin{cases}
        s_2-s_1,& s_1=1;\\
        -2s_1, & s_1=-1.
    \end{cases}
                && A_{{N-1},N}=A_{N,{N-1}}=\begin{cases}
        s_{N-1}-s_N,& s_N=1;\\
        -2s_N, & s_N=-1.
    \end{cases}\\
    & A_{0,2}=A_{2,0}=\begin{cases}
        -s_2,& s_1=1;\\
        0, & s_1=-1.
    \end{cases}
                && A_{{N-2},N}=A_{N,{N-2}}=\begin{cases}
        -s_{N-1},& s_N=1;\\
        0, & s_N=-1.
    \end{cases}
\end{align*}
$A_{i,j}=(s_i+s_{i+1})\delta_{i,j}-s_{i}\delta_{i,j+1}-s_{j}\delta_{i+1,j},\; \text{for all } i,j \in \hat{I} \text{ not treated above}.$

\subsection{The Affine Weyl Groupoid }

Recall that the Weyl group of $\g$ is defined as the Weyl group of its even subalgebra $\g_0 \subseteq \g$, and it is generated by the reflections associated with even roots. However, this group does not contain enough symmetries of $\g$.

The \emph{odd reflections}, first considered in \cite{leites1985embeddings}, mimic the reflection associated with even roots in a type $A$ region of the Dynkin diagram. That is, the odd reflection associated with the root $\alpha$ acts as follows:
\begin{align*}
	  & \alpha \mapsto -\alpha,                                                               \\
	  & \beta \mapsto \alpha +\beta, \; \text{if } (\alpha|\beta)\neq 0, \alpha\neq \pm \beta \\
	  & \beta \mapsto \beta,\; \text{if } (\alpha|\beta)= 0,\alpha\neq \pm \beta.             
\end{align*}

Even though the odd reflections provide additional symmetries, they do not preserve the Cartan data. The reflection associated with an odd root maps a set of positive roots to another set of positive roots (or, in the nomenclature of \cite{chengwang}, positive systems to positive systems).

Moreover, the original odd reflections do not behave well at the extremes of some Dynkin diagrams. This was a reason for the introduction of \emph{generalized reflections} by Yamane in \cite{Yamane1999}. 
We describe the (action) groupoid generated by the generalized reflections below.
First, consider the ``virtual root'' $\alpha_{i,\s}^{\dagger}$ corresponding to the simple root $\alpha_{i}$ defined as follows:
\begin{align*}
    \alpha_i^{\dagger}:=\begin{cases} 
                            \delta-2\varepsilon_1, & \text{if } i=0;\\
                            2\varepsilon_N, & \text{if $i=N$ and $\g=\osp(2m,2n)^{(1)}$};\\
                            \alpha_i, & \text{otherwise}.
                        \end{cases}
\end{align*}

Now, define a linear map $\iota_\s:P_\s\rightarrow P_{\bs 1} $ given by $\varepsilon_{i}\mapsto \varepsilon_{i}$ and $\delta\mapsto \delta$, where $\bs{1}$ is the parity sequence with all entries equal to $1$.

\begin{prop}[{\cite[Prop. 2.4.3]{Yamane1999}}] For each $i\in \hat I$, there is an even isometry $r_{i,\s}: P_\s\rightarrow P_{\sigma_i\cdot \s}$ given by
	\begin{align}
		r_{i,\s}(\alpha)=(\iota_{\sigma_i\cdot \s})^{-1}\left( \iota_\s(\alpha)-2\frac{(\iota_\s(\alpha)|\iota_\s(\alpha_{i}^{\dagger}))}{(\iota_\s(\alpha_{i}^{\dagger})|\iota_\s(\alpha_{i}^{\dagger}))}\iota_\s(\alpha_{i}^{\dagger}) \right). 
	\end{align}
	
\end{prop}

\begin{exa}
Consider $\g$ of type $B$ with $N>2$.
Let us analyze the action of the generalized reflections $r_{0,\s}$ and $r_{1,\s}$ on the roots $\alpha_{0}$, $\alpha_{1}$. Their parities depend on $s_1$ and $s_2$. Thus, we have four possibilities. 

In all cases, $r_{0,\s}(\alpha_{1})=\delta_{ \s} -\varepsilon_{1}-\varepsilon_{2}$. However, since the expression of $\alpha_0$ depends on $s_1$, we have:
\begin{align*}
	r_{0,\s}(\alpha_{1})=\delta -\varepsilon_{1}-\varepsilon_{2}=\begin{cases} 
	\alpha_{0},                 & \text{if } s_1=1;  \\
	\alpha_{0} + \alpha_{1}, & \text{if } s_1=-1. 
	\end{cases}
\end{align*}
Note that the action of $r_{0,\s}$ preserves the parity.

The case of $r_{1,\s}(\alpha_{0})$ is slightly more intricate:
\begin{align*}
	r_{1,\s}(\alpha_{0})                                                                       & =\begin{cases}     
	\delta -\varepsilon_{1}-\varepsilon_{2}, & \text{if } s_1=1;  \\
	\delta -2\varepsilon_{2} ,                                & \text{if } s_1=-1. 
	\end{cases}
       =\begin{cases}    
      \alpha_{0},                                      & \text{if } s_1=s_2=1;  \\
	\alpha_{0}+\alpha_{1},                                      & \text{if } s_1\neq s_2;  \\
	\alpha_{0} + 2\alpha_{1},                                   & \text{if } s_1=s_2=-1. 
	\end{cases}
\end{align*}

The following picture illustrates how the generalized reflections change the parities and the respective Dynkin diagrams. 

\begin{figure}[H]
	\centering
	\begin{tikzpicture}[scale=1.5]
		
		\begin{scope}
			\draw[rotate=45] (-0.1,0)--(0.1,0);
			\draw[rotate=-45] (-0.1,0)--(0.1,0);
			\draw (0,0) circle (0.1);
					
			\draw[rotate=45] (-0.85,0)--(-0.05,0);
			\draw[rotate=45] (0,0.05)--(0,0.85);
			\draw (-0.67, 0.67) circle (0.1);

			\draw (-0.67, -0.67) circle (0.1);
			        
			\draw (1,0) circle (0.1);
			 
			\draw (0.1,0)--(0.9,0);

			\draw (1.1,0.05)--(1.8,0.05);	
			\draw (1.1,-0.05)--(1.8,-0.05);	
			\draw[shift={(1.86,0)}, rotate=45] (-0.15,0)--(0,0);
			\draw[shift={(1.86,0)}, rotate=-45] (-0.15,0)--(0,0);	
				
			
			\draw (2,0) circle (0.1);
			\draw[fill=black] (2,0) circle (0.1);
			
			\node [left] at (-0.7,0.67) {\Tiny $0$};
			\node [left] at (-0.7,-0.67) {\Tiny $1$};
			\node [below] at (0,-0.05) {\Tiny $2$};
			\node [below] at (1,-0.05) {\Tiny $3$};
			\node [below] at (2,-0.05) {\Tiny $4$};
		\end{scope}
		
		----------------------------------
		\begin{scope}[shift={(0,-2.5)}]
			\draw[rotate=45] (-0.1,0)--(0.1,0);
			\draw[rotate=-45] (-0.1,0)--(0.1,0);
			\draw (0,0) circle (0.1);
					
			\draw[rotate=45] (-0.85,0)--(-0.05,0);
			\draw[rotate=45] (0,0.05)--(0,0.85);
			\draw (-0.67, 0.67) circle (0.1);
			
			\draw[rotate=45] (-0.85,0)--(-1.05,0);
			\draw[rotate=45] (-0.95,0.1)--(-0.95,-0.1);
			\draw (-0.67, -0.67) circle (0.1);
			        
			\draw[rotate=-45] (-0.95,0.1)--(-0.95,-0.1);
			\draw[rotate=-45] (-0.85,0)--(-1.05,0);
			\draw (1,0) circle (0.1);
			
			\draw (-0.62,0.58)--(-0.62,-0.58);
			\draw (-0.72,0.58)--(-0.72,-0.58);
			 
			\draw (0.1,0)--(0.9,0);
				
			\draw[shift={(1,0)}, rotate=45] (-0.1,0)--(0.1,0);
			\draw[shift={(1,0)}, rotate=-45] (-0.1,0)--(0.1,0);

			\draw (1.1,0.05)--(1.8,0.05);	
			\draw (1.1,-0.05)--(1.8,-0.05);	
			\draw[shift={(1.86,0)}, rotate=45] (-0.15,0)--(0,0);
			\draw[shift={(1.86,0)}, rotate=-45] (-0.15,0)--(0,0);	
				
			
			\draw (2,0) circle (0.1);
			\draw[fill=black] (2,0) circle (0.1);
			
			\node [left] at (-0.7,0.67) {\Tiny $0$};
			\node [left] at (-0.7,-0.67) {\Tiny $1$};
			\node [below] at (0,-0.05) {\Tiny $2$};
			\node [below] at (1,-0.05) {\Tiny $3$};
			\node [below] at (2,-0.05) {\Tiny $4$};
		\end{scope}

		----------------------------------
		\begin{scope}[shift={(0,-5)}]
			
			\draw (0,0) circle (0.1);
					
			\draw[rotate=45] (-0.85,0)--(-0.1,0);
			\draw[rotate=45] (0,0.1)--(0,0.85);
			\draw (-0.67, 0.67) circle (0.1);
			
			\draw[rotate=45] (-0.85,0)--(-1.05,0);
			\draw[rotate=45] (-0.95,0.1)--(-0.95,-0.1);
			\draw (-0.67, -0.67) circle (0.1);
			        
			\draw[rotate=-45] (-0.95,0.1)--(-0.95,-0.1);
			\draw[rotate=-45] (-0.85,0)--(-1.05,0);
			\draw (1,0) circle (0.1);
			
			\draw (-0.62,0.58)--(-0.62,-0.58);
			\draw (-0.72,0.58)--(-0.72,-0.58);
			 
			\draw (0.1,0)--(0.9,0);
				
			\draw[shift={(1,0)}, rotate=45] (-0.1,0)--(0.1,0);
			\draw[shift={(1,0)}, rotate=-45] (-0.1,0)--(0.1,0);

			\draw (1.1,0.05)--(1.8,0.05);	
			\draw (1.1,-0.05)--(1.8,-0.05);	
			\draw[shift={(1.86,0)}, rotate=45] (-0.15,0)--(0,0);
			\draw[shift={(1.86,0)}, rotate=-45] (-0.15,0)--(0,0);	
				
			
			\draw (2,0) circle (0.1);

			\node [left] at (-0.7,0.67) {\Tiny $0$};
			\node [left] at (-0.7,-0.67) {\Tiny $1$};
			\node [below] at (0,-0.05) {\Tiny $2$};
			\node [below] at (1,-0.05) {\Tiny $3$};
			\node [below] at (2,-0.05) {\Tiny $4$};
		\end{scope}
		----------------------------------

		\begin{scope}[shift={(5,0)}]
			\draw  (-1,0) circle (0.1);
			        
			\draw  (-0.9,0.05)--(-0.2,0.05);	
			\draw  (-0.9,-0.05)--(-0.2,-0.05);	
			 
			\draw[shift={(-0.14,0)}, rotate=45] (-0.15,0)--(0,0);
			\draw[shift={(-0.14,0)}, rotate=-45] (-0.15,0)--(0,0);
			
			\draw[ rotate=45] (-0.1,0)--(0.1,0);
			\draw[ rotate=-45] (-0.1,0)--(0.1,0);
			  
			\draw  (0,0) circle (0.1);

			\draw  (1,0) circle (0.1);
			 
			\draw  (0.1,0)--(0.9,0);

			\draw[shift={(0,0)}, ] (1.1,0)--(1.9,0);

			\draw  (2.1,0.05)--(2.8,0.05);	
			\draw  (2.1,-0.05)--(2.8,-0.05);	
			\draw[shift={(2.86,0)}, rotate=45] (-0.15,0)--(0,0);
			\draw[shift={(2.86,0)}, rotate=-45] (-0.15,0)--(0,0);	
			\draw[shift={(1,0)}] (1,0) circle (0.1);
				
			
			\draw[shift={(1,0)}] (2,0) circle (0.1);
			\draw[shift={(1,0)}, fill=black] (2,0) circle (0.1);
			
			\draw[shift={(2,0)}, rotate=45] (-0.1,0)--(0.1,0);
			\draw[shift={(2,0)}, rotate=-45] (-0.1,0)--(0.1,0);

			\node [shift={(0,0)}, below] at (-1,-0.05) {\Tiny $0$};
			\node [shift={(0,0)}, below] at (0,-0.05) {\Tiny $1$};
			\node [shift={(0,0)}, below] at (1,-0.05) {\Tiny $2$};
			\node [shift={(0,0)}, below] at (2,-0.05) {\Tiny $3$};
			\node [shift={(0,0)}, below] at (3,-0.05) {\Tiny $4$};
		\end{scope}
		
		----------------------------------

		\begin{scope}[shift={(5,-2.5)}]
			\draw  (-1,0) circle (0.1);
			        
			\draw  (-0.9,0.05)--(-0.2,0.05);	
			\draw  (-0.9,-0.05)--(-0.2,-0.05);	
			 
			\draw[shift={(-0.14,0)}, rotate=45] (-0.15,0)--(0,0);
			\draw[shift={(-0.14,0)}, rotate=-45] (-0.15,0)--(0,0);
			
			\draw[ rotate=45] (-0.1,0)--(0.1,0);
			\draw[ rotate=-45] (-0.1,0)--(0.1,0);
			  
			\draw  (0,0) circle (0.1);

			\draw  (1,0) circle (0.1);
			 
			\draw  (0.1,0)--(0.9,0);
				
			\draw[shift={(1,0)}, rotate=45] (-0.1,0)--(0.1,0);
			\draw[shift={(1,0)}, rotate=-45] (-0.1,0)--(0.1,0);
			
			\draw[shift={(0,0)}, ] (1.1,0)--(1.9,0);

			\draw  (2.1,0.05)--(2.8,0.05);	
			\draw  (2.1,-0.05)--(2.8,-0.05);	
			\draw[shift={(2.86,0)}, rotate=45] (-0.15,0)--(0,0);
			\draw[shift={(2.86,0)}, rotate=-45] (-0.15,0)--(0,0);	
			\draw[shift={(1,0)}] (1,0) circle (0.1);
				
			
			\draw[shift={(1,0)}] (2,0) circle (0.1);
			
			\draw[shift={(2,0)}, rotate=45] (-0.1,0)--(0.1,0);
			\draw[shift={(2,0)}, rotate=-45] (-0.1,0)--(0.1,0);

			\node [shift={(0,0)}, below] at (-1,-0.05) {\Tiny $0$};
			\node [shift={(0,0)}, below] at (0,-0.05) {\Tiny $1$};
			\node [shift={(0,0)}, below] at (1,-0.05) {\Tiny $2$};
			\node [shift={(0,0)}, below] at (2,-0.05) {\Tiny $3$};
			\node [shift={(0,0)}, below] at (3,-0.05) {\Tiny $4$};
		\end{scope}
		
		-----------------------------------------------
		
		\begin{scope}[shift={(5,-5)}]
			\draw  (-1,0) circle (0.1);
			        
			\draw  (-0.9,0.05)--(-0.2,0.05);	
			\draw  (-0.9,-0.05)--(-0.2,-0.05);	
			 
			\draw[shift={(-0.14,0)}, rotate=45] (-0.15,0)--(0,0);
			\draw[shift={(-0.14,0)}, rotate=-45] (-0.15,0)--(0,0);

			\draw  (0,0) circle (0.1);

			\draw  (1,0) circle (0.1);
			 
			\draw  (0.1,0)--(0.9,0);
				
			\draw[shift={(1,0)}, rotate=45] (-0.1,0)--(0.1,0);
			\draw[shift={(1,0)}, rotate=-45] (-0.1,0)--(0.1,0);
			
			\draw[shift={(0,0)}, ] (1.1,0)--(1.9,0);

			\draw  (2.1,0.05)--(2.8,0.05);	
			\draw  (2.1,-0.05)--(2.8,-0.05);	
			\draw[shift={(2.86,0)}, rotate=45] (-0.15,0)--(0,0);
			\draw[shift={(2.86,0)}, rotate=-45] (-0.15,0)--(0,0);	
			\draw[shift={(1,0)}] (1,0) circle (0.1);
				
			
			\draw[shift={(1,0)}] (2,0) circle (0.1);
			
			\node [shift={(0,0)}, below] at (-1,-0.05) {\Tiny $0$};
			\node [shift={(0,0)}, below] at (0,-0.05) {\Tiny $1$};
			\node [shift={(0,0)}, below] at (1,-0.05) {\Tiny $2$};
			\node [shift={(0,0)}, below] at (2,-0.05) {\Tiny $3$};
			\node [shift={(0,0)}, below] at (3,-0.05) {\Tiny $4$};
		\end{scope}

		\usetikzlibrary {arrows.meta}

		\begin{scope}
			\draw[{Classical TikZ Rightarrow[length=4]}-{Classical TikZ Rightarrow[length=4]}] (0.67,-0.75)--(0.67,-1.75);
		\end{scope}
		
		\begin{scope}[shift={(0,-2.5)}]
			\draw[{Classical TikZ Rightarrow[length=4]}-{Classical TikZ Rightarrow[length=4]}] (0.67,-0.75)--(0.67,-1.75);
		\end{scope}
		
		\begin{scope}[shift={(5.33,0)}]
			\draw[{Classical TikZ Rightarrow[length=4]}-{Classical TikZ Rightarrow[length=4]}] (0.67,-0.75)--(0.67,-1.75);
		\end{scope}
		
		\begin{scope}[shift={(5.33,-2.5)}]
			\draw[{Classical TikZ Rightarrow[length=4]}-{Classical TikZ Rightarrow[length=4]}] (0.67,-0.75)--(0.67,-1.75);
		\end{scope}
		
		\begin{scope}
			\draw[{Classical TikZ Rightarrow[length=4]}-{Classical TikZ Rightarrow[length=4]}] (2.5,-1.75)--(3.5,-.75);
		\end{scope}
		
		\begin{scope}[shift={(0,-2.5)}]
			\draw[{Classical TikZ Rightarrow[length=4]}-{Classical TikZ Rightarrow[length=4]}] (2.5,-1.75)--(3.5,-.75);
		\end{scope}
		
		\node [shift={(0,0)}, right] at (0.67,-1.25) {\small $r_2$};
		\node [shift={(0,0)}, right] at (0.67,-3.75) {\small $r_3$};
		
		\node [shift={(0,0)}, right] at (6,-1.25) {\small $r_3$};
		\node [shift={(0,0)}, right] at (6,-3.75) {\small $r_2$};
		
		\node [shift={(0,0)}, right] at (3,-1.4) {\small $r_1$};
		\node [shift={(0,0)}, right] at (3,-3.9) {\small $r_1$};
		
		\begin{scope}[shift={(0,0)}]
			\node [shift={(0,0)}, above] at (1,0.3) {\tiny $\s=(1,1,-1,-1)$};
		\end{scope}
		
		\begin{scope}[shift={(0,-2.5)}]
			\node [shift={(0,0)}, above] at (1,0.3) {\tiny $\s=(1,-1,1,-1)$};
		\end{scope}
		
		\begin{scope}[shift={(0,-5)}]
			\node [shift={(0,0)}, above] at (1,0.3) {\tiny $\s=(1,-1,-1,1)$};
		\end{scope}
		
		\begin{scope}[shift={(5,0)}]
			\node [shift={(0,0)}, above] at (1,0.3) {\tiny $\s=(-1,1,1,-1)$};
		\end{scope}
		
		\begin{scope}[shift={(5,-2.5)}]
			\node [shift={(0,0)}, above] at (1,0.3) {\tiny $\s=(-1,1,-1,1)$};
		\end{scope}
		
		\begin{scope}[shift={(5,-5)}]
			\node [shift={(0,0)}, above] at (1,0.3) {\tiny $\s=(-1,-1,1,1)$};
		\end{scope}
		
	\end{tikzpicture}
	\caption{All Dynkin diagrams of $\osp(5,4)^{(1)}$.}
	\label{orbit of diagrams}
\end{figure}

\end{exa}

\medskip

\section{The quantum affine orthosymplectic superalgebra }\label{sect-quantum-JD}

Given $r \in \Z$, define $[r]_q=\frac{q^{r}-q^{-r}}{q-q^{-1}}$.

\begin{dfn}[\cite{Yamane1999}]
	Let $\g$ be of type $B$, $C$ or $D$, and $\s \in \cS$. The quantum affine superalgebra $\Uq (\ghs)$ is the unital associative superalgebra generated by $E_i, F_i, K_i^{\pm 1}$, $i\in \hI$, subject to the relations below. The parity of generators is given by $|E_i|=|F_i|=|\alpha_{i,\s}|$ and $|K_i^{\pm 1}|=0$. 
		
	\noindent The defining relations are as follows.
	\begin{align}
		  & K_iK_j=K_jK_i.\label{DJ kk} \\ & K_iE_jK_i^{-1}=q^{A_{i,j}}E_j,\quad K_iF_jK_i^{-1}=q^{-A_{i,j}}F_j. \label{DJ KE}\\
		&[E_i,F_j]=\delta_{i,j}\frac{K_i-K_i^{-1}}{q-q^{-1}}. \label{DJ EF}\\
		& \text{If } A_{i,j}=0, \text{then } [E_i,E_j]=0.\label{EE rel}\\
		&\text{If } A_{i,i}\neq 0 \if \text{ and } |i|\frac{2A_{i,j}}{A_{i,i}} \in 2\Z (\luan{\text{only relevant in twisted case}})\fi, \text{then } \lb{E_i,\lb{E_i,\lb{\cdots,\lb{E_i,E_j}}}}=0, \label{Serre traditional} \text{ where } E_i \text{ appears } 1-\frac{2A_{i,j}}{A_{i,i}} \,\text{times}. \\
		&\text{If } A_{j,j}=0 \text{ and } A_{i,j}=-A_{j,k}\neq 0, \text{ i.e., if }   \label{Serre odd node}
		\begin{tikzpicture}[baseline=-3pt]
		\draw[rotate=45] (-0.1,0)--(0.1,0);
		\draw[rotate=-45] (-0.1,0)--(0.1,0);
		\draw (1,0) circle (0.1);
		\draw[shift={(1,0)},rotate=45] (-0.1,0)--(0.1,0);
		\draw[shift={(1,0)},rotate=-45] (-0.1,0)--(0.1,0);
		\draw (0.1,0)--(0.9,0);
		\draw (1.1,0)--(1.9,0);
		\draw[shift={(2,0)},rotate=45] (-0.1,0)--(0.1,0);
		\draw[shift={(2,0)},rotate=-45] (-0.1,0)--(0.1,0);
		\node [below] at (0,0) {\Tiny $i$};
		\node [below] at (1,0) {\Tiny $j$};
		\node [below] at (2,0) {\Tiny $k$};
		\end{tikzpicture},  \text{ or }  \begin{tikzpicture}[baseline=-3pt]
		\draw[rotate=45] (-0.1,0)--(0.1,0);
		\draw[rotate=-45] (-0.1,0)--(0.1,0);
		\draw (1,0) circle (0.1);
		\draw[shift={(1,0)},rotate=45] (-0.1,0)--(0.1,0);
		\draw[shift={(1,0)},rotate=-45] (-0.1,0)--(0.1,0);
		\draw (0.1,0)--(0.9,0);
		\draw (1.1,0.05)--(1.81,0.05);
		\draw (1.1,-0.05)--(1.81,-0.05);
		\draw[shift={(1.86,0)}, rotate=45] (-0.15,0)--(0,0);
		\draw[shift={(1.86,0)}, rotate=-45] (-0.15,0)--(0,0);
		\draw (2,0) circle (0.1);
		\draw[fill=black] (2,0) circle (0.033);
		\node [below] at (0,0) {\Tiny $i$};
		\node [below] at (1,0) {\Tiny $j$};
		\node [below] at (2,0) {\Tiny $k$};
		\end{tikzpicture}, {then } \\
		& \lb{\lb{\lb{E_i,E_j},E_k},E_j}=0. \notag\\
		& \label{Serre DJ Sdo}\text{If } \begin{tikzpicture}[baseline=-3pt]
		\draw (0,0) circle (0.1);
		\draw[rotate=45] (-0.1,0)--(0.1,0);
		\draw[rotate=-45] (-0.1,0)--(0.1,0);
		\draw (1,0) circle (0.1);
		\draw[shift={(1,0)},rotate=45] (-0.1,0)--(0.1,0);
		\draw[shift={(1,0)},rotate=-45] (-0.1,0)--(0.1,0);
		\draw (0.1,0)--(0.9,0);
		\draw (1.2,0.05)--(1.9,0.05);
		\draw (1.2,-0.05)--(1.9,-0.05);
		\draw[shift={(1.15,0)}, rotate=45] (0.15,0)--(0,0);
		\draw[shift={(1.15,0)}, rotate=-45] (0.15,0)--(0,0);
		\draw (2,0) circle (0.1);
		\node [below] at (0,0) {\Tiny $i$};
		\node [below] at (1,0) {\Tiny $j$};
		\node [below] at (2,0) {\Tiny $k$};
		\end{tikzpicture}, {then }\\
		& \lb{\lb{\lb{E_i,E_j},\lb{\lb{E_i,E_j},E_k}},E_j}=0.\notag \\
		& \label{Serre DJ Sde} \text{If } \begin{tikzpicture}[baseline=-3pt]
		\draw[shift={(-1,0)},rotate=45] (-0.1,0)--(0.1,0);
		\draw[shift={(-1,0)},rotate=-45] (-0.1,0)--(0.1,0);
		\draw (-0.1,0)--(-0.9,0);
		\draw (0,0) circle (0.1);
		\draw (1,0) circle (0.1);
		\draw[shift={(1,0)},rotate=45] (-0.1,0)--(0.1,0);
		\draw[shift={(1,0)},rotate=-45] (-0.1,0)--(0.1,0);
		\draw (0.1,0)--(0.9,0);
		\draw (1.2,0.05)--(1.9,0.05);
		\draw (1.2,-0.05)--(1.9,-0.05);
		\draw[shift={(1.15,0)}, rotate=45] (0.15,0)--(0,0);
		\draw[shift={(1.15,0)}, rotate=-45] (0.15,0)--(0,0);
		\draw (2,0) circle (0.1);
		\node [below] at (-1,0) {\Tiny $i$};
		\node [below] at (0,0) {\Tiny $j$};
		\node [below] at (1,0) {\Tiny $k$};
		\node [below] at (2,0) {\Tiny $l$};
		\end{tikzpicture}, \text{ or } \begin{tikzpicture}[baseline=-3pt]
		\draw (-1,0) circle (0.1);
		\draw[fill=black] (-1,0) circle (0.033);
		\draw[shift={(-0.85,0)}, rotate=45] (0.15,0)--(0,0);
		\draw[shift={(-0.85,0)}, rotate=-45] (0.15,0)--(0,0);
		\draw (-0.1,0.05)--(-0.8,0.05);
		\draw (-0.1,-0.05)--(-0.8,-0.05);
		\draw (0,0) circle (0.1);
		\draw (1,0) circle (0.1);
		\draw[shift={(1,0)},rotate=45] (-0.1,0)--(0.1,0);
		\draw[shift={(1,0)},rotate=-45] (-0.1,0)--(0.1,0);
		\draw (0.1,0)--(0.9,0);
		\draw (1.2,0.05)--(1.9,0.05);
		\draw (1.2,-0.05)--(1.9,-0.05);
		\draw[shift={(1.15,0)}, rotate=45] (0.15,0)--(0,0);
		\draw[shift={(1.15,0)}, rotate=-45] (0.15,0)--(0,0);
		\draw (2,0) circle (0.1);
		\node [below] at (-1,0) {\Tiny $i$};
		\node [below] at (0,0) {\Tiny $j$};
		\node [below] at (1,0) {\Tiny $k$};
		\node [below] at (2,0) {\Tiny $l$};
		\end{tikzpicture}, {then } \\
		& \lb{\lb{\lb{\lb{\lb{\lb{E_i,E_j},E_k},E_l},E_k},E_j},E_k}=0.\notag\\
		& \text{If } A_{i,j} A_{i,k} A_{j,k}\neq 0,\; A_{i,j} + A_{i,k} + A_{j,k}= 0,\; \text{and }\, |i||j|+|j||k|+|i||k|\equiv 1, \text{ then } \label{Serre triangle}\\
		&  (-1)^{|i||k|}[(\alpha_i|\alpha_k)]\lb{\lb{E_i,E_j},E_k}=(-1)^{|i||j|}[(\alpha_i|\alpha_j)]\lb{\lb{E_i,E_k},E_j}.\notag\\
		& \text{If } \begin{tikzpicture}[baseline=-3pt]
		\draw (-1,0) circle (0.1);
		\draw (-0.91,0.05)--(-0.20,0.05);	
		\draw (-0.91,-0.05)--(-0.20,-0.05);	
		\draw[shift={(-0.15,0)}, rotate=45] (-0.15,0)--(0,0);
		\draw[shift={(-0.15,0)}, rotate=-45] (-0.15,0)--(0,0);	
		\draw (0,0) circle (0.1);
		\draw (0.09,0.05)--(0.80,0.05);	
		\draw (0.09,-0.05)--(0.8,-0.05);	
		\draw[shift={(0.85,0)}, rotate=45] (-0.15,0)--(0,0);
		\draw[shift={(0.85,0)}, rotate=-45] (-0.15,0)--(0,0);
		\draw (1,0) circle (0.1);
		\draw[fill=black] (1,0) circle (0.033);
		\node [below] at (-1,0) {\Tiny $i$};
		\node [below] at (0,0) {\Tiny $j$};
		\node [below] at (1,0) {\Tiny $k$};
		\end{tikzpicture},\label{3 nodes relation}{then } \\
		& \lb{\lb{E_k,E_j},\lb{\lb{E_k,E_j},\lb{\lb{E_k,E_j},E_i}}}=(1-(-1)^{|k|}[2]) \lb{\lb{\lb{E_k,E_j},\lb{E_k,\lb{E_k,\lb{E_j,E_i}}}},E_j}.\notag \\
        & \text{If } \begin{tikzpicture}[baseline=-3pt]
		\draw (-1,0) circle (0.1);
		\draw (-0.91,0.05)--(-0.20,0.05);	
		\draw (-0.91,-0.05)--(-0.20,-0.05);	
		\draw[shift={(-0.15,0)}, rotate=45] (-0.15,0)--(0,0);
		\draw[shift={(-0.15,0)}, rotate=-45] (-0.15,0)--(0,0);	
		\draw (0,0) circle (0.1);
  \draw[rotate=45] (0.15,0)--(0.85,0);
        \draw[rotate=-45] (0.15,0)--(0.85,0);
         \draw (.67,0.67) circle (0.1);
        \draw (.67,-0.67) circle (0.1);
        \draw[shift={(0,0)},rotate=45] (-0.1,0)--(0.1,0);
		\draw[shift={(0,0)},rotate=-45] (-0.1,0)--(0.1,0);
        \node [right] at (0.7,0.67) {\Tiny $k$};
        \node [right] at (0.7,-0.67) {\Tiny $\ell$};
		\node [below] at (-1,-.05) {\Tiny $i$};
		\node [below] at (0,-.05) {\Tiny $j$};
		\end{tikzpicture},{then } \\
		& \lb{\lb{\lb{E_i,E_j},\lb{E_j,E_k}},\lb{E_j,E_\ell}}=\lb{\lb{\lb{E_i,E_j},\lb{E_j,E_\ell}},\lb{E_j,E_k}}.\notag \\
		& \text{The relations \eqref{EE rel}--\eqref{3 nodes relation} with $E$'s replaced by $F$'s.} \label{Serre F}
	\end{align}
\end{dfn}

Here, the bracket $[X,Y]$ is defined by the standard supersymmetric convention: for homogeneous elements $X$ and $Y$ with respect to the parity, it is given by
\begin{align*}
	[X,Y]=XY-(-1)^{|X||Y|} YX. 
\end{align*} 

\noindent If $a\in \Cx$, define 
\begin{align*}
	[X,Y]_a=XY-(-1)^{|X||Y|} aYX. 
\end{align*} 
The $q$-deformed bracket $\lb{\cdot,\cdot}$ is defined as follows.
Consider the root space decomposition of $\Uqghs$
\begin{align*}
	\Uqghs=\bigoplus_{\alpha \in \widehat \Delta_\s} (\Uqghs)_{\alpha}, 
\end{align*}
where $(\Uqghs)_{\alpha}:=\{X\in \Uqghs\, |\, K_iXK^{-1}_i=q^{(\alpha_i,\alpha)}X,\, i\in \hat I\}$. Then, if $X\in (\Uqghs)_{\alpha}$ and $Y\in (\Uqghs)_{\beta}$, define 
\begin{align*}
	\lb{X,Y}=[X,Y]_{q^{-(\alpha,\beta)}}=XY-(-1)^{|X||Y|} {q^{-(\alpha,\beta)}}YX. 
\end{align*}

The following identities hold: 
\begin{align}
	[[X,Y]_a,Z]_b= [X,[Y,Z]_c]_{abc^{-1}}+(-1)^{|Y||Z|}c[[X,Z]_{bc^{-1}},Y]_{ac^{-1}}, \label{q Jacobi 1} \\
	[X,[Y,Z]_a]_b= [[X,Y]_c,Z]_{abc^{-1}}+(-1)^{|X||Y|}c[Y,[X,Z]_{bc^{-1}}]_{ac^{-1}}. \label{q Jacobi 2} 
\end{align}

If $K_iX=q^{x}XK_i$, then
\begin{align}
	  & [X, K_i Y]=q^{-x}K_i [X,Y]_{q^{x}}, \label{K commutator 1} \\
	  & [K_i Y, X]=K_i [Y,X]_{q^{-x}}. \label{K commutator 2}      
\end{align}

\medskip

We have a central element $K_C$ defined as follows:

If $\g$ is of type $B$, then
\begin{align*}
	K_C := \begin{cases}
	K_0(K_1\cdots K_N)^2,    & \text{if } s_1 = -1; \\
	K_0K_1(K_2\cdots K_N)^2, & \text{if } s_1 = 1.  
	\end{cases}
\end{align*}

If $\g$ is of type $C$ or $D$, then
\begin{align*}
	K_C := \begin{cases}
	K_0(K_1\cdots K_{N-1})^2K_N,    & \text{if } s_1 = -1 \text{ and } s_N=-1; \\
    K_0K_1(K_2\cdots K_{N-1})^2K_N, & \text{if } s_1 = 1 \text{ and } s_N=-1;\\
	K_0(K_1\cdots K_{N-2})^2K_{N-1}K_N, & \text{if } s_1 = -1 \text{ and } s_N=1;\\
	 K_0K_1(K_2\cdots K_{N-2})^2K_{N-1}K_N, & \text{if } s_1 = 1 \text{ and } s_N=1.
	\end{cases}
\end{align*}

\if The superalgebra $\Uq (\ghs)$ has the following structures.

Let $\omega: U \rightarrow U$ be the automorphism of superalgebras defined by
\begin{align*}
	  & \omega(E_i)=F_i,\qquad \omega(F_i)=(-1)^{|i|}E_i,\qquad \omega(K_i)=K_i^{-1}  \quad & (i\in \hat{I}). 
\end{align*}

Let $\pi: U \rightarrow U$ be the anti-automorphism of superalgebras ($\pi(XY)=(-1)^{|X||Y|}\pi(Y)\pi(X)$) defined by
\begin{align*}
	  & \pi(E_i)=E_i,\qquad \pi(F_i)=F_i,\qquad \pi(K_i)=K_i^{-1}  \quad & (i\in \hat{I}). 
\end{align*} 
 
Let $\kappa: U \rightarrow U$ be the automorphism of $\C$-superalgebras defined by
\begin{align*}
	  & \kappa(E_i)=E_i,\qquad\kappa(F_i)=F_i,\qquad \kappa(K_i)=K_i^{-1}, \quad \kappa(q)=q^{-1}  \quad & (i\in \hat{I}). 
\end{align*}
\fi

\section{Formulas with diagrams}\label{sect-formulas}

In this section, we define an action of the affine braid group of type $B$ on the direct sum $\Uq (\gh_\bullet):=\bigoplus_{\s \in \cS} \Uq (\ghs)$ for $\g$ of types $B$, $C$ or $D$. Due to the definition of the ``virtual roots'' $\alpha_{i,\s}^{\dagger}$, the group of type B is used for all types of the algebra.

\subsection{Affine braid group action} \label{subsection braid group}

Let $\cB_N$ be the braid group associated with the Dynkin diagram
\begin{figure}[H]
	\centering
	\begin{tikzpicture}[scale=1.5]
		\draw (-1,0) circle (0.1);
		\draw (0,0) circle (0.1);
		\draw (1,0) circle (0.1);
		\draw (2,0) circle (0.1);
		\draw (3,0) circle (0.1);
		\draw (-0.9,0.05)--(-0.2,0.05);	
		\draw (-0.9,-0.05)--(-0.2,-0.05);	
		\draw[shift={(-0.15,0)}, rotate=45] (-0.15,0)--(0,0);
		\draw[shift={(-0.15,0)}, rotate=-45] (-0.15,0)--(0,0);
		 
		\draw (0.1,0)--(0.9,0);

		\draw[dashed] (1.1,0)--(1.9,0);

		\draw (2.1,0)--(2.9,0);

		\draw (3.1,0.05)--(3.8,0.05);	
		\draw (3.1,-0.05)--(3.8,-0.05);	
		\draw[shift={(3.86,0)}, rotate=45] (-0.15,0)--(0,0);
		\draw[shift={(3.86,0)}, rotate=-45] (-0.15,0)--(0,0);	
			
		
		\draw (4,0) circle (0.1);

		\node [below] at (-1,-0.05) {\Tiny $0$};
		\node [below] at (0,-0.05) {\Tiny $1$};
		\node [below] at (1,-0.05) {\Tiny $2$};
		\node [below] at (2,-0.05) {\Tiny $N-2$};
		\node [below] at (3,-0.05) {\Tiny $N-1$};
		\node [below] at (4,-0.05) {\Tiny $N$};
	\end{tikzpicture}
	\caption{Dynkin diagram of shape $CB_N$ with only even nodes}
	\label{CBN even}
\end{figure}

That is, the group $\cB_N$, $N>1$, is generated by the elements $T_0, T_1, \dots, T_{N-1}, T_N,$ subject to the relations:
\begin{align}
	  & T_{i}T_{i-1}T_{i}T_{i-1}=T_{i-1}T_{i}T_{i-1}T_{i}, &&\text{if } i=1,N, \label{braid relations ends}\\
	  & T_{i}T_{i-1}T_{i}=T_{i-1}T_{i}T_{i-1}, &&\text{if } i=2,\dots, N-1.    \label{braid relations middle}
\end{align}

\noindent Define the following elements of $\cB_N$ 
\begin{align*}
	  & T_{\omega_1}:=T_0 T_1 T_2 T_3\cdots T_N T_{N-1}\cdots T_2 T_1,                        \\
	  & T_{\omega_2}:=T_{1}^{-1}T_{\omega_1}T_{1}^{-1}T_{\omega_1},                           \\
	  & T_{\omega_{i+1}}:=T_i^{-1} T_{\omega_{i}}T_i^{-1}T_{\omega_{i}}T_{\omega_{i-1}}^{-1}, \quad i=2,\dots,N-1. 
\end{align*}
The elements  $T_{\omega_{i}}$, $i\in I$, satisfy the identities below, see \cite{Lusztig1989},
\begin{align*}
		&T_{\omega_i}T_{\omega_j}=T_{\omega_j}T_{\omega_i},                 \\
		& T_{\omega_{N-1}}^2 T_{\omega_N}^{-1}=T_N^{-1} T_{\omega_N}T_N^{-1} .
	\end{align*}

If $N=1$, by an abuse of notation, we denote by $\cB_1$ the free group generated by $T_0$ and $T_1$. In this case, we also define $ T_{\omega_1}:=T_0 T_1$.

The symmetric group $\mathfrak{S}_N$ acts naturally on $\cS$ by permuting indices, $\sigma  \s:=(s_{{\sigma^{-1}}(1)},\dots,s_{{\sigma^{-1}}(N)})$ for all $\sigma\in \mathfrak{S}$, $\s\in \cS$. Let $\sigma_i$ denote the transposition $(i,i+1)$, for all $i=1,\dots, N-1$.

We have a surjective group homomorphism $\pi: \cB_N\to \mathfrak{S}_N$ given by
\begin{align}\label{pi}
        \pi(T_i)=\begin{cases}
                    1, \text{ if } i=0,N;\\
                    \sigma_i, \text{ if } i=1,\dots, N-1.
                \end{cases}
 \end{align}
Note that the abelian subgroup generated by $T_{\omega_{i}}$, $i\in I$, is in the kernel of $\pi$.

In the following subsections, we define isomorphisms $T_{i,\s}:\Uqghs \rightarrow U_q(\gh_{\pi(T_i)\s})$, $i\in \hat{I}$, $\s \in \cS$. These formulas already appeared in \cite{Yamane1999} up to a scaling coefficient. With the coefficients written below, the isomorphisms \eqref{T0 DX even}--\eqref{TN-1 XB} define an action of $\cB_N$ on $\Uq (\gh_\bullet)$.

The formulas of the isomorphisms $T_{i,\s}$ are conveniently described in terms of the Dynkin diagrams corresponding to the parity sequences $\s$ and $\sigma_i\s$. 

\subsection{The operators \texorpdfstring{$T_{0,\s}$}{T0s} and \texorpdfstring{$T_{1,\s}$}{T1s}}

We define the operators $T_{0,\s}:\Uqghs \rightarrow \Uqghs$ and $T_{1,\s}:\Uqghs \rightarrow U_q(\gh_{\sigma_1\s})$.

\subsubsection{Case 1.}
Suppose that the $(0,2)$ section of the Dynkin diagram has the following shape
\begin{figure}[H]
	\centering
	\begin{tikzpicture}[scale=1.5]
		\draw (-.67,0.67) circle (0.1);
		\draw (-.67,-0.67) circle (0.1);
		\draw[rotate=45] (-0.85,0)--(-0.15,0);
		\draw[rotate=45] (0,0.15)--(0,0.85);
		
		        
		       
		\draw[rotate=45] (-0.1,0)--(0.1,0);
		\draw[rotate=-45] (-0.1,0)--(0.1,0);

		\node [left] at (-0.7,0.67) {\Tiny $0$};
		\node [left] at (-0.7,-0.67) {\Tiny $1$};
		\node [below] at (0,-0.05) {\Tiny $2$};
	\end{tikzpicture}
	\caption{$(0,2)$ section of the Dynkin diagrams $DB_N$, $DC_N$, and $DD_N$ if  $s_1=s_2=1$.}
	\label{DX even}
\end{figure}

Define $T_{0,\s}:\Uqghs \rightarrow \Uqghs$ by
\begin{equation}\label{T0 DX even}
	\begin{aligned}
		&  T_{0,\s}(E_0)=E_1,\hspace{1cm}
		  &   & T_{0,\s}(E_1)=E_0, \\
		&  T_{0,\s}(F_0)=F_1,
		  &   & T_{0,\s}(F_1)=F_0, \\
		&  T_{0,\s}(K_0)=K_1,
		  &   & T_{0,\s}(K_1)=K_0, 
	\end{aligned}
\end{equation}
and $T_{0,\s}$ acts trivially on the generators corresponding to all remaining nodes.

Note that the generalized reflection $r_{0,\s}$ is actually a diagram automorphism swapping the nodes $0$ and $1$. Even though the nodes $0$ and $2$ are linked, the operator $T_{0,\s}$ fixes all generators corresponding to the node $2$.

The operator $T_{1,\s}:\Uqghs \rightarrow U_q(\gh_{\sigma_1\s})$ fixes the generators corresponding to the nodes $0,3,\dots,N$, and the action on generators corresponding to the nodes $1$ and $2$ is given by the formulas of type $A$ \eqref{T_i type A 1} and \eqref{T_i type A 2}. Note that, since $s_1=s_2$, we have $\s=\sigma_1\s$.

\subsubsection{Case 2.}\label{s1=-s2=1}
Suppose that the $(0,2)$ section of the Dynkin diagram has the following shape
\begin{figure}[H]
	\centering
	\begin{tikzpicture}[scale=1.5]
		
		\draw[rotate=45] (-0.85,0)--(-0.15,0);
		\draw[rotate=45] (0,0.15)--(0,0.85);
		
		\draw[rotate=45] (-0.85,0)--(-1.05,0);
		\draw[rotate=45] (-0.95,0.1)--(-0.95,-0.1);
		        
		\draw[rotate=-45] (-0.95,0.1)--(-0.95,-0.1);
		\draw[rotate=-45] (-0.85,0)--(-1.05,0);
		
		\draw (-0.62,0.58)--(-0.62,-0.58);
		\draw (-0.72,0.58)--(-0.72,-0.58);
		
		\draw (-.67,0.67) circle (0.1);
		\draw (-.67,-0.67) circle (0.1);

		\draw[rotate=45] (-0.1,0)--(0.1,0);
		\draw[rotate=-45] (-0.1,0)--(0.1,0);

		\node [left] at (-0.7,0.67) {\Tiny $0$};
		\node [left] at (-0.7,-0.67) {\Tiny $1$};
		\node [below] at (0,-0.05) {\Tiny $2$};
	\end{tikzpicture}
	\caption{$(0,2)$ section of the Dynkin diagrams $DB_N$, $DC_N$, and $DD_N$ if  $s_1=-s_2=1$.}
	\label{DX odd}
\end{figure}


In this case, the operator $T_{0,\s}$ has the same formulas as \eqref{T0 DX even} and is trivial on all remaining nodes.

Let $\t=(t_1,\dots,t_N)=\sigma_1 \s$. Note that, since $s_1=-s_2$, the $(0,2)$ section of the diagram with parity $\t$ has the same shape as in Figure \ref{CX odd}.
The operator $T_{1,\s}:\Uqghs \rightarrow \Uqght$ acts on the generators corresponding to the node $0$ by
\begin{equation}\label{T1 DX odd}
	\begin{aligned}
		  & T_{1,\s}(E_0)=-t_1q^{-2t_1}\lb{E_0,E_1},       \\
		  & T_{1,\s}(F_0)=\dfrac{1}{q+q^{-1}}\lb{F_0,F_1}, \\
		  & T_{1,\s}(K_0^{\pm 1})=(K_0K_1)^{\pm 1}.        
	\end{aligned}
\end{equation}

The action of $T_{1,\s}$ on the generators corresponding to the nodes $1$ and $2$ is given by the formulas \eqref{T_i type A 1} and \eqref{T_i type A 2} of type $A$, and $T_{1,\s}$ fixes the generators corresponding to the nodes $3,\dots,N$.

\subsubsection{Case 3.}\label{s1=-s2=-1}
Suppose that the $(0,2)$ section of the Dynkin diagram has the following shape

\begin{figure}[H]
	\centering
	\begin{tikzpicture}[scale=1.5]

		\draw (-1.5,0) circle (0.1);
		\draw (-0.91-.5,0.05)--(-0.15-.5,0.05);	
		\draw (-0.91-.5,-0.05)--(-0.15-.5,-0.05);	
		\draw[shift={(-0.6,0)}, rotate=45] (-0.15,0)--(0,0);
		\draw[shift={(-0.6,0)}, rotate=-45] (-0.15,0)--(0,0);
		\draw (0.01-.5,0) circle (0.1);
		
		\draw[shift={(-0.5,0)},rotate=45] (-0.1,0)--(0.1,0);
		\draw[shift={(-0.5,0)},rotate=-45] (-0.1,0)--(0.1,0);

		\draw[dashed] (-.4,0)--(0.4,0);	
		
		\draw[shift={(0.5,0)},rotate=45] (-0.1,0)--(0.1,0);
		\draw[shift={(0.5,0)},rotate=-45] (-0.1,0)--(0.1,0);
		
		\draw[dashed] (.6,0)--(0.9,0);

		\node [below] at (-1-.5,-0.05) {\Tiny $0$};
		\node [below] at (0-.5,-0.05) {\Tiny $1$};
		\node [below] at (.5,-0.05) {\Tiny $2$};
	\end{tikzpicture}
	\caption{$(0,2)$ section of the Dynkin diagrams $CB_N$, $CC_N$, and $CD_N$ if  $s_1=-s_2=-1$.}
	\label{CX odd}
\end{figure}


In this case, define $T_{0,\s}:\Uqghs \rightarrow \Uqghs$ by
\begin{equation}\label{T0 CX odd}
	\begin{aligned}
		&  T_{0,\s}(E_0)=-s_1(q+q^{-1})F_0K_0,\hspace{1cm}
		  &   & T_{0,\s}(E_1)=\dfrac{s_1 q^{-2s_1}}{q+q^{-1}}\lb{E_1,E_0}, \\
		&  T_{0,\s}(F_0)=\dfrac{-s_1}{q+q^{-1}}K_0^{-1}E_0,
		  &   & T_{0,\s}(F_1)=-\lb{F_1,F_0},                               \\
		&  T_{0,\s}(K_0)=K_0^{-1},
		  &   & T_{0,\s}(K_1)=K_0K_1,                                    
	\end{aligned}
\end{equation}
and the action of $T_{0,\s}$ is trivial on the generators corresponding to all remaining nodes.

Let $\t=\sigma_1 \s$.  Note that the $(0,2)$ section of the diagram with parity $\t$ has the same shape as in Figure \ref{DX odd}. The operator $T_{1,\s}:\Uqghs \rightarrow \Uqght$ acts on the generators corresponding to the node $0$ by the following formulas
\begin{equation}\label{T1 CX odd}
	\begin{aligned}
		  & T_{1,\s}(E_0)=\dfrac{t_1q^{-2t_1}}{q+q^{-1}}\lb{E_0,E_1}, \\
		  & T_{1,\s}(F_0)=\lb{F_0,F_1},                               \\
		  & T_{1,\s}(K_0^{\pm 1})=(K_0K_1)^{\pm 1}.                   
	\end{aligned}
\end{equation}

Analogous to the Case 2, the action of $T_{1,\s}$ on the generators corresponding to the nodes $1$ and $2$ is given by the formulas \eqref{T_i type A 1} and \eqref{T_i type A 2} of type $A$, and $T_{1,\s}$ fixes the generators corresponding to the nodes $3,\dots,N$.

\subsubsection{Case 4.}\label{s1=s2=-1}
Suppose that the $(0,2)$ section of the Dynkin diagram has the following shape

\begin{figure}[H]
	\centering
	\begin{tikzpicture}[scale=1.5]

		\draw (-1.5,0) circle (0.1);
		\draw (-0.91-.5,0.05)--(-0.15-.5,0.05);	
		\draw (-0.91-.5,-0.05)--(-0.15-.5,-0.05);	
		\draw[shift={(-0.6,0)}, rotate=45] (-0.15,0)--(0,0);
		\draw[shift={(-0.6,0)}, rotate=-45] (-0.15,0)--(0,0);
		\draw (0.01-.5,0) circle (0.1);

		\draw[dashed] (-.4,0)--(0.4,0);	
		
		\draw[shift={(0.5,0)},rotate=45] (-0.1,0)--(0.1,0);
		\draw[shift={(0.5,0)},rotate=-45] (-0.1,0)--(0.1,0);
		
		\draw[dashed] (.6,0)--(0.9,0);

		\node [below] at (-1-.5,-0.05) {\Tiny $0$};
		\node [below] at (0-.5,-0.05) {\Tiny $1$};
		\node [below] at (.5,-0.05) {\Tiny $2$};
	\end{tikzpicture}
	\caption{$(0,2)$ section of the Dynkin diagrams $CB_N$, $CC_N$, and $CD_N$ if  $s_1=s_2=-1$.}
	\label{CX even}
\end{figure}


In this case, the operator $T_{0,\s}$ has the same formulas as in \eqref{T0 CX odd} and is trivial on  all remaining nodes.

The operator $T_{1,\s}:\Uqghs \rightarrow \Uqghs$ acts on the generators corresponding to the node $0$ by the following formulas
\begin{equation}\label{T1 CX even}
	\begin{aligned}
		  & T_{1,\s}(E_0)=\dfrac{q^{-2t_1}}{q+q^{-1}}\lb{\lb{E_0,E_1},E_1}, \\
		  & T_{1,\s}(F_0)=\dfrac{1}{q+q^{-1}}\lb{\lb{F_0,F_1},F_1},         \\
		  & T_{1,\s}(K_0^{\pm 1})=(K_0K_1^2)^{\pm 1}.                       
	\end{aligned}
\end{equation}

As before, the action of $T_{1,\s}$ on the generators corresponding to the nodes $1$ and $2$ is given by the formulas \eqref{T_i type A 1} and \eqref{T_i type A 2} of type $A$, and $T_{1,\s}$ fixes the generators corresponding to the nodes $3,\dots,N$.  Note that, since $s_1=s_2$, we have $\s=\sigma_1\s$.

\subsubsection{Case 5.} If $\g$ is of type $B$ and $N=1$, since we assume that $m>0$ and $n\geq 1$, we have that $s_1=-1$ and the Dynkin diagram is given in Figure \ref{CB1}. However, for completion, we also give the formulas for the case $s_1=1$, which corresponds to the twisted affine Lie superalgebra $\sl_3^{(2)}$.

The operator $T_{0,\s}$ has the same formulas as in \eqref{T0 CX odd}. The operator $T_{1,\s}:\Uqghs \rightarrow \Uqghs$ acts on the generators corresponding to the node $0$ by the following formulas
\begin{equation}\label{T1 CB1}
	\begin{aligned}
		  & T_{1,\s}(E_0)=\dfrac{q^{-s_1}(1-(-1)^{|1|}q^{-s_1})}{(q+q^{-1})(1-(-1)^{|1|}q^{-3s_1})}\lb{\lb{\lb{\lb{E_0,E_1},E_1},E_1},E_1}, \\
		  & T_{1,\s}(F_0)=\dfrac{q^{-s_1}(1-(-1)^{|1|}q^{-s_1})}{(q+q^{-1})(1-(-1)^{|1|}q^{-3s_1})}\lb{\lb{\lb{\lb{F_0,F_1},F_1},F_1},F_1},         \\
		  & T_{1,\s}(K_0^{\pm 1})=(K_0K_1^4)^{\pm 1},
	\end{aligned}
\end{equation}
and the action of $T_{1,\s}$ on the generators corresponding to the node $1$ is given by the formulas \eqref{T_i type A 1} of type $A$.

\subsection{The operator \texorpdfstring{$T_{2,\s}$}{T2s}}

Suppose that $N>3$. The action of $T_{2,\s}$ on the generators corresponding to the nodes $1$, $2$, and $3$ is given by the formulas \eqref{T_i type A 1} and \eqref{T_i type A 2} of type $A$, and $T_{2,\s}$ fixes the generators corresponding to the nodes $4,\dots,N$, if $N>4$.

In contrast to the action of $T_{0,\s}$ on the generators corresponding to the node $2$, which is always trivial, the operator $T_{2,\s}$ has a non-trivial action on the generators corresponding to the node $0$ every time that the nodes $0$ and $2$ are linked.

\subsubsection{Case 1.}
Suppose that the $(0,2)$ section of the Dynkin diagram has the same shape as in Figures  \ref{DX even} or \ref{DX odd}. This occurs if  $s_1=1$.

Let $\t=\sigma_2 \s$. In this case, the operator $T_{2,\s}:\Uqghs \rightarrow \Uqght$ acts on the generators corresponding to the node $0$ by the following formulas
\begin{equation}\label{T2 DX}
	\begin{aligned}
		  & T_{2,\s}(E_0)=t_2q^{-t_2}\lb{E_0,E_2},    \\
		  & T_{2,\s}(F_0)=-(-1)^{|0||2|}\lb{F_0,F_2}, \\
		  & T_{2,\s}(K_0^{\pm 1})=(K_0K_2)^{\pm 1}.   
	\end{aligned}
\end{equation}

\subsubsection{Case 2.}
If the $(0,2)$ section of the Dynkin diagram has the same shape as in Figures \ref{CX odd} or \ref{CX even}. The action of $T_{2,\s}$ on the generators corresponding to the node $0$ is trivial.

\subsection{The operators \texorpdfstring{$T_{i,\s}$, $3\leq i\leq N-2$}{Tis}}

Suppose that the $(i-1, i+1)$ section of the Dynkin diagram has a type $A$ like shape. That is,

\begin{figure}[H]
	\centering
	\begin{tikzpicture}[scale=1.5]
		
		\draw[dashed] (-1.6,0)--(-1.9,0);	
		
		\draw[shift={(-0.5,0)},rotate=45] (-0.1,0)--(0.1,0);
		\draw[shift={(-0.5,0)},rotate=-45] (-0.1,0)--(0.1,0);

		\draw (-0.9-.5,0)--(-0.1-.5,0);		
		
		\draw[shift={(-1.5,0)},rotate=45] (-0.1,0)--(0.1,0);
		\draw[shift={(-1.5,0)},rotate=-45] (-0.1,0)--(0.1,0);

		\draw (-.4,0)--(0.4,0);	
		
		\draw[shift={(0.5,0)},rotate=45] (-0.1,0)--(0.1,0);
		\draw[shift={(0.5,0)},rotate=-45] (-0.1,0)--(0.1,0);
		
		\draw[dashed] (.6,0)--(0.9,0);

		\node [below] at (-1-.5,-0.05) {\Tiny $i-1$};
		\node [below] at (0-.5,-0.05) {\Tiny $i$};
		\node [below] at (.5,-0.05) {\Tiny $i+1$};
	\end{tikzpicture}
	\caption{Type A section of the Dynkin diagram.}
	\label{A like}
\end{figure}

Let $\t=\sigma_i \s$. Define $T_{i,\s}:\Uqghs \rightarrow \Uqght$ by

\begin{equation}\label{T_i type A 1}
	\begin{aligned}
		  & T_{i,\s}(E_{i-1})=t_i q^{-t_i}\lb{E_{i-1},E_i},                   &   & T_{i,\s}(E_i)=-t_{i+1} F_i K_i,    \\
		  & T_{i,\s}(F_{i-1})=-(-1)^{|{i-1}||i|}\lb{F_{i-1},F_i},\hspace{1cm} &   & T_{i,\s}(F_i)=-t_i K_i^{-1}E_i,    \\
		  & T_{i,\s}(K_{i-1}^{\pm 1})=(K_{i-1}K_i)^{\pm 1},                   &   & T_{i,\s}(K_i^{\pm 1})=K_i^{\mp 1}, 
	\end{aligned}
\end{equation}

\begin{equation}\label{T_i type A 2}
	\begin{aligned}
		  & T_{i,\s}(E_{i+1})=t_{i+1} (-1)^{|i||{i+1}|}q^{-t_{i+1}} \lb{E_{i+1},E_i}, \\
		  & T_{i,\s}(F_{i+1})=- \lb{F_{i+1},F_i},                                     \\
		  & T_{i,\s}(K_{i+1}^{\pm 1})=(K_iK_{i+1})^{\pm 1}.                           
	\end{aligned}
\end{equation}

\subsection{The operators \texorpdfstring{$T_{N,\s}$}{T Ns} and \texorpdfstring{$T_{N-1,\s}$}{T N-1 s}}

\subsubsection{Case 1.}
Suppose that the $(N-2,N)$ section of the Dynkin diagram has the following shape.
\begin{figure}[H]
    \centering
    \begin{tikzpicture}[scale=1.5]
        \draw[rotate=45] (-0.1,0)--(0.1,0);
	\draw[rotate=-45] (-0.1,0)--(0.1,0);

        \draw[rotate=45] (0.15,0)--(0.85,0);
        \draw[rotate=-45] (0.15,0)--(0.85,0);

        \draw (.67,0.67) circle (0.1);
        \draw (.67,-0.67) circle (0.1);

        \node [right] at (0.7,0.67) {\Tiny $N-1$};
        \node [right] at (0.7,-0.67) {\Tiny $N$};
	\node [below] at (-0.2,-0.05) {\Tiny $N-2$};
\end{tikzpicture}
    \caption{$(N-2,N)$ section of the Dynkin diagrams $CD_N$ and $DD_N$ if  $s_{N-1}=s_N=1$.}
    \label{XD even}
\end{figure}

Define $T_{N,\s}:\Uqgs \rightarrow \Uqgs$ by
\begin{equation}\label{TN XD even}
    \begin{aligned}
  &  T_{N,\s}(E_N)=E_{N-1},\hspace{1cm}
                    &&  T_{N,\s}(E_{N-1})=E_N,\\
  &  T_{N,\s}(F_N)=F_{N-1},
                    &&  T_{N,\s}(F_{N-1})=F_N,\\
  &  T_{N,\s}(K_N)=K_{N-1},
                    &&  T_{N,\s}(K_{N-1})=K_N,
\end{aligned}
\end{equation}
while  $T_{N,\s}$ acts trivially on the generators corresponding to all remaining nodes.

Analogous to the case of $T_{0,\s}$, the generalized reflection $r_{N,\s}$ is actually a diagram automorphism swapping the nodes $N-1$ and $N$. Also, even though the nodes $N-2$ and $N$ are linked, the operator $T_{N,\s}$ fixes all generators corresponding to the node $N-2$.

The operator $T_{N-1,\s}:\Uqghs \rightarrow U_q(\gh_{\sigma_{N-1}\s})$ fixes the generators corresponding to the nodes $0,1,\dots,N-3,N$, and the action on generators corresponding to the nodes $N-2$ and $N-1$ is given by the formulas of type $A$ \eqref{T_i type A 1} and \eqref{T_i type A 2}. Note that, since $s_{N-1}=s_N$, we have $\s=\sigma_{N-1}\s$.

\subsubsection{Case 2.}\label{sN-1=-sN=-1}
Suppose that the $(N-2,N)$ section of the Dynkin diagram has the following shape
\begin{figure}[H]
    \centering
    \begin{tikzpicture}[scale=1.5]
     \draw[rotate=45] (-0.1,0)--(0.1,0);
	\draw[rotate=-45] (-0.1,0)--(0.1,0);

        \draw[rotate=45] (0.15,0)--(0.85,0);
        \draw[rotate=-45] (0.15,0)--(0.85,0);

        \draw (.67,0.67) circle (0.1);
        \draw (.67,-0.67) circle (0.1);

        \node [right] at (0.7,0.67) {\Tiny $N-1$};
        \node [right] at (0.7,-0.67) {\Tiny $N$};
	\node [below] at (-0.2,-0.05) {\Tiny $N-2$};

        \draw[rotate=45] (0.85,0)--(1.05,0);
        \draw[rotate=45] (0.95,0.1)--(0.95,-0.1);
        
        \draw[rotate=-45] (0.95,0.1)--(0.95,-0.1);
	\draw[rotate=-45] (0.85,0)--(1.05,0);

        \draw (0.62,0.58)--(0.62,-0.58);
        \draw (0.72,0.58)--(0.72,-0.58);

\end{tikzpicture}
    \caption{$(N-2,N)$ section of the Dynkin diagrams $CD_N$ and $DD_N$ if  $s_{N-1}=-s_N=-1$.}
    \label{XD odd}
\end{figure}

In this case, the operator $T_{N,\s}$ has the same formulas as \eqref{TN XD even} and is trivial on all remaining nodes.

Let $\t=\sigma_{N-1} \s$. Note that, since $s_{N-1}=-s_N$, the $(N-2,N)$ section of the diagram with parity $\t=\sigma_{N-1} \s$ has the same shape as in Figure \ref{XC odd}.
The operator $T_{N-1,\s}:\Uqgs \rightarrow \Uqgt$ acts on the generators corresponding to the node $N$ by
\begin{equation}\label{TN-1 XD odd}
    \begin{aligned}
  & T_{N-1,\s}(E_N)=\frac{q^2}{q+q^{-1}}\lb{E_N,E_{N-1}},\\
   & T_{N-1,\s}(F_N)=\lb{F_N,F_{N-1}},\\
   & T_{N-1,\s}(K_N^{\pm 1})=(K_NK_{N-1})^{\pm 1}.
\end{aligned}
\end{equation}

\noindent The action of $T_{N-1,\s}$ on the generators corresponding to the nodes $N-2$ and $N-1$ is given by the formulas \eqref{T_i type A 1} and \eqref{T_i type A 2} of type $A$, and $T_{N-1,\s}$ fixes the generators corresponding to the nodes $0,1,\dots,N-3$.

\subsubsection{Case 3.}\label{sN-1=-sN=1}
Suppose that the $(N-2,N)$ section of the Dynkin diagram has the following shape.

\begin{figure}[H]
    \centering
    \begin{tikzpicture}[scale=1.5]

	\draw[dashed] (1.9,0)--(1.6,0);	
	
	\draw[shift={(2,0)}, rotate=45] (-0.1,0)--(0.1,0);
			\draw[shift={(2,0)}, rotate=-45] (-0.1,0)--(0.1,0);
			
	\draw[dashed] (2.1,0)--(2.9,0);	

 \draw (3,0) circle (0.1);
	
	\draw[shift={(3,0)}, rotate=45] (-0.1,0)--(0.1,0);
			\draw[shift={(3,0)}, rotate=-45] (-0.1,0)--(0.1,0);	
		
	\draw (3.15,0.05)--(3.87,0.05);	
		\draw (3.15,-0.05)--(3.87,-0.05);	
		\draw[shift={(3.1,0)}, rotate=45] (0.15,0)--(0,0);
			\draw[shift={(3.1,0)}, rotate=-45] (0.15,0)--(0,0);	
	

        \draw (4,0) circle (0.1);
        
	\node [below] at (2,-0.05) {\Tiny $N-2$};
	\node [below] at (3,-0.05) {\Tiny $N-1$};
	\node [below] at (4,-0.05) {\Tiny $N$};
\end{tikzpicture}
    \caption{$(N-2,N)$ section of the Dynkin diagrams $CC_N$ and $DC_N$ if  $s_{N-1}=-s_N=1$.}
    \label{XC odd}
\end{figure}

In this case, define $T_{N,\s}:\Uqgs \rightarrow \Uqgs$ by
\begin{equation}\label{TN XC odd}
    \begin{aligned}
        &  T_{N,\s}(E_N)=(q+q^{-1})F_NK_N,\hspace{1cm}
                    &&  T_{N,\s}(E_{N-1})=\dfrac{-q^{2}}{q+q^{-1}}\lb{E_{N-1},E_N},\\
  &  T_{N,\s}(F_N)=\dfrac{1}{q+q^{-1}}K_N^{-1}E_N,
                    &&  T_{N,\s}(F_{N-1})=-\lb{F_{N-1},F_N},\\
  &  T_{N,\s}(K_N)=K_N^{-1},
                    &&  T_{N,\s}(K_{N-1})=K_NK_{N-1},
    \end{aligned}
\end{equation}
while  $T_{N,\s}$ acts trivially on the generators corresponding to all remaining nodes.

Let $\t=\sigma_{N-1} \s$. Note that, since $s_{N-1}=-s_N$, the $(N-2,N)$ section of the diagram with parity $\t=\sigma_{N-1} \s$ has the same shape as in Figure \ref{XD odd}.
The operator $T_{N-1,\s}:\Uqgs \rightarrow \Uqgt$ acts on the generators corresponding to the node $N$ by
\begin{equation}\label{TN-1 XC odd}
    \begin{aligned}
  & T_{N-1,\s}(E_N)=-q^{-2}\lb{E_N,E_{N-1}},\\
   & T_{N-1,\s}(F_N)=\dfrac{-1}{q+q^{-1}}\lb{F_N,F_{N-1}},\\
   & T_{N-1,\s}(K_N^{\pm 1})=(K_NK_{N-1})^{\pm 1}.
\end{aligned}
\end{equation}

\noindent The action of $T_{N-1,\s}$ on the generators corresponding to the nodes $N-2$ and $N-1$ is given by the formulas \eqref{T_i type A 1} and \eqref{T_i type A 2} of type $A$, and $T_{N-1,\s}$ fixes the generators corresponding to the nodes $0,1,\dots,N-3$.

\subsubsection{Case 4.}\label{sN-1=sN=-1}
Suppose that the $(N-2,N)$ section of the Dynkin diagram has the following shape

\begin{figure}[H]
    \centering
    \begin{tikzpicture}[scale=1.5]

	\draw[dashed] (1.9,0)--(1.6,0);	
	
	\draw[shift={(2,0)}, rotate=45] (-0.1,0)--(0.1,0);
			\draw[shift={(2,0)}, rotate=-45] (-0.1,0)--(0.1,0);
			
	\draw[dashed] (2.1,0)--(2.9,0);	

 \draw (3,0) circle (0.1);

	\draw (3.15,0.05)--(3.87,0.05);	
		\draw (3.15,-0.05)--(3.87,-0.05);	
		\draw[shift={(3.1,0)}, rotate=45] (0.15,0)--(0,0);
			\draw[shift={(3.1,0)}, rotate=-45] (0.15,0)--(0,0);	
	

        \draw (4,0) circle (0.1);
        
	\node [below] at (2,-0.05) {\Tiny $N-2$};
	\node [below] at (3,-0.05) {\Tiny $N-1$};
	\node [below] at (4,-0.05) {\Tiny $N$};
\end{tikzpicture}
    \caption{$(N-2,N)$ section of the Dynkin diagrams $CC_N$ and $DC_N$ if  $s_{N-1}=s_N=-1$.}
    \label{XC even}
\end{figure}

\noindent In this case, the operator $T_{N,\s}$ has the same formulas as \eqref{TN XC odd} and is trivial on all remaining nodes.
The operator $T_{N-1,\s}:\Uqgs \rightarrow \Uqgs$ acts on the generators corresponding to the node $N$ by
\begin{equation}\label{TN-1 XC even}
    \begin{aligned}
  & T_{{N-1},\s}(E_N)=\dfrac{q^{2}}{q+q^{-1}}\lb{\lb{E_N,E_{N-1}},E_{N-1}},\\
   & T_{{N-1},\s}(F_N)=\dfrac{1}{q+q^{-1}}\lb{\lb{F_N,F_{N-1}},F_{N-1}},\\
   & T_{{N-1},\s}(K_N^{\pm 1})=(K_NK_{N-1}^2)^{\pm 1}.
\end{aligned}
\end{equation}

\noindent As before, the action of $T_{N-1,\s}$ on the generators corresponding to the nodes $N-2$ and $N-1$ is given by the formulas \eqref{T_i type A 1} and \eqref{T_i type A 2} of type $A$, and $T_{N-1,\s}$ fixes the generators corresponding to the nodes $0,1,\dots,N-3$.  Note that, since $s_{N-1}=s_N$, we have $\s=\sigma_{N-1}\s$.

\subsubsection{Case 5.}\label{B end}
Suppose that the $(N-2,N)$ section of the Dynkin diagram has the following shape

\begin{figure}[H]
	\centering
	\begin{tikzpicture}[scale=1.5]

		\draw[dashed] (1.9,0)--(1.6,0);	
			
		\draw[shift={(2,0)}, rotate=45] (-0.1,0)--(0.1,0);
		\draw[shift={(2,0)}, rotate=-45] (-0.1,0)--(0.1,0);
					
		\draw (2.1,0)--(2.9,0);	
			
		\draw[shift={(3,0)}, rotate=45] (-0.1,0)--(0.1,0);
		\draw[shift={(3,0)}, rotate=-45] (-0.1,0)--(0.1,0);	
				
		\draw (3.1,0.05)--(3.8,0.05);	
		\draw (3.1,-0.05)--(3.8,-0.05);	
		\draw[shift={(3.86,0)}, rotate=45] (-0.15,0)--(0,0);
		\draw[shift={(3.86,0)}, rotate=-45] (-0.15,0)--(0,0);	
			
		
		\draw (4,0) circle (0.1);
		\draw[fill=black] (4,0) circle (0.033);

		\node [below] at (2,-0.05) {\Tiny $N-2$};
		\node [below] at (3,-0.05) {\Tiny $N-1$};
		\node [below] at (4,-0.05) {\Tiny $N$};
	\end{tikzpicture}
	\caption{The $(N-2,N)$ section of the Dynkin diagrams $CB_N$ and $DB_N$.}
	\label{XB}
\end{figure}


Define $T_{N,\s}:\Uqghs \rightarrow \Uqghs$ by
\begin{equation}
\begin{aligned}\label{TN XB}
	  & T_{N,\s}(E_{N-1})= q^{-s_N} \lb{\lb{E_{N-1},E_N},E_N}, &   & T_{N,\s}(E_N)=- F_N K_N,               \\
	  & T_{N,\s}(F_{N-1})=\lb{\lb{F_{N-1},F_N},F_N},           &   & T_{N,\s}(F_N)=-(-1)^{|N|} K_N^{-1}E_N, \\
	  & T_{N,\s}(K_{N-1}^{\pm 1})=(K_{N-1}K_N^2)^{\pm 1},      &   & T_{N,\s}(K_N^{\pm 1})=K_N^{\mp 1},     
\end{aligned}
\end{equation}

while $T_{N,\s}$ acts trivially on the generators corresponding to all remaining nodes.

Let $\t=\sigma_{N-1} \s$. Note that the diagram with parity $\t=\sigma_{N-1} \s$ has the same shape as in Figure \ref{XB}, but the parities of the simple roots may have changed.
The operator $T_{N-1,\s}:\Uqghs \rightarrow \Uqght$ acts on the generators corresponding to the node $N$ by
\begin{equation}\label{TN-1 XB}
	\begin{aligned}
		  & T_{N-1,\s}(E_N)=t_N(-1)^{|N-1||N|}q^{-(t_{N-1}+t_N)/2}\lb{E_N,E_{N-1}}, \\
		  & T_{N-1,\s}(F_N)=-q^{-(t_{N-1}-t_N)/2}\lb{F_N,F_{N-1}},                  \\
		  & T_{N-1,\s}(K_N^{\pm 1})=(K_NK_{N-1})^{\pm 1}.                           
	\end{aligned}
\end{equation}

The action of $T_{N-1,\s}$ on the generators corresponding to the nodes $N-2$ and $N-1$ is given by the formulas \eqref{T_i type A 1} and \eqref{T_i type A 2}, and $T_{N-1,\s}$ fixes the generators corresponding to the nodes $0,1,\dots,N-3$.  

\begin{rem}
The isomorphisms $T_{i,\s}$ generate a groupoid if one considers
the category whose objects are the superalgebras $\Uqghs$, $\s\in \cS$, and whose morphisms are  $T_{i,\s}$, $i\in \hat{I}, \s \in \cS$, their inverses, and compositions.

In our situation, the groupoid structure is equivalent to the group action as follows.

Define the following automorphisms of $\Uq (\gh_\bullet)=\bigoplus_{\s\in \cS} \Uqghs$
\begin{align}{\label{autosum}}
	  & T_i=\bigoplus_{\s\in \cS} T_{i,\s} &   & (i\in \hat{I}). 
\end{align}
    
\end{rem}

We adopt the following convention. For any $T\in \cB_N$, we denote by $T_\s$ the restriction of $T$ to the $\Uqghs$ summand in $\Uq (\gh_\bullet)$. For example,  $(T_{i}T_{j}T_{k})_\s= T_{i,\sigma_j\sigma_k\s}T_{j,\sigma_k\s}T_{k,\s}$ maps $\Uqghs$ to $U_q(\gh_{\sigma_i \sigma_j\sigma_k\s})$. 

\medskip

We have the following result.

\begin{thm}\label{thm-braid action}
	The automorphisms $T_{i}$, $i\in \hat{I}$, define an action of the braid group $\cB_N$ on $\Uq (\gh_\bullet)$, i.e., they satisfy the relations \eqref{braid relations ends}--\eqref{braid relations middle}.

\end{thm}
\begin{proof} In {\cite[Prop. 7.4.1]{Yamane1999}}, the operators $T_i$ were defined up to normalization factors and it was shown that they provide isomorphisms. Our formulas contain the explicit coefficients for these isomorphisms to satisfy the braid relations. The relations are proved by direct, case-by-case computation. In most cases, the result is clear after simplifications not involving the Serre relations \eqref{Serre traditional}--\eqref{Serre F}. However, in some cases, the appropriate Serre relations must be used. As the full proof is extensive, we give the details for some of the most complex and non-trivial cases below as an illustration.
	
	Let $\g$ be of type $B$ or $D$ with $N\geq 3$, and let $\s$ be such that $s_1=-s_2=s_3=1$. Then,
	\begin{multline}
		(T_1T_0T_1T_0)_\s(E_2)-(T_0T_1T_0T_1)_\s(E_2)=\\
		=\dfrac{1}{q+q^{-1}}\left(-E_0E_1E_2-(q+q^{-1})E_0E_2E_1 + E_1E_0E_2 + (q+q^{-1}) E_1E_2E_0 -  E_2E_0E_1 +E_2E_1E_0\right),
	\end{multline}
	which equals $0$ by applying the Serre relation \eqref{Serre triangle} with $i=0, j=1$, and $k=2$.

  Consider $\g$ of type $B$ or $D$ with $N\geq 3$ and let $\s$ be such that $s_{1}=s_{2}=-s_3=-1$. Then,
 \begin{multline}
		(T_{1}T_{2}T_1)_\s(E_{0})-(T_2T_{1}T_{2})_\s(E_{0})=\\
		=\dfrac{q^{-3}}{(q+q^{-1})^2}\Big(q^3 E_{0}\left(  E_{1} E_{2}^2 +   E_{2}^2 E_{1} - (q+q^{-1}) E_{2} E_{1} E_{2}  \right) +
            q^3\left(E_{1} E_{2}^2 +   E_{2}^2 E_{1} - (q+q^{-1}) E_{2} E_{1} E_{2}  
                    \right)E_{0}\\
		+q^5 E_{1}\left(  E_{0} E_{2}^2 +   E_{2}^2 E_{0} - (q+q^{-1}) E_{2} E_{0} E_{2}  \right) + 
                q\left(E_{0} E_{2}^2 +   E_{2}^2 E_{0} - (q+q^{-1}) E_{2} E_{0} E_{2}  
                    \right)E_{1}\Big),
	\end{multline}
 which becomes $0$ by applying the Serre relation \eqref{Serre traditional} with $i=2$, $j=0$ and $j=1$. Note that, $E_2$ is even in the image.

Now, consider $\g=\osp(2m+1,2n)$ with $N\geq 3$ and let $\s$ be such that $s_{N-2}=s_{N-1}=s_N=1$. Then,
	\begin{multline}\label{braid check N-1}
		(T_{N-1}T_NT_{N-1}T_N)_\s(E_{N-2})-(T_NT_{N-1}T_NT_{N-1})_\s(E_{N-2})=\\
		=q^{-3}\big(q E_{N-2} E_{N-1}^2 E_{N}^2 - q E_{N-2} E_{N}^2 E_{N-1}^2 - q^2 E_{N-1}^2 E_{N-2} E_{N}^2 + q^2 E_{N}^2 E_{N-1}^2 E_{N-2} +\\ +(q-q^2) E_{N-1} E_{N-2} E_{N-1} E_{N}^2 + (q^2 -q) E_{N}^2 E_{N-1} E_{N-2} E_{N-1} - (1+q) E_{N-2} E_{N-1} E_{N} E_{N-1} E_{N} + \\ + (1+q) E_{N-2} E_{N} E_{N-1} E_{N} E_{N-1} + (q^2 + q^3)E_{N-1} E_{N} E_{N-1} E_{N-2} E_{N} - (q^2 + q^3) E_{N} E_{N-1} E_{N} E_{N-1} E_{N-2} \big).
	\end{multline}
	
	Using the fact that $E_{N-2}$ and $E_{N}$ commute and the Serre relation \eqref{Serre traditional} with $i=N-2$, $j=N-1$, we can rewrite \eqref{braid check N-1} with the $E_{N-2}$ factor either to the left or to the right of each monomial. Then,
	\begin{multline}
		(T_{N-1}T_NT_{N-1}T_N)_\s(E_{N-2})-(T_NT_{N-1}T_NT_{N-1})_\s(E_{N-2})=\\
		=\frac{q^{-3}+q^{-2}}{q+q^{-1}}\Big( E_{N-2}\left( E_{N-1}^2 E_{N}^2 -  E_{N}^2 E_{N-1}^2 - (q+q^{-1}) E_{N-1} E_{N} E_{N-1} E_{N} + (q+q^{-1}) E_{N} E_{N-1} E_{N} E_{N-1}\right) \\
		-q^2\left( E_{N-1}^2 E_{N}^2 -  E_{N}^2 E_{N-1}^2 - (q+q^{-1}) E_{N-1} E_{N} E_{N-1} E_{N} + (q+q^{-1}) E_{N} E_{N-1} E_{N} E_{N-1} \right)E_{N-2} \Big),
	\end{multline}
	which becomes $0$ by applying the Serre relation \eqref{Serre traditional} with $i=N-1, j=N$.

Suppose that $\g$ is of type $D$ with $N\geq 3$, and let $\s$ be such that $s_{N-2}=-s_{N-1}=s_N=1$. Then,
 \begin{multline}
		(T_{N-1}T_NT_{N-1}T_N)_\s(E_{N-2})-(T_NT_{N-1}T_NT_{N-1})_\s(E_{N-2})=\\
		=-E_{N-2}E_{N-1}E_{N} +E_{N-2}E_{N}E_{N-1} -(q+q^{-1})E_{N-1}E_{N-2}E_{N} -E_{N-1}E_{N}E_{N-2} \\+(q+q^{-1})E_{N}E_{N-2}E_{N-1} +E_{N}E_{N-1}E_{N-2}  
	\end{multline}
 which vanishes by applying the Serre relation \eqref{Serre triangle} with $i=N-2, j=N-1$, and $k=N$.
\end{proof}


Recall that the elements $T_{\omega_i}$, $i\in I$, defined in \ref{subsection braid group}, act trivially on the parity sequences via the homomorphism $\pi$. Therefore, we have that $T_{\omega_i,\s}$ are, in fact, automorphisms. Note that $ T_{\omega_i,\s}:(\Uqghs)_{\alpha_j} \rightarrow (\Uqghs)_{\alpha_j-\delta_{i,j}\delta}$. In particular, $T_{\omega_i,\s}(E_j)=E_j$ and $T_{\omega_i,\s}(F_j)=F_j$ for all $i\neq j \in I$.

Note that the explicit form of $T_{\omega_i,s}(E_i)$ and $T_{\omega_i,s}(F_i)$ consists of long nested commutators whose closed form provides little computational insight. Instead, the key properties used in the proof of Lemma \ref{lem:surjectivity} are these mentioned in the above paragraph and the commutation relations of the braid group generators.

\if
\noindent Define the following automorphisms of $\Uqghs$ 
\begin{align*}
	  & T_{\omega_1,\s}:=(T_0 T_1 T_2 T_3\cdots T_N T_{N-1}\cdots T_2 T_1)_\s,                        \\
	  & T_{\omega_2,\s}:=(T_{1}^{-1}T_{\omega_1}T_{1}^{-1}T_{\omega_1})_\s,                           \\
	  & T_{\omega_{i+1},\s}:=(T_i^{-1} T_{\omega_{i}}T_i^{-1}T_{\omega_{i}}T_{\omega_{i-1}}^{-1})_\s. 
\end{align*}
It is straightforward to check that they preserve the parity.

\begin{prop}{\cite{Lusztig1989}} The automorphisms $T_{\omega_{i},\s}$, $i\in I$, $\s \in \cS$, satisfy the following identities
	\begin{align*}
		&(T_{\omega_i}T_{\omega_j})_\s=(T_{\omega_j}T_{\omega_i})_\s,                 \\
		&( T_{\omega_{N-1}}^2 T_{\omega_N}^{-1})_\s=(T_N^{-1} T_{\omega_N}T_N^{-1})_\s .
	\end{align*}
\end{prop}
Note that $ T_{\omega_i,\s}:(\Uqghs)_{\alpha_j} \rightarrow (\Uqghs)_{\alpha_j-\delta_{i,j}\delta}$. In particular, $T_{\omega_i,\s}(E_j)=E_j$ and $T_{\omega_i,\s}(F_j)=F_j$ for all $i\neq j \in I$.

\fi

\section{Loop-like generators}\label{sect-quantum-newD}

Our aim is to provide a realization of a superalgebra $U_q^D(\ghs)$  in terms of current generators, akin to the defining relations found in the {\it new Drinfeld realization} of quantum affine algebras.

First, let us establish the following notations:
\begin{align*}
	                                                                 & X^\pm_i(z):=\sum_{r\in \Z}X^\pm_{i,r}z^{-r}, \qquad K_i^\pm(z):=K_i^{\pm 1}\exp \left(\pm (q-q^{-1})\sum_{t>0}H_{i,\pm t}z^{\mp t}\right), \\
	                                                                 & u_{i,j,r}:=\begin{cases}                               
	\dfrac{[2rA_{N,N}]_q}{r}-\dfrac{[rA_{N,N}]_q}{r}, & \text{if $i=j=N$, $|N|=1$                , and $\gs=\osp(2m+1,2n)$}; \\
    \dfrac{[rA_{i,j}]_q}{r}, & \text{otherwise}.
	\end{cases}
\end{align*}

The definition of $u_{i,j,r}$ in the special case for $i=j=N$, $|N|=1$ , and $\gs=\osp(2m+1,2n)$ is the key to obtain the new relations for type $B$. A similar situation occurs in the case of quantum twisted affine algebras, see \cite{damiani2012}.

Next, we define the following quantum superalgebra.

\begin{dfn}\label{def drinfeld real}
	The quantum superalgebra $U_q^D(\ghs)$ is the unital associative superalgebra generated by {\it current generators} 
	$X^\pm_{i,r}, H_{i,t}$, $K^{\pm 1}_i, C^{\pm 1}$, where $i \in I,$ $r,t\in \Z,$ $ t\neq 0$, subject to the relations below.
		
	The parity of generators is given by $|X^\pm_{i,r}|=|i|$, and all remaining generators have parity $0$. 
		
	The defining relations are as follows:
	\begin{align}
                    & \text{$C$ is central},\quad K_iK_j=K_jK_i. \label{KK rel}\\
                    & K_iX^\pm_j(z)K_i^{-1}=q^{\pm A_{i,j}}X^\pm_j(z). \label{KX rel}\\
                    & [H_{i,t},H_{j,s}]=\delta_{t+s,0}\, u_{i,j,t}\frac{C^t-C^{-t}}{q-q^{-1}}. \label{HH rel}\\
                    & [H_{i,t},X^{\pm}_j(z)]=\pm  u_{i,j,t}C^{-(t\pm|t|)/2}z^t X^\pm_j(z). \label{HX rel}\\
                    & [X^+_i(z),X^-_j(w)]=\frac{\delta_{i,j}}{q-q^{-1}}\left(\delta\left(C\frac{w}{z}\right)K_i^+(w)-\delta\left(C\frac{z}{w}\right)K_i^-(z)\right). \label{XX rel}\\[.5cm]
                    & \hspace{-0.5cm}\text{If $A_{i,j}\neq 0$ (except when $i=j=N$, $|N|=1$, and $\g=\osp(2m+1,2n)$), then} \notag\\
                    & \left(z-q^{\pm A_{i,j}}w\right)X^\pm_i(z)X^\pm_j(w)+(-1)^{|i||j|}\left(w-q^{\pm A_{i,j}}z\right)X^\pm_j(w)X^\pm_i(z)=0. \label{s1 rel}  \\[.5cm]
                    & \hspace{-.5cm} \text{If $A_{i,j}=0$, then }\notag \\ &[X^\pm_i(z),X^\pm_j(w)]=0.  \label{XX i not j} \\[.5cm]
		          & \hspace{-.5cm}\text{If } A_{i,i}\neq 0 \text{ and } k=1-\frac{2A_{i,j}}{A_{i,i}}, \text{then } \notag\\
                    &\Sym_{{z_1,z_2,\dots,z_k}}\lb{X^\pm_i(z_1),\lb{X^\pm_i(z_2),\lb{\cdots,\lb{X^\pm_i(z_k),X^\pm_{j}(w)}}}}=0. \label{Serre1} \\[.5cm]
                    &\hspace{-.5cm}\text{If } A_{j,j}=0 \text{ and } A_{i,j}=-A_{j,k}\neq 0, \text{ i.e., if }   
		\begin{tikzpicture}[baseline=-3pt]
		\draw[rotate=45] (-0.1,0)--(0.1,0);
		\draw[rotate=-45] (-0.1,0)--(0.1,0);
		\draw (1,0) circle (0.1);
		\draw[shift={(1,0)},rotate=45] (-0.1,0)--(0.1,0);
		\draw[shift={(1,0)},rotate=-45] (-0.1,0)--(0.1,0);
		\draw (0.1,0)--(0.9,0);
		\draw (1.1,0)--(1.9,0);
		\draw[shift={(2,0)},rotate=45] (-0.1,0)--(0.1,0);
		\draw[shift={(2,0)},rotate=-45] (-0.1,0)--(0.1,0);
		\node [below] at (0,0) {\Tiny $i$};
		\node [below] at (1,0) {\Tiny $j$};
		\node [below] at (2,0) {\Tiny $k$};
		\end{tikzpicture},  \text{ or }  \begin{tikzpicture}[baseline=-3pt]
		\draw[rotate=45] (-0.1,0)--(0.1,0);
		\draw[rotate=-45] (-0.1,0)--(0.1,0);
		\draw (1,0) circle (0.1);
		\draw[shift={(1,0)},rotate=45] (-0.1,0)--(0.1,0);
		\draw[shift={(1,0)},rotate=-45] (-0.1,0)--(0.1,0);
		\draw (0.1,0)--(0.9,0);
		\draw (1.1,0.05)--(1.81,0.05);
		\draw (1.1,-0.05)--(1.81,-0.05);
		\draw[shift={(1.86,0)}, rotate=45] (-0.15,0)--(0,0);
		\draw[shift={(1.86,0)}, rotate=-45] (-0.15,0)--(0,0);
		\draw (2,0) circle (0.1);
		\draw[fill=black] (2,0) circle (0.033);
		\node [below] at (0,0) {\Tiny $i$};
		\node [below] at (1,0) {\Tiny $j$};
		\node [below] at (2,0) {\Tiny $k$};
		\end{tikzpicture}, {then }\notag \\
                    & \Sym_{{z_1,z_2}}\lb{X^\pm_j(z_1),\lb{X^\pm_{i}(w_1),\lb{X^\pm_j(z_2),X^\pm_{k}(w_2)}}}=0.\label{Serre odd}\\[.5cm]
                    &  \hspace{-.5cm}\text{If $\g=\osp(2m+1,2n)$ and $|N|=1$, then }\notag\\
&\Sym_{{z_1,z_2,z_3}}z_3^{\pm1}\lb{X^\pm_N(z_1),\lb{X^\pm_N(z_2),X^\pm_N(z_3)}}=0,\label{cubic relation} \\
                            & \Sym_{{z_1,z_2}} \left(z_1^{\pm2}[X^\pm_N(z_1),X^\pm_N(z_2)]_{q^{A_{N,N}}}   - q^{2A_{N,N}}(z_1z_2)^{\pm1}[X^\pm_N(z_1),X^\pm_N(z_2)]_{q^{-3A_{N,N}}} \right)=0, and\
                           \label{quadratic N}\\
&\begin{aligned}[t]  		                                  &\Sym_{{z_1,z_2}} z_1^{\pm 1}\left(q^{A_{N,N}}[[X^\pm_N(z_1),X^\pm_N(z_2)]_{q^{A_{N,N}}},X^\pm_{N-1}(w)]_{q^{-2A_{N,N}}}  \right.  \\
                                            &\hphantom{dddddddddd}\left.  + (q+q^{-1})[[X^\pm_{N-1}(w),X^\pm_N(z_1)]_{q^{A_{N,N}}},X^\pm_{N}(z_2)]\right) =0.
		                \end{aligned}\label{cubic N N-1}\\[.5cm]                     
                  & \hspace{-0.5cm}\text{If } \begin{tikzpicture}[baseline=-3pt]
		\draw (0,0) circle (0.1);
		\draw[rotate=45] (-0.1,0)--(0.1,0);
		\draw[rotate=-45] (-0.1,0)--(0.1,0);
		\draw (1,0) circle (0.1);
		\draw[shift={(1,0)},rotate=45] (-0.1,0)--(0.1,0);
		\draw[shift={(1,0)},rotate=-45] (-0.1,0)--(0.1,0);
		\draw (0.1,0)--(0.9,0);
		\draw (1.2,0.05)--(1.9,0.05);
		\draw (1.2,-0.05)--(1.9,-0.05);
		\draw[shift={(1.15,0)}, rotate=45] (0.15,0)--(0,0);
		\draw[shift={(1.15,0)}, rotate=-45] (0.15,0)--(0,0);
		\draw (2,0) circle (0.1);
		\node [below] at (0,0) {\Tiny $i$};
		\node [below] at (1,0) {\Tiny $j$};
		\node [below] at (2,0) {\Tiny $k$};
		\end{tikzpicture}, \text{then }\notag\\	
        & \Sym_{{z_1,z_2}}\Sym_{{w_1,w_2,w_3}}\Big( \lb{\lb{X^\pm_{i}(z_1),X^\pm_{j}(w_1)},\lb{\lb{X^\pm_{i}(z_2),X^\pm_{j}(w_2)},\lb{X^\pm_{j}(w_3),X^\pm_{k}(y)}}}\Big)=0. \label{Serre type D odd}\\[.5cm]
        &\hspace{-.5cm}  \text{If } \begin{tikzpicture}[baseline=-3pt]
		\draw[shift={(-1,0)},rotate=45] (-0.1,0)--(0.1,0);
		\draw[shift={(-1,0)},rotate=-45] (-0.1,0)--(0.1,0);
		\draw (-0.1,0)--(-0.9,0);
		\draw (0,0) circle (0.1);
		\draw (1,0) circle (0.1);
		\draw[shift={(1,0)},rotate=45] (-0.1,0)--(0.1,0);
		\draw[shift={(1,0)},rotate=-45] (-0.1,0)--(0.1,0);
		\draw (0.1,0)--(0.9,0);
		\draw (1.2,0.05)--(1.9,0.05);
		\draw (1.2,-0.05)--(1.9,-0.05);
		\draw[shift={(1.15,0)}, rotate=45] (0.15,0)--(0,0);
		\draw[shift={(1.15,0)}, rotate=-45] (0.15,0)--(0,0);
		\draw (2,0) circle (0.1);
		\node [below] at (-1,0) {\Tiny $i$};
		\node [below] at (0,0) {\Tiny $j$};
		\node [below] at (1,0) {\Tiny $k$};
		\node [below] at (2,0) {\Tiny $l$};
		\end{tikzpicture}, \text{then }\notag \\
        &\begin{aligned}[t]
                &\Sym_{{z_1,z_2}}\Sym_{{w_1,w_2,w_3}}\Big(\\
                & \hspace{0.5cm} \lb{\lb{\lb{\lb{\lb{\lb{X^\pm_{i}(y_1),X^\pm_{j}(z_1)},X^\pm_{k}(w_1)},X^\pm_{l}(y_2)},X^\pm_{k}(w_2)},X^\pm_{j}(z_2)},X^\pm_{k}(w_3)}\Big)=0.\label{Serre type D even}
            \end{aligned}\\[.5cm]
        &\hspace{-.5cm} \text{If } A_{i,j} A_{i,k} A_{j,k}\neq 0,\; A_{i,j} + A_{i,k} + A_{j,k}= 0,\; \text{and }\, |i||j|+|j||k|+|i||k|\equiv 1, \text{ then } \notag	\\
            &(-1)^{|i||k|}[(\alpha_i,\alpha_k)]\lb{\lb{X^\pm_{i}(z),X^\pm_{j}(w)},X^\pm_{k}(y)}=(-1)^{|i||j|}[(\alpha_i,\alpha_j)]\lb{\lb{X^\pm_{i}(z),X^\pm_{k}(y)},X^\pm_{j}(w)}.\label{Serre triangle currents}
	\end{align}
\end{dfn}

\begin{rem}
    The relation \eqref{cubic N N-1} is a consequence of the remaining relations.
\end{rem}

\medskip
\begin{prop}\label{automorphisms} The superalgebra $U_q^D(\ghs)$ has the following symmetries.
	\begin{enumerate}  
		\item The map $\eta_\s: U_q^D(\ghs) \rightarrow U_q^D(\ghs)$ is an anti-automorphism of superalgebras defined by:
		      \begin{align*}
		      	  & \eta_\s(C)=C,\qquad \eta_\s(X^\pm_i(z))=X^\pm_i(z^{-1}),\qquad \eta_\s(K_i^\pm(z))=K_i^\mp(C^{-1}z^{-1}),  \quad  i\in {I}. 
		      \end{align*} 
		       
		\item The map $\Omega_\s: U_q^D(\ghs) \rightarrow U_q^D(\ghs)$ is an automorphism defined by:
		      \begin{align*}
		      	  & \Omega_\s(C)=C,\quad  \Omega_\s(X^\pm_i(z))=(-1)^{|i|}X^\mp_i(z), \quad \Omega_\s(K_i^\pm(z))=K_i^\mp(z),  \quad i\in {I}. 
		      \end{align*} 
		      
		\item For each $i\in I$, the automorphism $\tau_{i,\s}: U_q^D(\ghs) \rightarrow U_q^D(\ghs)$ of superalgebras is defined as:
		      \begin{align*}
		      	  & \tau_{i,\s}(C)=C,\qquad \tau_{i,\s}(X^\pm_{j,r})=(-1)^{i\delta_{i,j}}X^\pm_{j,r\mp \delta_{i,j}}, \qquad \tau_{i,\s}(H_{j,r})=H_{j,r}, \qquad  \tau_{i,\s}(K_j)=C^{-\delta_{i,j}}K_j \quad i\in {I}. 
		      \end{align*}  
	\end{enumerate}
\end{prop}

\medskip
\noindent Our main result is the following theorem.

\begin{thm}{\label{thm_main}}
	There exists a surjective homomorphism $\psi_\s: U_q^D(\ghs) \rightarrow \Uqghs$ defined as follows:
	\begin{align*}
		                         & \psi_\s(X^+_{i,r})=(-1)^{ir}(T_{\omega_i}^{-r})_\s(E_i), \quad \psi_\s(X^-_{i,r})=(-1)^{ir}(T_{\omega_i}^{r})_\s(F_i),\\ &\psi_\s(K_i)=K_i,\quad \psi_\s(C)=K_C \quad i\in I, r\in \Z.
	\end{align*}
\end{thm}

Our strategy to prove Theorem \ref{thm_main} is as follows. First, we consider a larger superalgebra $\Tilde{U}_q^D(\ghs)$ with the same set of generators as $U_q^D(\ghs)$, but with fewer relations. It is easier to establish that the analogous homomorphism $\Tilde{\psi}_\s: \Tilde{U}_q^D(\ghs)\rightarrow \Uqghs$ is surjective. Then, we demonstrate that the ideal $\cI_\s$ of $\Tilde{U}_q^D(\ghs)$, such that $U_q^D(\ghs)\simeq \Tilde{U}_q^D(\ghs)/\cI_\s$, lies in the kernel of $\Tilde{\psi}_\s$.

\begin{dfn}
	The quantum superalgebra $\Tilde{U}_q^D(\ghs)$ is the unital associative superalgebra generated by $X^\pm_{i,r}, H_{i,t}$, $K^{\pm 1}_i, C^{\pm 1}$, where $i \in I,\; r,t\in \Z, t\neq 0$, subject to the relations below.
		
	The parity of the generators is given by $|X^\pm_{i,r}|=|i|$, and all remaining generators have parity $0$. 
	
	The defining relations are as follows:
	\begin{align}
		  & \text{$C$ is central},\quad K_iK_j=K_jK_i. \label{KK Util}                                                                                                      \\  
		  & K_iX^\pm_j(z)K_i^{-1}=q^{\pm A_{i,j}}X^\pm_j(z). \label{KX Util}                                                                                           \\
		  & [H_{i,t},X^{\pm}_j(z)]=\pm u_{i,j,t}C^{-(t\pm |t|)/2}z^t X^\pm_j(z), \quad \text{ for } t\in\{1,-1\}. \label{HX Util}                                           \\
		  & [X^+_i(z),X^-_j(w)]=\frac{\delta_{i,j}}{q-q^{-1}}\left(\delta\left(C\frac{w}{z}\right)K_i^+(w)-\delta\left(C\frac{z}{w}\right)K_i^-(z)\right) \label{XXK Util}. 
	\end{align}
\end{dfn}

Let $\Tilde\eta_\s, \Tilde\Omega_\s$, and $\Tilde\tau_{i,\s}$ be the endomorphisms of $\Tilde{U}_q^D(\ghs)$ defined by the same formulas as the maps $\eta_\s, \Omega_\s$, and $\tau_{i,\s}$ (see Proposition \ref{automorphisms}), respectively.

\begin{lem}\label{lem:surjectivity}
	There exists a surjective homomorphism $\Tilde{\psi}_\s: \Tilde{U}_q^D(\ghs)\rightarrow \Uqghs$ given by
	\begin{align*}
		                         & \Tilde{\psi}_\s(X^+_{i,r})=(-1)^{ir}(T_{\omega_i}^{-r})_\s(E_i), \quad \Tilde{\psi}_\s(X^-_{i,r})=(-1)^{ir}(T_{\omega_i}^{r})_\s(F_i), \\ &\Tilde{\psi}_\s(K_i)=K_i,\quad \Tilde{\psi}_\s(C)=K_C,\quad i\in I, r\in \Z.
	\end{align*}
\end{lem}
\begin{proof} 
	We check that $\Tilde{\psi}_\s$ is a homomorphism. 
	
	The relations \eqref{KK Util} follow from the relations \eqref{DJ kk} and the fact that $A_{i,0}+2\sum_{j=1}^N A_{i,j}=0$ when $s_1=-1$, and $A_{i,0}+A_{i,1}+2\sum_{j=2}^N A_{i,j}=0$ when $s_1=1$.
	
	The relations \eqref{KX Util} follow from the relations \eqref{DJ KE} and $T_{\omega_i,\s}(K_j)=C^{-\delta_{i,j}}K_j$.
    
The relations \eqref{HX Util}, and the relations \eqref{XXK Util}  are checked as in \cite{damiani2012, Yamane1999} with appropriate changes in the notation regarding the parities, except for the base case of the relation \eqref{HX Util} with $i=j$, which we establish directly below, and the case when $\gs$ is of type $B$,  $i=j= N$ and $|N|=1$, which we prove separately. The compatibility of $\Tilde{\psi}_\s$ with the equation \eqref{XXK Util} is then proved by induction, as in \cite[Lemma 5.8]{Heckenberger2007}.
	We first verify the relation \eqref{HX Util} with $i=j$, $(\alpha_i,\alpha_i)\neq 0$ and $|i|=0$, directly. For this, note that in type $A$ the operator $T_{\omega_1,\s}$ acts as $T_{\omega_1,\s}=\tau\, T_{1,\s}$, where $\tau$ is the diagram rotation $i\mapsto i+1$, so that $T_{\omega_1,\s}(F_1)=-K_0^{-1}E_0$. A direct computation then yields
	\begin{align}\label{HX Util base even}
		K_1^{-1}[[E_1,T_{\omega_1,\s}(F_1)],F_1]_{q^{-A_{1,1}}}=-[A_{1,1}]\,T_{\omega_1,\s}(F_1).
	\end{align}
	Combining \eqref{HX Util base even} with \eqref{K commutator 2}, we obtain, exactly as for the odd non-isotropic node below,
	\begin{align*}
		\Tilde{\psi}_\s([H_{i,1},X^-_{i,0}])
		&=\Tilde{\psi}_\s(K_i^{-1}[[X^+_{i,0},X^-_{i,1}],X^-_{i,0}]_{q^{-A_{i,i}}})\\
		&=(-1)^i K_i^{-1}[[E_i,T_{\omega_i,\s}(F_i)],F_i]_{q^{-A_{i,i}}}
		=-[A_{i,i}]\,\Tilde{\psi}_\s(X^-_{i,1}).
	\end{align*}

	We now check the relations \eqref{HX Util} with $X^-_{j,r}$ and $i\neq j$. The case of $X^+_{j,r}$ is analogous. Since $T_{\omega_i,\s} \Tilde{\psi}_\s=\Tilde{\psi}_\s \Tilde{\tau}_{i,\s}$, it is sufficient to consider $r=0$.
 Note that by the definition of $T_{\omega_i,\s}$, we have:
	\begin{align*}
		T_{\omega_i,\s}(F_iF_j-(-1)^{|i||j|}q^{-A_{i,j}}F_jF_i)=(-1)^{|i||j|}T_{\omega_j,\s}(F_jF_i-(-1)^{|i||j|}q^{-A_{i,j}}F_iF_j). 
	\end{align*}
	
	\noindent In addition, the element $H_{i,1}$ can be written as $H_{i,1}=K_i^{-1}[X^+_{i,0}, X^-_{i,1}]$. 
	
	\noindent Applying $\Tilde{\psi}_\s$ and subsequently \eqref{K commutator 2} and \eqref{q Jacobi 1} we get
	\begin{align*}
		\Tilde{\psi}_\s([H_{i,1},X^-_{j,0}]) & =\Tilde{\psi}_\s([K_i^{-1}[X^+_{i,0}, X^-_{i,1}],X^-_{j,0}])=(-1)^i [K_i^{-1}[E_{i}, T_{\omega_i,\s}(F_{i})],F_j]                                    \\
		                                     & =(-1)^i K_i^{-1}[E_i, T_{\omega_i,\s}(F_iF_j-(-1)^{|i||j|}q^{-A_{i,j}}F_jF_i)]           \\
		                                     & =(-1)^{i+|i||j|} K_i^{-1} [E_i, T_{\omega_j,\s}(F_jF_i-(-1)^{|i||j|}q^{-A_{i,j}}F_iF_j)] \\
		                                     & =(-1)^{i+|i||j|} T_{\omega_j,\s}(K_i^{-1} [E_i, F_jF_i-(-1)^{|i||j|}q^{-A_{i,j}}F_iF_j]) \\
		                                     & =(-1)^{i} [A_{i,j}]T_{\omega_j,\s}(F_j)           =-[A_{i,j}]\Tilde{\psi}_\s(X^-_{j,1}).                                                   
	\end{align*}

	It remains to check the relation \eqref{HX Util}  when $\gs$ is of type $B$,  $i=j= N$ and $|N|=1$.	For this, note that in $U_q(\osp(1,2)^{(1)})$, we have $T_{\omega_1}(F_1)=T_0 T_1(F_1)$. Thus, by a direct computation:
	\begin{align*}
		K_1^{-1} [[E_1,T_{\omega_1}(F_1)], F_1]_{q^{-A_{1,1}}}
		  & =-u_{1,1,1}\frac{q^{-A_{1,1}} }{q+q^{-1}}[E_0, E_1]_{q^{A_{0,1}}} K_1^{-1} K_0^{-1} \\
		  & =-u_{1,1,1} T_{\omega_1}(F_1).                                                      
	\end{align*}
	
	\noindent Combining it with \eqref{K commutator 2} we obtain
	\begin{align*}
		\Tilde{\psi}_\s( [H_{N,1},X^-_{N,0}])
		  & = \Tilde{\psi}_\s( K_N^{-1}[[X^+_{N,0},X^-_{N,1}],X^-          _{N,0}]_{q^{-A_{N,N}}})\\
		  & =(-1)^N K_N^{-1}[[E_N,T_{\omega_N}(F_N)], F_N]_{q^{-A_{N,N}}}                           \\
		  & =-(-1)^N u_{N,N,1} T_{\omega_N}(F_N) =-u_{N,N,1}\Tilde{\psi}_\s(X^-_{N,1}).                                                  
	\end{align*}


	It is clear that $E_i, F_i, K_i^{\pm1}\in \Im{\Tilde{\psi}_\s}$ for $i\in I$. Also, we have that $(K_1\cdots K_N)^{-2}\Tilde{\psi}_\s(C)=K_0\in \Im{\Tilde{\psi}_\s}$ when $s_1=-1$, and $K^{-1}_1(K_2\cdots K_N)^{-2}\Tilde{\psi}_\s(C)=K_0\in \Im{\Tilde{\psi}_\s}$ when $s_1=1$.
Since $\Im{\Tilde{\psi}_\s}$ is $T_{\omega_i,\s}$-stable, and $(T_1T_2\cdots T_{N-1}T_N T_{N-1}\cdots T_2T_1)_{\s}^{-1}(F_1)\in \Im{\Tilde{\psi}_\s}$, we obtain $T_{0,\s}(F_1)\in \Im{\Tilde{\psi}_\s}$.
	
\noindent	If $s_1=1$, then $T_{0,\s}(F_1)=F_0$, and we are done.
	
\noindent	If $s_1=-1$, then $T_{0,\s}(F_1)=-\lb{F_1,F_0}$ and $[E_1,\lb{F_1,F_0}]K_1=(q+q^{-1})F_0 \in \Im{\Tilde{\psi}_\s} $.

\noindent The proof that $E_0 \in \Im{\Tilde{\psi}_\s}$ is analogous.

\end{proof}

To simplify the notation, define the following elements of $\Tilde{U}_q^D(\ghs)$:
\begin{align*}
	  & Sd_{i,j}^\pm(r,s)=[X^\pm_{i,r},X^\pm_{j,s}],                                          \\
	  & Se_{i,j}^\pm(r,s)=\lb{X^\pm_{i,r},\lb{\cdots\lb{X^\pm_{i,r},X^\pm_{j,s}}}}, \text{ where } X^\pm_{i,r} \text{ appears } 1-\frac{2A_{i,j}}{A_{i,i}} \,\text{times}.                    \\
	  & So_i^\pm(r,s,t)=\lb{X^\pm_{i,r},\lb{X^\pm_{i+1,s},\lb{X^\pm_{i,r},X^\pm_{i-1,t}}}}, \\
	  & S_N^\pm(r,s)=\lb{X^\pm_{N,r},\lb{X^\pm_{N,r},\lb{X^\pm_{N,r},X^\pm_{N-1,s}}}}, \\
        & S_{Do}^\pm(r,s,t)=\lb{\lb{X^\pm_{N-2,r},X^\pm_{N-1,s}},\lb{\lb{X^\pm_{N-2,r},X^\pm_{N-1,s}},\lb{X^\pm_{N-1,s},X^\pm_{N,t}}}} \\
        & S_{De}^\pm(r,s,t,p)=\lb{\lb{\lb{\lb{\lb{\lb{X^\pm_{N-3,r},X^\pm_{n-2,s}},X^\pm_{N-1,t}},X^\pm_{N,p}},X^\pm_{N-1,t}},X^\pm_{N-2,s}},X^\pm_{N-1,t}} \\
   & S_{T}^\pm(r,s,t)=\lb{\lb{X^\pm_{N-2,r},X^\pm_{N-1,s}},X^\pm_{N,t}}-\lb{\lb{X^\pm_{N-2,r},X^\pm_{N,t}},X^\pm_{N-1,s}}
\end{align*}
for all $r,s, t,p\in \Z$, and $i,j,k,l\in I$.

\noindent Let $\cI_\s$ be the two-sided ideal of $\Tilde{U}_q^D(\ghs)$ generated by
\begin{align*}
	  & Sd_{i,j}^\pm(r,s),\, \text{for all $r,s\in \Z$, and $i,j\in I$, such that $A_{i,j}= 0$,}                               \\ 
	  & Se_{i,j}^\pm(r,s),\, \text{for all $r,s\in \Z$, and $i,j\in I$, $i\neq N$, such that $A_{i,j}A_{i,i}\neq 0$,} \\ 
	  & So_i^\pm(r,s,t),\, \text{for all $r,s, t\in \Z$, and $i\in I$, $i\neq 1, N$, such that $A_{i,i}= 0$},                 \\ 
	  & S_N^\pm(r,s),\, \text{for all $r,s\in \Z$},                                       \\ 
        & S_{Do}^\pm(r,s,t), \,\text{for all $r,s,t\in \Z$}, \text{if the diagram present in \eqref{Serre type D odd} occurs},\\
        & S_{De}^\pm(r,s,t,p),  \,\text{for all $r,s,t,p\in \Z$}, \text{if the diagram present in \eqref{Serre type D even} occurs},\\
	  & S_{T}^\pm(r,s),\, \text{for all $r,s\in \Z$},  \text{if } A_{i,j} A_{i,k} A_{j,k}\neq 0,\; A_{i,j} + A_{i,k} + A_{j,k}= 0,\; \text{and }\, |i||j|+|j||k|+|i||k|\equiv 1.                           
\end{align*}

\begin{lem}
	$\Tilde{\psi}_\s(\cI_\s)=0$.	
	\begin{proof}
		Recall that $T_{\omega_i,\s}(E_j)=E_j$ and $T_{\omega_i,\s}(F_j)=F_j$. Let us show that $Sd_{i,j}^+(r,s)\in \Ker \Tilde{\psi}_\s$.  First, suppose $i\neq j$ and $A_{i,j}^\s=0$. Then, using the relation \eqref{EE rel}, we have 
		\begin{align*}
			\Tilde{\psi}_\s(Sd_{i,j}^+(r,s))=\Tilde{\psi}_\s([X^+_{i,r},X^+_{j,s}])=(-1)^{ri+sj}(T_{\omega_i,\s}^{-r}T_{\omega_j,\s}^{-s})([E_i,E_j])=0.
		\end{align*}
  Now, suppose that $i=j$ and $A_{i,i}^\s=0$. This implies that $|i|=1$ and $2E_i^2=[E_i,E_i]=0$. Then, the previous computation with $r=s$ shows that $Sd_{i,i}^+(r,r)\in \Ker \Tilde{\psi}_\s$. To conclude that $Sd_{i,i}^+(r,s)\in \Ker \Tilde{\psi}_\s$, for all $r,s\in \Z$, we commute $Sd_{i,i}^+(r,r)$ with $H_{i\pm 1,1}$ and use induction.

		The proof that all other generators of $\cI_\s$ lie in the kernel of $\Tilde{\psi}_\s$ is similar. To show that $Se_{i,j}^\pm(r,s)$ and $S_N^\pm(r,s)$ are in $\Ker \Tilde{\psi}_\s$,  we use the relation \eqref{Serre traditional}. For $So_i^\pm(r,s,t)$, we use \eqref{Serre odd node}. For $S_{Do}^\pm(r,s,t)$, we use \eqref{Serre DJ Sdo}.  For $S_{De}^\pm(r,s,t,p)$, we use \eqref{Serre DJ Sde}, while for $S_{T}^\pm(r,s)$, we use \eqref{Serre triangle}.
		           
	\end{proof}  
\end{lem}
 Note that $\cI_\s$ is invariant under the action of $\Tilde\eta,\, \Tilde\Omega$, and $\Tilde\tau_i$. 

\begin{lem}\label{lem_quo}
	$U_q^D(\ghs)\simeq \Tilde{U}_q^D(\ghs)/\cI_\s$.
	
	\begin{proof}
		The proof follows the same strategy used in \cite{damiani2012}. 
		It is clear that all generators of $\cI_\s$ vanish in $U_q^D(\ghs)$, in particular the relations \eqref{XX i not j} belong to $\cI_\s$. Moreover, relations \eqref{KK rel}, \eqref{KX rel}, and \eqref{XX rel} coincide with those of \eqref{KK Util}, \eqref{KX Util}, and \eqref{XXK Util}, respectively.

  We will show that the remaining  relations, in the definition of $U_q^D(\ghs)$, can be deduced from the defining relations of $\Tilde{U}_q^D(\ghs)$ and $\cI_\s$ in the following order. We first check \eqref{Serre1} and \eqref{Serre odd}, after that we show \eqref{s1 rel} and   \eqref{cubic relation}. Next, we use  \eqref{cubic relation} to show \eqref{quadratic N}. After that, employing 
  \eqref{s1 rel} and \eqref{quadratic N}, we obtain the relations \eqref{HX rel} and then use it to obtain \eqref{HH rel}. We finish showing 
  the relation \eqref{cubic N N-1}.

   The relations \eqref{Serre1} and \eqref{Serre odd} are obtained by induction, after commuting the generators $Se_{i,j}^\pm(r,s)$, $So_i^\pm(r,s,t)$, and $S_N^\pm(r,s)$ with $H_{i,1}$, for appropriate choices of $i\in I$, and using the fact that $\cI_\s$ is invariant under the action of $\Tilde\tau_i$ for all $i\in I$.

	 Next, we show that relations \eqref{s1 rel} hold in $\Tilde{U}_q^D(\ghs)/\cI_\s$. The proof splits into five cases.

 \noindent Case $1$. Suppose $|i|=0$, $j=i\pm 1$ and, if $\g_\s$ is of type B, $i\neq N$. Then, in $\Tilde{U}_q^D(\ghs)/\cI_\s$, we have:
\begin{align*}
			0=Se_{i,j}^+(0,s)=\lb{X^+_{i,0},\lb{X^+_{i,0},X^+_{j,s}}}=\lb{\lb{X^+_{j,s},X^+_{i,0}},X^+_{i,0}}=[[X^+_{j,s},X^+_{i,0}]_{q^{-A_{i,j}}},X^+_{i,0}]_{q^{-A_{i,j}-A_{i,i}}}. 
		\end{align*}
		
        \noindent Commuting $Se_{i,j}^+(0,s)$ with $X^-_{i,1}$ and using the relations \eqref{q Jacobi 1}, \eqref{XXK Util}, \eqref{K commutator 1}, and \eqref{K commutator 2}, we obtain:
		\begin{align*}
			0 & =[[[X^+_{j,s},X^+_{i,0}]_{q^{-A_{i,j}}},X^+_{i,0}]_{q^{-A_{i,j}-A_{i,i}}},X^-_{i,1}]                                                                                           \\
			  & =[[X^+_{j,s},X^+_{i,0}]_{q^{-A_{i,j}}},[X^+_{i,0},X^-_{i,1}]]_{q^{-A_{i,j}-A_{i,i}}} +     [[[X^+_{j,s},X^+_{i,0}]_{q^{-A_{i,j}}},X^-_{i,1}],X^+_{i,0}]_{q^{-A_{i,j}-A_{i,i}}} \\
			  & =[[X^+_{j,s},X^+_{i,0}]_{q^{-A_{i,j}}},[X^+_{i,0},X^-_{i,1}]]_{q^{-A_{i,j}-A_{i,i}}} + [[X^+_{j,s},  [X^+_{i,0},X^-_{i,1}]]_{q^{-A_{i,j}}},X^+_{i,0}]_{q^{-A_{i,j}-A_{i,i}}}   \\
			  & =[[X^+_{j,s},X^+_{i,0}]_{q^{-A_{i,j}}},K_i H_{i,1}]_{q^{-A_{i,j}-A_{i,i}}} + [[X^+_{j,s},K_i  H_{i,1}]_{q^{-A_{i,j}}},X^+_{i,0}]_{q^{-A_{i,j}-A_{i,i}}}                        \\
			  & =q^{-A_{i,j}}K_i\Big(q^{-A_{i,i}}[[X^+_{j,s},X^+_{i,0}]_{q^{-A_{i,j}}}, H_{i,1}]  +   [[X^+_{j,s}, H_{i,1}],X^+_{i,0}]_{q^{-A_{i,j}-2A_{i,i}}}              \Big).             
		\end{align*}
		
		\noindent Next, we use the relations \eqref{q Jacobi 1} and \eqref{HX Util} to obtain:
		\begin{align*}
			0 & =q^{-A_{i,i}}C\big([X^+_{j,s},[X^+_{i,0}, H_{i,1}]]_{q^{-A_{i,j}}} +[[X^+_{j,s},H_{i,1}], X^+_{i,0}]_{q^{-A_{i,j}}}\big)  -  [A_{i,j}]_q  [X^+_{j,s+1},X^+_{i,0}]_{q^{-A_{i,j}-2A_{i,i}}}       \\
			  & =-q^{-A_{i,i}}[A_{i,i}]_q[X^+_{j,s},X^+_{i,1}]_{q^{-A_{i,j}}} -[A_{i,j}]_q\left(q^{-A_{i,i}}[X^+_{j,s+1}, X^+_{i,0}]_{q^{-A_{i,j}}}  +   [X^+_{j,s+1},X^+_{i,0}]_{q^{-A_{i,j}-2A_{i,i}}}\right) \\
			  & =-q^{-A_{i,i}}[A_{i,i}]_q[X^+_{j,s},X^+_{i,1}]_{q^{-A_{i,j}}} -[A_{i,j}]_q(1+q^{-A_{i,i}})\big(X^+_{j,s+1}X^+_{i,0}-    q^{-A_{i,j}-A_{i,i}}  X^+_{i,0}X^+_{j,s+1}  \big).                      
		\end{align*}
		        
		By hypothesis $|i|=0$, $j=i\pm 1$, and $i\neq N$ if $\g_\s$ is of type B. Then, we have that $A_{i,i}=-2A_{i,j}$ and $[A_{i,i}]_q=-( q^{A_{i,j}}+q^{-A_{i,j}})[A_{i,j}]_q$. Thus, 
		\begin{align*}
			\lb{X^+_{j,s},X^+_{i,1}} +\lb{X^+_{i,0},X^+_{j,s+1}}=0. 
		\end{align*}
		
		\noindent Since $\cI_\s$ is $\Tilde\tau_i$-invariant, we conclude that 
		\begin{align}\label{ij relation even}
			\lb{X^+_{j,s},X^+_{i,r+1}} +\lb{X^+_{i,r},X^+_{j,s+1}}=0.
		\end{align}

		\noindent Case $2$. $|i|=0$ and $i=j\neq N$. Let $\ell=i\pm 1$. Then,  in $\Tilde{U}_q^D(\ghs)/\cI_\s$ we have:
		\begin{align*}
			0=\lb{X^\pm_{i,0},\lb{X^\pm_{i,0},X^\pm_{\ell,1}}}=[X^\pm_{i,0},[X^\pm_{i,0},X^\pm_{\ell,1}]_{q^{-A_{i,\ell}}}]_{q^{-A_{i,\ell}-A_{i,i}}}. 
		\end{align*}
		Then, using the equations \eqref{q Jacobi 1}, \eqref{XXK Util}, and \eqref{HX Util}, we get:
		\begin{align*}
			0 & =[[X^+_{i,0},[X^+_{i,0},X^+_{\ell,1}]_{q^{-A_{i,\ell}}}]_{q^{-A_{i,\ell}-A_{i,i}}},X^-_{\ell,0}] \\
			  & =C[X^+_{i,0}, [X^+_{i,0},K_\ell H_{\ell,1}]_{q^{-A_{i,\ell}}}]_{q^{-A_{i,\ell}-A_{i,i}}}         \\
			  & =C[X^+_{i,0}, [X^+_{i,0},H_{\ell,1}]K_\ell]_{q^{-A_{i,\ell}-A_{i,i}}}                            \\
			  & =-[A_{i,\ell}]_q[X^+_{i,0}, X^+_{i,1}K_\ell]_{q^{-A_{i,\ell}-A_{i,i}}}                           \\ 
			  & =-[A_{i,\ell}]_q[X^+_{i,0}, X^+_{i,1}]_{q^{-A_{i,i}}}K_\ell .                                     
		\end{align*}
		This implies:
		\begin{align*}
			[X^+_{i,0}, X^+_{i,1}]_{q^{-A_{i,i}}}=\lb{X^+_{i,0}, X^+_{i,1}}=0. 
		\end{align*}
		By commuting this last equation with $H_{i,1}$, using \eqref{HX Util}, and induction, we conclude that:
		\begin{align}\label{ii relation}
			\lb{X^+_{i,r}, X^+_{i,s+1}} + \lb{X^+_{i,s}, X^+_{i,r+1}}=0. 
		\end{align}
		
		\noindent Case $3$. $|i|=1$, $i\neq j$, and $i,j\neq N$ if $\g$ is of type $D$. This case involves longer calculations  and to make
  the exposition clearer  we break longer expressions into parts. However, the strategy is similar.
		First, note that:
		\begin{align*}
			So_i^\pm(0,0,0) & =\lb{X^+_{i,0},\lb{X^+_{i+1,0},\lb{X^+_{i,0},X^+_{i-1,0}}}}                                                                                    \\
			                & =X^+_{i,0} X^+_{i+1,0} X^+_{i,0} X^+_{i-1,0} +(-1)^{|i-1|+|i+1|+|i-1||i+1|} X^+_{i,0} X^+_{i-1,0} X^+_{i,0} X^+_{i+1,0}+                       \\
			                & \hphantom{==}+(-1)^{|i-1|+|i+1|} X^+_{i+1,0} X^+_{i,0} X^+_{i-1,0} X^+_{i,0} + (-1)^{|i-1||i+1|} X^+_{i-1,0} X^+_{i,0} X^+_{i+1,0} X^+_{i,0} - \\
			                & \hphantom{==}-(-1)^{|i-1|}(q^{-A_{i,i-1}}+q^{-A_{i,i+1}}) X^+_{i,0} X^+_{i+1,0} X^+_{i-1,0} X^+_{i,0}.                                         
		\end{align*}
		
		\noindent Next, we commute each monomial with $X^-_{i-1,1}$:
		\begin{align*}
			  & [X^+_{i,0} X^+_{i+1,0} X^+_{i,0} X^+_{i-1,0},X^-_{i-1,1}]=X^+_{i,0} X^+_{i+1,0} X^+_{i,0} H_{i-1,1}K_1,                                                                                              \\
			  & [X^+_{i,0} X^+_{i-1,0} X^+_{i,0} X^+_{i+1,0},X^-_{i-1,1}]=(-1)^{|i-1|+|i-1||i+1|}C^{-1}q^{A_{i-1,i}}[A_{i-1,i}]_qX^+_{i,0}  X^+_{i,1} X^+_{i+1,0}K_1,                                                 \\
			  & [X^+_{i+1,0} X^+_{i,0} X^+_{i-1,0} X^+_{i,0},X^-_{i-1,1}]=(-1)^{|i-1|}C^{-1}q^{A_{i-1,i}}[A_{i-1,i}]_qX^+_{i+1,0}  X^+_{i,0} X^+_{i,1}K_1,                                                           \\
			  & [X^+_{i-1,0} X^+_{i,0} X^+_{i+1,0} X^+_{i,0},X^-_{i-1,1}]=(-1)^{|i-1||i+1|}q^{2A_{i-1,i}}\big([A_{i-1,i}]_qC^{-1}(X^+_{i,1}  X^+_{i+1,0} X^+_{i,0} + X^+_{i,0}  X^+_{i+1,0} X^+_{i,1})+              \\
			  & \hphantom{[X^+_{i-1,0} X^+_{i,0} X^+_{i+1,0} X^+_{i,0},X^-_{i-1,1}]==}+ X^+_{i,0}  X^+_{i+1,0} X^+_{i,0} H_{i-1,1}    \big)K_1 ,                                                                      \\
			  & [X^+_{i,0} X^+_{i+1,0} X^+_{i-1,0} X^+_{i,0},X^-_{i-1,1}]=(-1)^{|i-1|}q^{A_{i-1,i}}\big(C^{-1}[A_{i-1,i}]_q X^+_{i,0}  X^+_{i+1,0} X^+_{i,1} +X^+_{i,0}  X^+_{i+1,0} X^+_{i,0} H_{i-1,1}    \big)K_1 .\\
		\end{align*}
		
		\noindent Summing up and simplifying we prove that 
		\begin{align*}
			[So_i^\pm(0,0,0),X^-_{i-1,1}]=C^{-1}q^{A_{i-1,i}}[A_{i-1,i}]_q\big((-1)^{|i+1|}X^+_{i,0}X^+_{i,1}X^+_{i+1,0} + (-1)^{|i+1|}X^+_{i+1,0}X^+_{i,0}X^+_{i,1} \\ +q^{A_{i-1,i}}X^+_{i,1}X^+_{i+1,0}X^+_{i,0} -q^{-A_{i-1,i}}X^+_{i,0}X^+_{i+1,0}X^+_{i,1} \big)K_{i-1},
		\end{align*}
		and the element $Y$ belongs to $\cI_\s$,
		\begin{align*}
			Y:= (-1)^{|i+1|}X^+_{i,0}X^+_{i,1}X^+_{i+1,0} + (-1)^{|i+1|}X^+_{i+1,0}X^+_{i,0}X^+_{i,1}+ &             \\
			+q^{A_{i-1,i}}X^+_{i,1}X^+_{i+1,0}X^+_{i,0} -q^{-A_{i-1,i}}X^+_{i,0}X^+_{i+1,0}X^+_{i,1}.   & 
		\end{align*}
		Now, we commute $Y$ with $X^-_{i,0}$.
		\begin{align*}
			  & [X^+_{i,0}X^+_{i,1}X^+_{i+1,0},X^-_{i,0}]=(-1)^{|i+1|}\big(q^{A_{i,i+1}}CX^+_{i,0}X^+_{i+1,0}H_{i,1}K_i +q^{A_{i,i+1}}[A_{i,i+1}]X^+_{i,0}X^+_{i+1,1}K_i +        \\
			  & \hphantom{[X^+_{i,0}X^+_{i,1}X^+_{i+1,0},X^-_{i,0}]==}+X^+_{i,1}X^+_{i+1,0}\frac{q^{-A_{i,i+1}}K_{i}^{-1} - q^{A_{i,i+1}}K_{i}}{q-q^{-1}}\Big) ,                   \\
			  & [X^+_{i+1,0}X^+_{i,0}X^+_{i,1},X^-_{i,0}]=CX^+_{i+1,0}X^+_{i,0}H_{i,1}K_i -                                                                                       
			X^+_{i+1,0}X^+_{i,1}\frac{K_{i} - K_{i}^{-1}}{q-q^{-1}},\\
			  & [X^+_{i,1}X^+_{i+1,0}X^+_{i,0},X^-_{i,0}]=-(-1)^{|i+1|}q^{A_{i,i+1}}CX^+_{i+1,0}X^+_{i,0}H_{i,1}K_i -(-1)^{|i+1|}q^{A_{i,i+1}}[A_{i,i+1}]X^+_{i+1,1}X^+_{i,0}K_i+ \\
			  & \hphantom{[X^+_{i,1}X^+_{i+1,0}X^+_{i,0},X^-_{i,0}]==}+X^+_{i,1}X^+_{i+1,0}\frac{K_{i} - K_{i}^{-1}}{q-q^{-1}},                                                    \\
			  & [X^+_{i,0}X^+_{i+1,0}X^+_{i,1},X^-_{i,0}]=                                                                                                                        
			CX^+_{i,0}X^+_{i+1,0}H_{i,1}K_i + (-1)^{|i+1|}X^+_{i+1,0}X^+_{i,1}\frac{q^{-A_{i,i+1}}K_{i}^{-1} - q^{A_{i,i+1}}K_{i}}{q-q^{-1}} .   
		\end{align*}
		
		\noindent Putting the last four expressions together and simplifying,
		\begin{align*}
			[Y,X^-_{i,0}] & =q^{A_{i,i+1}}[A_{i,i+1}]_q\big(X^+_{i,0}X^+_{i+1,1} - (-1)^{|i+1|}q^{-A_{i,i+1}}X^+_{i+1,1}X^+_{i,0}        \\
			              & \hphantom{==}+(-1)^{|i+1|}(X^+_{i+1,0}X^+_{i,1}- (-1)^{|i+1|}q^{-A_{i,i+1}}X^+_{i,1}X^+_{i+1,0})\big)K_{i}   \\
			              & =q^{A_{i,i+1}}[A_{i,i+1}]_q\big(\lb{X^+_{i,0}X^+_{i+1,1}}  +(-1)^{|i+1|}\lb{X^+_{i+1,0}X^+_{i,1}}\big)K_{i}. 
		\end{align*}
		
		\noindent Thus, $\lb{X^+_{i,0},X^+_{i+1,1}}  +(-1)^{|i+1|}\lb{X^+_{i+1,0},X^+_{i,1}}=0$.
	 The invariance of $\cI_\s$ under the action of $\Tilde\tau_i$ and $\Tilde\tau_{i+1}$ implies that
		\begin{align}\label{i i+1 relation odd}
			\lb{X^+_{i,r},X^+_{i+1,s+1}}  +(-1)^{|i+1|}\lb{X^+_{i+1,s},X^+_{i,r+1}}=0. 
		\end{align}
		
		Similarly, if we commute $So_i^\pm(0,0,0)$ with $X^-_{i+1,1}$ instead of  $X^-_{i-1,1}$ and repeat the previous computations, we obtain:
		\begin{align}\label{i i-1 relation odd}
			\lb{X^+_{i,r},X^+_{i-1,s+1}}  +(-1)^{|i-1|}\lb{X^+_{i-1,s},X^+_{i,r+1}}=0. 
		\end{align}

\noindent Case 4. $|i|=1$, $i=N$, $j=N-1$, and $\g$ is of type $D$. That is, the $(N-2,N)$ section of the Dynkin diagram has the same shape as in Figure \ref{XD odd}. Then,  in $\Tilde{U}_q^D(\ghs)/\cI_\s$:
\begin{align*}
			0=S_{T}^+(0,0,0)
                &=\lb{\lb{X^+_{N-2,0},X^+_{N-1,0}},X^+_{N,0}}-\lb{\lb{X^+_{N-2,0},X^+_{N,0}},X^+_{N-1,0}}\\
                &=[[X^+_{N-2,0},X^+_{N-1,0}]_{q^{-1}},X^+_{N,0}]_{q}-[[X^+_{N-2,0},X^+_{N,0}]_{q^{-1}},X^+_{N-1,0}]_{q}. 
		\end{align*}

    Now, we commute $S_{T}^+(0,0,0)$ with $X^-_{N-2,1}$ and use the relations \eqref{q Jacobi 1}, \eqref{XXK Util}, and \eqref{K commutator 1}, to obtain:
		\begin{align*}
			0 & =[S_{Ts}^+(0,0,0),X^-_{N-2,1}]                                                                                           \\
			  & = K_{N-2}\left( [[H_{N-2,1},X^+_{N-1,0}]_{q^{-2}},X^+_{N,0}]  - [[H_{N-2,1},X^+_{N,0}]_{q^{-2}},X^+_{N-1,0}] \right), 
		\end{align*}
which, after expanding, moving the $H_{N-2,1}$ in each summand to the left with the help of equation \eqref{HX Util}, multiplying by $K_{N-2}^{-1}$, yields
\begin{align*}
			\lb{X^+_{N-1,0},X^+_{N,1}}  -\lb{X^+_{N,0},X^+_{N-1,1}}=0
		\end{align*}
Then, applying $\Tilde\tau_{N-1}^{-r}$ and $\Tilde\tau_{N}^{-s}$ we get
		\begin{align}
			\lb{X^+_{N-1,r},X^+_{N,s+1}}  -\lb{X^+_{N,s},X^+_{N-1,r+1}}=0. \label{ij type D} 
		\end{align}

\noindent Case 5. $|i|=0$, $i=j=N$, and $\gs$ is of type $B$. We will show first that the relation \eqref{cubic relation} holds in $\Tilde{U}_q^D(\ghs)/\cI_\s$ if $\alpha_{N,\s}$ is either even or odd. Then, we will show that in the particular case when $|N|=0$, the relations \eqref{cubic relation} implies \eqref{s1 rel} with $i=j=N$. On the other hand, when $|N|=1$, the relations \eqref{cubic relation} implies \eqref{quadratic N}.

		We have that $S_N^\pm(0,0)=\lb{X^\pm_{N,0},\lb{X^\pm_{N,0},\lb{X^\pm_{N,0},X^\pm_{N-1,0}}}}\in \cI_\s$. Then, using the equations \eqref{q Jacobi 1}, \eqref{XXK Util}, \eqref{K commutator 1}, and \eqref{HX Util}, we get
		\begin{align*}
			0 & =  [S_N^\pm(0,0),X^-_{N-1,1}]=[X^+_{N,0},[X^+_{N,0},[X^+_{N,0},[X^+_{N-1,0},X^-_{N-1,1}]]_{q^{A_{N,N}}}]]_{q^{-A_{N,N}}} \\
			  & =[X^+_{N,0},[X^+_{N,0},[X^+_{N,0},K_{N-1}H_{N-1,1}]_{q^{A_{N,N}}}]]_{q^{-A_{N,N}}}                                       \\
			  & =q^{3A_{N,N}}K_{N-1}[X^+_{N,0},[X^+_{N,0},[X^+_{N,0},H_{N-1,1}]]_{q^{-A_{N,N}}}]_{q^{-2A_{N,N}}}                         \\
			  & =-q^{3A_{N,N}}u_{N-1,N,1}C^{-1}K_{N-1}[X^+_{N,0},[X^+_{N,0},X^+_{N,1}]_{q^{-A_{N,N}}}]_{q^{-2A_{N,N}}}.                  
		\end{align*}
        
		\noindent Hence, 
        \begin{align}\label{triple N deg 0}
            [X^+_{N,0},[X^+_{N,0},X^+_{N,1}]_{q^{-A_{N,N}}}]_{q^{-2A_{N,N}}}=\lb{X^+_{N,0},\lb{X^+_{N,0},X^+_{N,1}}} =0,
        \end{align}
        and
		the relation \eqref{cubic relation} follows from the fact that $\cI_\s$ is invariant by the action of $\Tilde\tau_N$, by commuting the previous equation with $H_{N,1}$, and induction.

Now,  to show that \eqref{s1 rel} with $i=j=N$ and \eqref{quadratic N} holds in $\Tilde{U}_q^D(\ghs)/\cI_\s$
    we use \eqref{triple N deg 0}
		\begin{align*}
			0=[X^+_{N,0},[X^+_{N,0},X^+_{N,1}]_{q^{-A_{N,N}}}]_{q^{-2A_{N,N}}}. 
		\end{align*}
		
		\noindent Then, using the identities \eqref{q Jacobi 1}, \eqref{XXK Util}, \eqref{K commutator 1}, \eqref{K commutator 2} and \eqref{HX Util}, we get
		\begin{align*}
			  & 0=[[X^+_{N,0},[X^+_{N,0},X^+_{N,1}]_{q^{-A_{N,N}}}]_{q^{-2A_{N,N}}},X^-_{N,-1}]=[X^+_{N,0},[X^+_{N,0},\frac{CK_N-(CK_N)^{-1}}{q-q^{-1}}]_{q^{-A_{N,N}}}]_{q^{-2A_{N,N}}}+                                                                                    \\ 
			  & \hphantom{=}+(-1)^{|N|}[X^+_{N,0},[(CK_N)^{-1}H_{N,-1},X^+_{N,1}]_{q^{-A_{N,N}}}]_{q^{-2A_{N,N}}} +[(CK_N)^{-1}H_{N,-1},[X^+_{N,0},X^+_{N,1}]_{q^{-A_{N,N}}}]_{q^{-2A_{N,N}}} \\
			  & =(CK_N)^{-1}\Big(-(-1)^{|N|}q^{A_{N,N}}[X^+_{N,0},X^+_{N,0}]_{q^{-3A_{N,N}}} + (-1)^{|N|}q^{A_{N,N}}[X^+_{N,0},[H_{N,-1},X^+_{N,1}]]_{q^{-3A_{N,N}}}+                         \\
			  & \hphantom{=} +[[H_{N,-1},X^+_{N,0}],X^+_{N,1}]]_{q^{-A_{N,N}}} +[X^+_{N,0},[H_{N,-1},X^+_{N,1}]]_{q^{-A_{N,N}}}\Big)                                                          \\
			  & =(CK_N)^{-1}\Big((u_{N,N,-1}-(-1)^{|N|})q^{A_{N,N}}[X^+_{N,0},X^+_{N,0}]_{q^{-3A_{N,N}}} +u_{N,N,-1}[X^+_{N,0},X^+_{N,0}]_{q^{-A_{N,N}}}+                                     \\
			  & \hphantom{=}+u_{N,N,-1}[X^+_{N,-1},X^+_{N,1}]_{q^{-A_{N,N}}}\Big),
		\end{align*}

\noindent which implies
\begin{align}\label{quadratic N deg 0}
    &(u_{N,N,-1}-(-1)^{|N|})q^{A_{N,N}}[X^+_{N,0},X^+_{N,0}]_{q^{-3A_{N,N}}} +u_{N,N,-1}(\lb{X^+_{N,0},X^+_{N,0}}+\lb{X^+_{N,-1},X^+_{N,1}})=0.
\end{align}

We now consider this last equation in the cases $|N|=0$ and $|N|=1$ separately.

If $|N|=1$, the equation \eqref{quadratic N deg 0} implies 
		\begin{align}\label{N quadratic r=s=-1}
			[X^+_{N,1},X^+_{N,-1}]_{q^{A_{N,N}}}- q^{2A_{N,N}}[X^+_{N,0},X^+_{N,0}]_{q^{-3A_{N,N}}}=0. 
		\end{align}
		
		\noindent The relation \eqref{quadratic N} follows from the fact that $\cI_\s$ is invariant by the action of $\Tilde\tau_N$, by commuting the equation \eqref{N quadratic r=s=-1} with $H_{N,1}$, and induction.

If $|N|=0$, the equation \eqref{quadratic N deg 0} implies

		\begin{align*}
			\lb{X^+_{N,0}, X^+_{N,0}} + \lb{X^+_{N,-1}, X^+_{N,1}}=0. 
		\end{align*}
		By commuting this last equation with $H_{i,1}$, using \eqref{HX Util}, and induction, we conclude that:
		\begin{align}\label{NN relation}
			\lb{X^+_{i,r}, X^+_{i,s+1}} + \lb{X^+_{i,s}, X^+_{i,r+1}}=0. 
		\end{align}

		To summarize, the equations \eqref{ij relation even}--\eqref{ij type D} and \eqref{NN relation} can all be written as
		\begin{align*}
			\lb{X^+_{i,r},X^+_{j,s+1}}  +(-1)^{|i||j|}\lb{X^+_{j,s},X^+_{i,r+1}}=0, 
		\end{align*}
		which is equivalent to \eqref{s1 rel}.
		
Note that if $|N|=0$ and $\gs$ is of type $B$, then the relation \eqref{s1 rel} with $i=j=N$ implies \eqref{cubic relation}. Therefore, we do not impose the relation \eqref{cubic relation} in the definition of the algebra.

        We now use the equations \eqref{XX rel}, \eqref{s1 rel}  and \eqref{quadratic N} to show that \eqref{HX rel} is satisfied as follows.
		Denoting $K_{i,r}^+=(q-q^{-1})C^{-r}K^{-1}_i[X^+_{i,r},X^-_{i,0}]$, for $r>0$, the series  $K_{i}^+(z)$ can be rewritten as $$K_{i}^+(z)=K_i(1+\sum_{ r\geq 0}K_{i, r}^+ z^{- r}).$$

		\noindent Then, using the equations \eqref{K commutator 2} and \eqref{q Jacobi 1}, we get:
		\begin{align*}
			[K_{i,r}^+,X^+_{j,0}] & =(q-q^{-1})C^{-r}[K^{-1}_i[X^+_{i,r},X^-_{i,0}],X^+_{j,0}]                =(q-q^{-1})C^{-r}K^{-1}_i[[X^+_{i,r},X^-_{i,0}],X^+_{j,0}]_{q^{A_{i,j}}}                                                                                                                \\
			                      & =(q-q^{-1})C^{-r}K^{-1}_i\left( [X^+_{i,r},[X^-_{i,0},X^+_{j,0}]]_{q^{A_{i,j}}}+(-1)^{|i||j|}[[X^+_{i,r},X^+_{j,0}]_{q^{A_{i,j}}},X^-_{i,0}]  \right)                                   \\
			                      & =(q-q^{-1})C^{-r}K^{-1}_i\left(-(-1)^{|i|}[X^+_{i,r},\delta_{i,j}\frac{K_i-K_i^{-1}}{q-q^{-1}}]_{q^{A_{i,j}}} +(-1)^{|i||j|}[[X^+_{i,r},X^+_{j,0}]_{q^{A_{i,j}}},X^-_{i,0}]     \right) \\
			                      & =(q-q^{-1})C^{-r}K^{-1}_i\left((-1)^{|i|}[A_{i,i}]_q\delta_{i,j}K_iX^+_{i,r} +(-1)^{|i||j|}[[X^+_{i,r},X^+_{j,0}]_{q^{A_{i,j}}},X^-_{i,0}]     \right).                                  
		\end{align*}
		We consider two cases.
		
		\noindent Case 1. Suppose $(i,j)\neq (N,N)$. Then, using the relation \eqref{s1 rel} 
		\begin{align*}
			[X^+_{i,r+1},X^+_{j,s}]_{q^{A_{i,j}}}+ (-1)^{|i||j|}[X^+_{j,s+1},X^+_{i,r}]_{q^{A_{i,j}}}=0, 
		\end{align*}
		we get: 
		\begin{align*}
			[K_{i,r}^+,X^+_{j,0}] & =(q-q^{-1})C^{-r}K^{-1}_i\left((-1)^{|i|}[A_{i,i}]_q\delta_{i,j}K_iX^+_{i,r} +(-1)^{|i||j|}[[X^+_{i,r},X^+_{j,0}]_{q^{A_{i,j}}},X^-_{i,0}]     \right) \\
			                      & =(q-q^{-1})C^{-r}K^{-1}_i\left((-1)^{|i|}[A_{i,i}]_q\delta_{i,j}K_iX^+_{i,r} -[[X^+_{j,1},X^+_{i,r-1}]_{q^{A_{i,j}}},X^-_{i,0}]     \right).           
		\end{align*}		
		Next subsequently applying \eqref{q Jacobi 1}, \eqref{XX rel}, \eqref{K commutator 1}, and \eqref{K commutator 2}, we have:
		
		\begin{align*}
			[[X^+_{j,1},X^+_{i,r-1}]_{q^{A_{i,j}}},X^-_{i,0}] & =[X^+_{j,1},[X^+_{i,r-1},X^-_{i,0}]]_{q^{A_{i,j}}} + (-1)^{|i|}[[X^+_{j,1},X^-_{i,0}],X^+_{i,r-1}]_{q^{A_{i,j}}}                                                             \\
			                                                  & =\left[X^+_{j,1},\frac{C^{r-1}K_i K^+_{i,r-1}}{q-q^{-1}}\right]_{q^{A_{i,j}}} + (-1)^{|i|}\delta_{i,j}\left[\frac{CK_i K^+_{i,1}}{q-q^{-1}},X^+_{i,r-1}\right]_{q^{A_{i,j}}} \\
			                                                  & =\frac{q^{-A_{i,j}}C^{r-1}K_i }{q-q^{-1}}[X^+_{j,1},K^+_{i,r-1}]_{q^{2A_{i,j}}} + (-1)^{|i|}\delta_{i,j}\frac{CK_i }{q-q^{-1}}[K^+_{i,1},X^+_{i,r-1}]                        \\
			                                                  & =\frac{q^{-A_{i,j}}C^{r-1}K_i }{q-q^{-1}}[X^+_{j,1},K^+_{i,r-1}]_{q^{2A_{i,j}}}+ (-1)^{|i|}\delta_{i,j}CK_i [H_{i,1},X^+_{i,r-1}]                                            \\
			                                                  & =\frac{q^{-A_{i,j}}C^{r-1}K_i }{q-q^{-1}}[X^+_{j,1},K^+_{i,r-1}]_{q^{2A_{i,j}}}+ (-1)^{|i|}[A_{i,i}]_q\delta_{i,j}K_i X^+_{i,r}.                                             
		\end{align*}
		Thus,
		\begin{align*}
			[K_{i,r}^+,X^+_{j,0}] &  =-q^{-A_{i,j}}C^{-1}[X^+_{j,1},K^+_{i,r-1}]_{q^{2A_{i,j}}}.                                                                                 
		\end{align*}
		Which can be rewritten as
		\begin{align*}
			K_{i,r}^+X^+_{j,0}-q^{A_{i,j}}C^{-1}K_{i,r-1}^+X^+_{j,1}= X^+_{j,0}K_{i,r}^+-q^{-A_{i,j}}C^{-1}X^+_{j,1}K_{i,r-1}^+. 
		\end{align*}
Wrapping up these relations into a series we obtain:
		\begin{align*}
			(1-q^{A_{i,j}}z^{-1}C^{-1}(-1)^j\tau_j^{-1})K_i^{-1}K^+_{i}(z)X^+_{j,0}= (1-q^{-A_{i,j}}z^{-1}C^{-1}(-1)^j\tau_j^{-1})X^+_{j,0}K_i^{-1}K^+_{i}(z), 
		\end{align*}
		which is equivalent to 
		\begin{align*}
			(q-q^{-1})\sum_{r>0}[H_{i,r},X^+_{j,0}]z^{-r}=\left(\log(1-q^{-A_{i,j}}z^{-1}C^{-1}(-1)^j\tau_j^{-1})-\log(1-q^{A_{i,j}}z^{-1}C^{-1}(-1)^j\tau_j^{-1})\right)X^+_{j,0}. 
		\end{align*}
		Therefore, for $r>0$, we have
		\begin{align*}
			[H_{i,r},X^+_{j,0}]=\frac{[rA_{i,j}]_q}{r}C^{-r}X^+_{j,r}. 
		\end{align*}
		
		\noindent The relation \eqref{HX rel} when $(i,j)\neq (N,N)$ follows for all $r\neq 0$ by applying the anti-automorphism $\tilde{\eta}$.
		
		\noindent Case 2. Now, we suppose that $i=j=N$, $|N|=1$, and $\gs$ is of type $B$. In this case computation is longer, but the strategy is the same.
		
		Let $r=2$. Then, using the relation \eqref{quadratic N}
		\begin{align*}
			[X^+_{N,2},X^+_{N,0}]_{q^{A_{N,N}}}- q^{2A_{N,N}}[X^+_{N,1},X^+_{N,1}]_{q^{-3A_{N,N}}}=0 
		\end{align*}
		we get
		\begin{align*}
			[K_{N,2}^+,X^+_{N,0}] & = (q-q^{-1})C^{-2}K^{-1}_N\left((-1)^{|N|}[A_{N,N}]_qK_NX^+_{N,2} +(-1)^{|N|}[[X^+_{N,2},X^+_{N,0}]_{q^{A_{N,N}}},X^-_{N,0}]     \right)               \\
			                      & =  (q-q^{-1})C^{-2}K^{-1}_N\Big((-1)^{|N|}[A_{N,N}]_q K_NX^+_{N,2} +(-1)^{|N|} q^{2A_{N,N}}[[X^+_{N,1},X^+_{N,1}]_{q^{-3A_{N,N}}},X^-_{N,0}]     \Big). 
		\end{align*}
		
		\noindent Using the equations \eqref{q Jacobi 1}, \eqref{K commutator 1} and \eqref{K commutator 2}, we get
		\begin{align}\label{KNr commutator}
			[[X^+_{N,1},X^+_{N,1}]_{q^{-3A_{N,N}}},X^-_{N,0}] & =[X^+_{N,1},\frac{CK_NK^+_{N,1}}{q-q^{-1}}]_{q^{-3A_{N,N}}} + (-1)^{|N|}[\frac{CK_NK^+_{N,1}}{q-q^{-1}},X^+_{N,1}]_{q^{-3A_{N,N}}} \notag \\
			                                                  & =\frac{CK_N}{q-q^{-1}}\left(q^{-A_{N,N}}[X^+_{N,1},K^+_{N,1}]_{q^{-2A_{N,N}}} + (-1)^{|N|}[K^+_{N,1},X^+_{N,1}]_{q^{-3A_{N,N}}}\right)    \\
			                                                  & =((-1)^{|N|}-q^{-3A_{N,N}})\frac{CK_N}{q-q^{-1}}\left(K^+_{N,1}X^+_{N,1} + (-1)^{|N|}q^{-A_{N,N}} X^+_{N,1}K^+_{N,1}  \right).\notag      
		\end{align}
		
		\noindent Thus,
		\begin{align}\label{KN2}
			[K_{N,2}^+,X^+_{N,0}] & =  (-1)^{|N|}(q-q^{-1})C^{-2}K^{-1}_N\Big([A_{N,N}]_q K_NX^+_{N,2} + q^{2A_{N,N}}[[X^+_{N,1},X^+_{N,1}]_{q^{-3A_{N,N}}},X^-_{N,0}]     \Big) \notag \\
			                      & =(q^{2A_{N,N}}-(-1)^{|N|}q^{-A_{N,N}})C^{-1}\left(K^+_{N,1}X^+_{N,1} + (-1)^{|N|}q^{-A_{N,N}} X^+_{N,1}K^+_{N,1}  \right)+                          \\
			                      & \hphantom{=}+(-1)^{|N|}(q^{A_{N,N}}-q^{-A_{N,N}}) C^{-2}X^+_{N,2}.\notag                                                                            
		\end{align}
		
		Further let $r>2$, and using the relation 
		\begin{align*}
			[X^+_{N,r},X^+_{N,0}]_{q^{A_{N,N}}} + [X^+_{N,2},X^+_{N,r-2}]_{q^{A_{N,N}}} -q^{2A_{N,N}}[X^+_{N,r-1},X^+_{N,1}]_{q^{-3A_{N,N}}} - &     \\
			- q^{2A_{N,N}}[X^+_{N,1},X^+_{N,r-1}]_{q^{-3A_{N,N}}}                                                                              & =0, 
		\end{align*}
		we get
		\begin{align*}
			[K_{N,r}^+,X^+_{N,0}] & = (q-q^{-1})C^{-r}K^{-1}_N(-1)^{|N|}\left([A_{N,N}]_qK_NX^+_{N,r} +[[X^+_{N,r},X^+_{N,0}]_{q^{A_{N,N}}},X^-_{N,0}]     \right)            \\
			                      & =  (q-q^{-1})C^{-r}K^{-1}_N(-1)^{|N|}\big([A_{N,N}]_qK_NX^+_{N,r}+ q^{2A_{N,N}}[[X^+_{N,r-1},X^+_{N,1}]_{q^{-3A_{N,N}}},X^-_{N,0}]+       \\
			                      & \hphantom{=}+  q^{2A_{N,N}}[[X^+_{N,1},X^+_{N,r-1}]_{q^{-3A_{N,N}}},X^-_{N,0}] - [[X^+_{N,2},X^+_{N,r-2}]_{q^{A_{N,N}}},X^-_{N,0}] \big). 
		\end{align*}
		
		\noindent Similar to \eqref{KNr commutator}, we have for the second and third summands 
		\begin{align*}
			  & [[X^+_{N,r-1},X^+_{N,1}]_{q^{-3A_{N,N}}},X^-_{N,0}] + [[X^+_{N,1},X^+_{N,r-1}]_{q^{-3A_{N,N}}},X^-_{N,0}]=                                                        \\
			  & \hphantom{aaaaaaaaaaaaaaa}=\frac{K_N}{q-q^{-1}}((-1)^{|N|}-q^{-3A_{N,N}})\Big( C\big(K_{N,1}^+ X^+_{N,r-1}+ (-1)^{|N|}q^{-A_{N,N}} X^+_{N,r-1} K_{N,1}^+   \big)+ \\
			  & \hphantom{aaaaaaaaaaaaaaa=}+  C^{r-1}\big(K_{N,r-1}^+ X^+_{N,1}+ (-1)^{|N|}q^{-A_{N,N}} X^+_{N,1} K_{N,r-1}^+   \big)         \Big),                              
		\end{align*}
		and for the last one
		\begin{align*}
			  & [[X^+_{N,2},X^+_{N,r-2}]_{q^{A_{N,N}}},X^-_{N,0}]=\frac{K_N}{q-q^{-1}}\Big( C^{r-2}\big( q^{-A_{N,N}} X^+_{N,2} K_{N,r-2}^+ -q^{A_{N,N}}K_{N,r-2}^+ X^+_{N,2}   \big)+             \\
			  & \hphantom{[[X^+_{N,2},X^+_{N,r-2}]_{q^{A_{N,N}}},X^-_{N,0}]=}+  (-1)^{|N|}C^{2}[K_{N,2}^+, X^+_{N,r-2}]         \Big)                                                              \\
			  & \hphantom{[[X^+_{N,2},X^+_{N,r-2}]_{q^{A_{N,N}}},X^-_{N,0}]}=\frac{K_N}{q-q^{-1}}\Big( C^{r-2}\big( q^{-A_{N,N}} X^+_{N,2} K_{N,r-2}^+ -q^{A_{N,N}}K_{N,r-2}^+ X^+_{N,2}   \big) + \\
			  & \hphantom{[[X^+_{N,2},X^+_{N,r-2}]_{q^{A_{N,N}}},X^-_{N,0}]=}+((-1)^{|N|}q^{2A_{N,N}}-q^{-A_{N,N}})C\left(K^+_{N,1}X^+_{N,1} + (-1)^{|N|}q^{-A_{N,N}} X^+_{N,1}K^+_{N,1}  \right)+ \\
			  & \hphantom{[[X^+_{N,2},X^+_{N,r-2}]_{q^{A_{N,N}}},X^-_{N,0}]=} + (q^{A_{N,N}}-q^{-A_{N,N}}) X^+_{N,2}\Big).                                                                         
		\end{align*}
		Substituting them back into the expression for $[K_{N,r}^+,X^+_{N,0}]$,
		\begin{align*}
			[K_{N,r}^+,X^+_{N,0}] & = ((-1)^{|N|}q^{A_{N,N}}-q^{-2A_{N,N}})X^+_{N,1}K_{N,r-1}^ + +(q^{2A_{N,N}}-(-1)^{|N|}q^{-A_{N,N}})K_{N,r-1}^+X^+_{N,1} + \\
			                      & \hphantom{=}+(-1)^{|N|}( q^{A_{N,N}}K_{N,r-2}^+X^+_{N,2}-q^{-A_{N,N}}X^+_{N,2}K_{N,r-2}^+).                               
		\end{align*}
		Wrapping up these relations into a series we obtain
		\begin{align*}
			  & (1-(-1)^N q^{2A_{N,N}}(Cz\tau_N)^{-1})(1+(-1)^{|N|+N}q^{-A_{N,N}}(Cz\tau_N)^{-1}))K_N^{-1}K^+_{N}(z)X^+_{N,0}=                         \\
			  & \hphantom{==========}= (1-(-1)^N q^{-2A_{N,N}}(Cz\tau_N)^{-1}))(1+(-1)^{|N|+N}q^{A_{N,N}}(Cz\tau_N)^{-1}))X^+_{N,0}K_N^{-1}K^+_{N}(z), 
		\end{align*}
		which is equivalent to 
		\begin{align*}
			(q-q^{-1})\sum_{r>0}[H_{N,r},X^+_{N,0}]z^{-r}= & \Big(\log(1-(-1)^N q^{-2A_{N,N}}(Cz\tau_N)^{-1})+                \\
			                                               & +\log(1+(-1)^{|N|+N}q^{A_{N,N}}(Cz\tau_N)^{-1})-                 \\
			                                               & -\log(1-(-1)^Nq^{2A_{N,N}}(Cz\tau_N)^{-1})-                      \\
			                                               & -\log(1+(-1)^{|N|+N}q^{-A_{N,N}}(Cz\tau_N)^{-1}) \Big)X^+_{N,0}. 
		\end{align*}
		Therefore, for $r>0$, we have
		\begin{align*}
			[H_{N,r},X^+_{N,0}]=C^{-r}\left(\frac{[2rA_{N,N}]_q}{r}-(-1)^{(|N|+1)r}\frac{[rA_{N,N}]_q}{r}\right)X^+_{N,r}. 
		\end{align*}
		
		The relation \eqref{HX rel} with $i=j=N$ and $|N|=1$ follows for all $r\neq 0$ by applying the anti-automorphism $\tilde{\eta}$.

		\medskip
		
		Now, we show that the relation \eqref{HH rel} is a consequence of the equations \eqref{HX rel} and \eqref{q Jacobi 2}.

		\noindent Since $[H_{i,r}, H_{j,s}]=-[H_{j,s}, H_{i,r}]=\Tilde{\Omega}[H_{j,-s}, H_{i,-r}]$, we may assume $|r|\geq s>0$. Then,
		\begin{align*}
			[H_{i,r}, K^+_{j,s}]= & [H_{i,r},(q-q^{-1})C^{-s}K_j^{-1}[X^+_{j,s},X^-_{j,0}]]=                                                                                \\
			=                     & (q-q^{-1})C^{-s}K_j^{-1}\big([[H_{i,r},X^+_{j,s}],X^-_{j,0}]+[X^+_{j,s},[H_{i,r},X^-_{j,0}]]\big)                                       \\
			=                     & (q-q^{-1})C^{-s}K_j^{-1}u_{i,j,r}\big(C^{-(r+|r|)/2}[X^+_{j,r+s},X^-_{j,0}] -C^{-(r-|r|)/2}[X^+_{j,s},X^-_{j,r}]\big)                   \\
			=                     & C^{-s}K_j^{-1}u_{i,j,r}\big(C^{-(r+|r|)/2}(C^{r+s}K_jK^+_{j,r+s}-K_j^{-1}K^-_{j,r+s}) -(C^{s}K_jK^+_{j,r+s}-C^{r}K_j^{-1}K^-_{j,r+s})\big)                  \\
			=                     & C^{-s}K_j^{-1}u_{i,j,r}\big(C^{s}(C^{(r-|r|)/2}-C^{(|r|-r)/2})K_jK^+_{j,r+s} +(C^{(r+|r|)/2}-C^{-(r+|r|)/2})K_j^{-1}K^-_{j,r+s}  \big).
		\end{align*}
		By hypothesis, $|r|\geq s>0$. If $r>0$, then $r+s>0$ and $K^-_{j,r+s}=0$. Thus, $ [H_{i,r}, K^+_{j,s}]=0$.
		 If $r<0$, then $r+s\leq 0$ and $K^+_{j,r+s}=\delta_{r+s,0}$. Thus, in both cases, we have
		\begin{align*}
			[H_{i,r}, K^+_{j,s}]=u_{i,j,r}\delta_{r+s,0}(C^{r}-C^{-r}). 
		\end{align*} 
		Therefore,
		\begin{align*}
			[H_{i,r}, {H}_{j,s}]=u_{i,j,r}\delta_{r+s,0}\frac{C^{r}-C^{-r}}{q-q^{-1}}. 
		\end{align*}

		It remains to show that the relation \eqref{cubic N N-1} holds in $\Tilde{U}_q^D(\ghs)/\cI_\s$. 
		
		Recall that $A_{N,N}=-A_{N,N-1}$. Hence, we can write
		\begin{align*}
			S_N^\pm(0,0) & =[X^+_{N,0},[X^+_{N,0},[X^+_{N,0},X^+_{N-1,0}]_{q^{A_{N,N}}}]]_{q^{-A_{N,N}}}  \\
			             & =[[[X^+_{N-1,0},X^+_{N,0}]_{q^{A_{N,N}}},X^+_{N,0}],X^+_{N,0}]_{q^{-A_{N,N}}}. 
		\end{align*}		
		Then, using the equations \eqref{q Jacobi 1}, \eqref{K commutator 2}, and \eqref{K commutator 1}, we have 
		\begin{align*}
			0   & =[S_N^\pm(0,0),X^-_{N,1}]   =[[[X^+_{N-1,0},K_NH_{N,1}]_{q^{A_{N,N}}},X^+_{N,0}],X^+_{N,0}]_{q^{-A_{N,N}}} \\
			    & +(-1)^{|N|}[[[X^+_{N-1,0},X^+_{N,0}]_{q^{A_{N,N}}},K_NH_{N,1}],X^+_{N,0}]_{q^{-A_{N,N}}}                    +[[[X^+_{N-1,0},X^+_{N,0}]_{q^{A_{N,N}}},X^+_{N,0}],K_NH_{N,1}]_{q^{-A_{N,N}}}                                                                                                  \\
			    & =K_N\Big(q^{A_{N,N}} [[[X^+_{N-1,0},H_{N,1}],X^+_{N,0}]_{q^{-A_{N,N}}},X^+_{N,0}]_{q^{-2A_{N,N}}}+                                                                                          \\
			    & \hphantom{=}+(-1)^{|N|}[[[X^+_{N-1,0},X^+_{N,0}]_{q^{A_{N,N}}},H_{N,1}],X^+_{N,0}]_{q^{-2A_{N,N}}} +q^{-A_{N,N}}[[[X^+_{N-1,0},X^+_{N,0}]_{q^{A_{N,N}}},X^+_{N,0}],H_{N,1}]\Big)           \\
			 & =K_N\Big(q^{A_{N,N}} [[[X^+_{N-1,0},H_{N,1}],X^+_{N,0}]_{q^{-A_{N,N}}},X^+_{N,0}]_{q^{-2A_{N,N}}}+                                                                                          \\
			    & (-1)^{|N|}[[[X^+_{N-1,0},[X^+_{N,0},H_{N,1}]]_{q^{A_{N,N}}},X^+_{N,0}]_{q^{-2A_{N,N}}} +(-1)^{|N|}[[[X^+_{N-1,0},H_{N,1}],X^+_{N,0}]_{q^{A_{N,N}}},X^+_{N,0}]_{q^{-2A_{N,N}}} \\
			    & \hphantom{=}+q^{-A_{N,N}}[[X^+_{N-1,0},X^+_{N,0}]_{q^{A_{N,N}}},[X^+_{N,0},H_{N,1}]] +q^{-A_{N,N}}[[X^+_{N-1,0},[X^+_{N,0},H_{N,1}]]_{q^{A_{N,N}}},X^+_{N,0}]+                              \\
			    & \hphantom{=}+q^{-A_{N,N}}[[[X^+_{N-1,0},H_{N,1}],X^+_{N,0}]_{q^{A_{N,N}}},X^+_{N,0}]\Big).                     
		\end{align*}
		Next we apply the equation \eqref{HX Util} then expanding and reassembling the commutators, we get
		\begin{align*}
			0   & =-C^{-1}K_N\Big( u_{N-1,N,1}\big(q^{A_{N,N}}[[X^+_{N-1,1},X^+_{N,0}]_{q^{-A_{N,N}}},X^+_{N,0}]_{q^{-2A_{N,N}}}+                                                         \\
			    & \hphantom{=}+(-1)^{|N|}[[X^+_{N-1,1},X^+_{N,0}]_{q^{A_{N,N}}},X^+_{N,0}]_{q^{-2A_{N,N}}} +q^{-A_{N,N}}[[X^+_{N-1,1},X^+_{N,0}]_{q^{A_{N,N}}},X^+_{N,0}]                 
			\big)+\\
			    & \hphantom{=}+u_{N,N,1}\big((-1)^{|N|}[[[X^+_{N-1,0},X^+_{N,1}]_{q^{A_{N,N}}},X^+_{N,0}]_{q^{-2A_{N,N}}} +q^{-A_{N,N}}[[X^+_{N-1,0},X^+_{N,0}]_{q^{A_{N,N}}},X^+_{N,1}]+ \\
			    & \hphantom{=}+q^{-A_{N,N}}u_{N,N,1}[[X^+_{N-1,0},X^+_{N,1}]_{q^{A_{N,N}}},X^+_{N,0}]\big)\Big)                                                                           \\
			\if & =-C^{-1}K_N\Big(u_{N-1,N,1}\big( q^{-A_{N,N}}(q^{A_{N,N}}+q^{-A_{N,N}}+(-1)^{|N|})[[X^+_{N-1,0},X^+_{N,1}]_{q^{A_{N,N}}},X^+_{N,0}] \big)+                              \\
			    & +u_{N,N,1}\big((-1)^{|N|}[[[X^+_{N-1,0},X^+_{N,1}]_{q^{A_{N,N}}},X^+_{N,0}]_{q^{-2A_{N,N}}} +q^{-A_{N,N}}[[X^+_{N-1,0},X^+_{N,0}]_{q^{A_{N,N}}},X^+_{N,1}]+             \\
			    & +q^{-A_{N,N}}u_{N,N,1}[[X^+_{N-1,0},X^+_{N,1}]_{q^{A_{N,N}}},X^+_{N,0}]\big)\Big)                                                                       \\
			    & =-C^{-1}K_N\Big(u_{N-1,N,1}\big( q^{-A_{N,N}}(q^{A_{N,N}}+q^{-A_{N,N}}+(-1)^{|N|})[[X^+_{N-1,0},X^+_{N,1}]_{q^{A_{N,N}}},X^+_{N,0}] \big)+                              \\
			    & +u_{N,N,1}\big( (-1)^{|N|} X^+_{N-1,0}X^+_{N,1}X^+_{N,0} -(-1)^{|N|+|N-1||N|}q^{A_{N,N}} X^+_{N,1}X^+_{N-1,0}X^+_{N,0}-                                                 \\
			    & -(-1)^{|N-1||N|}q^{-2A_{N,N}} X^+_{N,0}X^+_{N-1,0}X^+_{N,1} +q^{-A_{N,N}} X^+_{N,0}X^+_{N,1}X^+_{N-1,0}+                                                                \\
			    & +q^{-A_{N,N}} X^+_{N-1,0}X^+_{N,0}X^+_{N,1} -(-1)^{|N-1||N|} X^+_{N,0}X^+_{N-1,0}X^+_{N,1}-                                                                             \\
			    & -(-1)^{|N|+|N-1||N|}q^{-A_{N,N}} X^+_{N,1}X^+_{N-1,0}X^+_{N,0} +(-1)^{|N|} X^+_{N,1}X^+_{N,0}X^+_{N-1,0}+                                                               \\
			    & +q^{-A_{N,N}} X^+_{N-1,0}X^+_{N,1}X^+_{N,0} -(-1)^{|N-1||N|} X^+_{N,1}X^+_{N-1,0}X^+_{N,0}-                                                                             \\
			    & -(-1)^{|N|+|N-1||N|}q^{-A_{N,N}} X^+_{N,0}X^+_{N-1,0}X^+_{N,1}+(-1)^{|N|} X^+_{N,0}X^+_{N,1}X^+_{N-1,0}\big)\Big)                                                       \\
			    & =-C^{-1}K_N\Big(u_{N-1,N,1}\big( q^{-A_{N,N}}(q^{A_{N,N}}+q^{-A_{N,N}}+(-1)^{|N|})[[X^+_{N-1,0},X^+_{N,1}]_{q^{A_{N,N}}},X^+_{N,0}] \big)+                              \\
			    & +u_{N,N,1}\big( (q^{A_{N,N}}+q^{-A_{N,N}}+(-1)^{|N|})\big((-1)^{|N|}q^{-A_{N,N}} X^+_{N-1,0}X^+_{N,1}X^+_{N,0}-                                                         \\
			    & -(-1)^{|N|+|N-1||N|} X^+_{N,1}X^+_{N-1,0}X^+_{N,0}  -(-1)^{|N-1||N|}q^{-A_{N,N}} X^+_{N,0}X^+_{N-1,0}X^+_{N,1}   + X^+_{N,0}X^+_{N,1}X^+_{N-1,0}\big)+                  \\
			    & + (-1)^{|N|} X^+_{N,1}X^+_{N,0}X^+_{N-1,0} + q^{-A_{N,N}} X^+_{N-1,0}X^+_{N,0}X^+_{N,1}-                                                                                \\ 
			    & -(-1)^{|N|}q^{-2A_{N,N}} X^+_{N-1,0}X^+_{N,1}X^+_{N,0} -q^{A_{N,N}} X^+_{N,0}X^+_{N,1}X^+_{N-1,0}\big)\Big)                                                             \\
			    & =-C^{-1}K_N\Big( (q^{A_{N,N}}+q^{-A_{N,N}}+(-1)^{|N|})(u_{N-1,N,1}+ (-1)^{|N|}u_{N,N,1})[[X^+_{N-1,0},X^+_{N,1}]_{q^{A_{N,N}}},X^+_{N,0}]+                              \\
			    & +(-1)^{|N|}u_{N,N,1}[[X^+_{N,1},X^+_{N,0}]_{q^{A_{N,N}}},X^+_{N-1,0}]_{q^{-2A_{N,N}}}                                                                                   
			\Big)\\\fi
			    & =-C^{-1}K_N(q^{A_{N,N}}+q^{-A_{N,N}}+(-1)^{|N|})\Big( q^{-A_{N,N}}(q+q^{-1})[[X^+_{N-1,0},X^+_{N,1}]_{q^{A_{N,N}}},X^+_{N,0}]+                                          \\
			    & \hphantom{=}+[[X^+_{N,1},X^+_{N,0}]_{q^{A_{N,N}}},X^+_{N-1,0}]_{q^{-2A_{N,N}}}                                                                                          
			\Big).
		\end{align*}
				Thus,	
		\begin{align}\label{cubic N N-1, 00}
			(q+q^{-1})[[X^+_{N-1,0},X^+_{N,1}]_{q^{A_{N,N}}},X^+_{N,0}]                               
			+q^{A_{N,N}}[[X^+_{N,1},X^+_{N,0}]_{q^{A_{N,N}}},X^+_{N-1,0}]_{q^{-2A_{N,N}}} \in \cI_\s. 
		\end{align}
		As before, the relations \eqref{cubic N N-1} follow by induction, after applying the automorphisms $\Tilde\tau_{N-1}$ and $\Tilde\tau_N$, and commuting equation \eqref{cubic N N-1, 00} with $H_{N,1}$.
	\end{proof}
	    
\end{lem}

\subsubsection*{Proof of Theorem \ref{thm_main}}
Since $\cI_\s\subset \Ker  \Tilde{\psi}_\s$,  we have an induced surjective homomorphism $\psi_\s:\Tilde{U}_q^D(\ghs)/\cI_\s\rightarrow \Uqghs$, which, together with Lemma \ref{lem_quo}, gives the desired surjective homomorphism ${U}_q^D(\ghs)\rightarrow \Uqghs.$   \qed

\begin{conj}
    The surjective homomorphism $\psi_s: U_q^D(\widehat{\mathfrak{g}}_s) \to U_q(\widehat{\mathfrak{g}}_s)$ from Theorem \ref{thm_main} is an isomorphism.
\end{conj}

The proof of this conjecture and the PBW basis of $U_q^D(\widehat{\mathfrak{g}}_s)$ will appear in \cite{BKL}.

\medskip
\section*{Acknowledgments}
The authors thank A. Tsymbaliuk for a valuable comment. The authors are also grateful to the anonymous referee for valuable comments that helped improve the exposition.

Luan Bezerra was supported by the Guangdong Basic and Applied Basic Research Foundation grant (No. 2024A1515013079). 
Vyacheslav Futorny was supported by NSFC grants (No.12350710178 and No.12350710787). Iryna Kashuba acknowledges financial
support from Guangdong Basic and Applied Basic Research Foundation grant (No.2024A1515013079) and the NSFC grant (No.12350710787).
\medskip

\section*{Declarations}

Data sharing is not applicable to this article as no new data were created or analyzed in this study.

The authors have no conflict of interest to declare that are relevant to the content of this
article.

\bibliography{main}{}
\bibliographystyle{siam}
\end{document}